\documentclass[12pt,reqno]{amsart}

\usepackage{amsmath, amsthm, amssymb}

\usepackage{geometry}
\input xy
\xyoption{all}
\usepackage{mathrsfs}
\usepackage{enumerate}

\usepackage{tikz}
\usepackage[utf8]{inputenc}

\usepackage{hyperref}

\newtheorem{dummy}{}[section]
\newtheorem{theorem}[dummy]{Theorem}
\newtheorem{proposition}[dummy]{Proposition}
\newtheorem{lemma}[dummy]{Lemma}
\newtheorem{corollary}[dummy]{Corollary}

\theoremstyle{definition}
\newtheorem{definition}[dummy]{Definition}

\newtheorem{example}[dummy]{Example}
\theoremstyle{remark}
\newtheorem{remark}[dummy]{Remark}

\newcommand{\sbigoplus}{%
  \mathop{\mathchoice{\textstyle\bigoplus}{\bigoplus}{\bigoplus}{\bigoplus}}%
}
\newcommand{\Amone}{\Lambda_{-1}}

\newcommand{\A}{\Lambda}
\newcommand{\B}{\Delta}
\newcommand{\D}{\Theta} 

\newcommand{\sR}{\mathscr{R}}
\renewcommand{\H}{\mathbf{H}}
\newcommand{\sH}{\H}
\newcommand{\bR}{\mathbf{R}}
\newcommand{\Der}{{\mathfrak{D}}}
\newcommand{\tR}{{\bR '}}
\newcommand{\soR}{\overline{\sR}}

\newcommand{\QQ}{\mathbb Q}
\newcommand{\ZZ}{\mathbb Z}
\newcommand{\CC}{\mathbb C}
\newcommand{\C}{\mathcal C}
\newcommand{\Cont}{\mathrm{Cont}}

\newcommand{\X}{\mathcal X}
\newcommand{\I}{\mathcal I}
\newcommand{\U}{\mathcal U}
\newcommand{\V}{\mathcal V}
\newcommand{\Y}{\mathcal Y}
\newcommand{\Z}{\mathcal Z}

\newcommand{\End}{\mathrm{End}}
\newcommand{\bP}{\mathbb{P}}

\newcommand{\M}{\overline{\mathcal{M}}}
\newcommand{\cO}{\mathcal{O}}

\newcommand{\DD}{\mathcal{D}}

\newcommand{\vir}{\mathrm{vir}}

\newcommand{\Aut}{\mathrm{Aut}}

\newcommand{\ev}{\mathrm{ev}}
\newcommand{\bt}{\mathbf{t}}
\renewcommand{\ev}{\mathrm{ev}}

\newcommand{\diag}{\mathrm{diag}}

\newcommand{\red}{\mathrm{red}}
\newcommand{\eff}{\mathrm{eff}}
\newcommand{\comp}{\mathrm{comp}}
\newcommand{\PP}{\mathbb P}

\DeclareMathOperator{\spann}{span}

\begin{document}

\title[Structure of GW invariants of Quintic $3$-folds]{Structure of Higher Genus Gromov--Witten Invariants of Quintic 3-folds}

\author[S.~Guo]{Shuai Guo}
\address{School of Mathematical Sciences and Beijing International Center for Mathematical Research, Peking University, Beijing, China }
\email{guoshuai@math.pku.edu.cn}

\author[F.~Janda]{Felix Janda}
\address{Department of Mathematics, University of Michigan, 2074 East Hall, 530 Church Street, Ann Arbor, MI 48109, USA}
\email{janda@umich.edu}

\author[Y.~Ruan]{Yongbin Ruan}
\address{Department of Mathematics, University of Michigan, 2074 East Hall, 530 Church Street, Ann Arbor, MI 48109, USA}
\email{ruan@umich.edu}

\begin{abstract}
  There is a set of remarkable physical predictions for the structure
  of BCOV's higher genus B-model of mirror quintic 3-folds which can
  be viewed as conjectures for the Gromov--Witten theory of quintic
  3-folds.
  They are (i) Yamaguchi--Yau's finite generation, (ii) the
  holomorphic anomaly equation, (iii) the orbifold regularity and (iv)
  the conifold gap condition.
  Moreover, these properties are expected to be universal properties
  for all the Calabi--Yau 3-folds.
  This article is devoted to proving first three conjectures.
  
  The main geometric input to our proof is a log GLSM moduli space and
  the comparison formula between its reduced virtual cycle
  (reproducing Gromov--Witten invariants of quintic 3-folds) and its
  nonreduced virtual cycle \cite{CJR18P}.
  Our starting point is a Combinatorial Structural Theorem expressing
  the Gromov--Witten cohomological field theory as an action of a
  generalized $R$-matrix in the sense of Givental.
  An $R$-matrix computation implies a graded finite generation
  property.
  Our graded finite generation implies Yamaguchi--Yau's (nongraded)
  finite generation,  as well as the orbifold regularity.
  By differentiating the Combinatorial Structural Theorem carefully, we
  derive the holomorphic anomaly equations.
  Our technique is purely A-model theoretic and does \emph{not} assume
  any knowledge of B-model.
  Finally, above structural theorems hold for a family of theories
  (the extended quintic family) including the theory of quintic as a
  special case.
\end{abstract}

\setcounter{tocdepth}{2}

\maketitle 
\tableofcontents
 
\section{Introduction}

The computation of the Gromov--Witten (GW) theory of compact
Calabi--Yau 3-folds is a central and yet difficult problem in geometry
and physics.
During last twenty years, it has attracted a lot of attention from
both physicists and mathematicians.
In the early 90s, the physicist Candelas and his collaborators
\cite{CdOGP91} surprised the mathematical community by using mirror
symmetry to derive a conjectural formula for a certain generating
function (the $J$-function) of genus zero Gromov--Witten invariants of
a quintic 3-fold in terms of the period integral (or the $I$-function)
of its B-model mirror.
The effort to prove the formula has directly lead to the birth of
mirror symmetry as a mathematical subject.
In a seminal work \cite{BCOV93} in 1993, Bershadsky, Cecotti, Ooguri
and Vafa (BCOV) introduced the higher genus B-model in physics in an
effort to push mirror symmetry to higher genus.
During the subsequent years, a series of conjectural formulae was
proposed by physicists based on the BCOV B-model.
Let $F_g(Q)$ be the generating function of genus $g$ GW-invariants of
a quintic 3-fold.
BCOV had already proposed a conjectural formula for $F_1$ and $F_2$ in
their original paper (see also \cite{YaYa04}).
A conjectural formula for $F_3$ was proposed by Katz--Klemm--Vafa in
1999 \cite{KKV99}.
Afterwards, it became clear that we need a better understanding of the
structure of $F_g$.
A fundamental physical prediction of Yamaguchi--Yau is that $F_g$ can
be written as a polynomial of five generators constructed explicitly
from the period integral of its mirror.
Among the five generators, four of them are the holomorphic limit of
certain non-holomorphic objects in the B-model.
The famous holomorphic anomaly equation of BCOV can be recasted into
equations determining the dependence on four of the generators.
Abusing terminology, we still refer to them as holomorphic anomaly
equations even though the four generators are holomorphic.
The initial condition of the holomorphic anomaly equations is expected
to be a degree $3g-3$ polynomial $f_g$ of a single variable $Z$.
Two other physical predictions regarding the structure of $f_g$ are
the \emph{conifold gap condition} and the \emph{orbifold regularity}
which determine the lower and upper parts of $f_g$.
Furthermore, the above structures of Gromov--Witten invariants are
expected to be present in some fashion for all Calabi--Yau 3-folds,
and hence can be considered as universal properties.
Based on the above B-model structural predictions and an additional
A-model conjecture called Castelnuovo bound, Klemm and his
collaborators \cite{HKQ09} derived a formula for $F_g$ for all
$g\leq 51$!

The progress for mathematicians to prove these conjectures has been
slow.
The genus zero conjecture was proved by Givental \cite{Gi98b} and
Lian--Liu--Yau \cite{LLY97}, and it was considered to be a major event
in mathematics during the 90s.
It took another ten years for Zinger to prove BCOV's conjecture in
genus one \cite{Zi08}.
It took yet another ten years for the authors' recent proof of the
genus two BCOV conjecture \cite{GJR17P}.
The geometric input to our work on genus two is a construction of
a certain \emph{reduced virtual cycle} on an appropriate log
compactification of the GLSM moduli space \cite{CJR18P} (see also
\cite{CJRS18P}).
Its localization formula expresses the Gromov--Witten invariants of
quintic 3-folds in terms of a graph sum of (rather mysterious)
\emph{effective invariants} and (rather well-understood) twisted
GW-invariants of $\PP^4$.
This formula works in arbitrary genus as long as one can compute the
effective invariants.
In particular, there is only one effective invariant for $g=2,3$, and
it can be computed from the $g=2,3$, degree zero GW-invariant.
So in principle, we can push our technique to $g=3$ to prove the
conjectural formula of Katz--Klemm--Vafa.
The main difficulty is directly related to the physical prediction of
Yamaguchi--Yau that $F_g$ should be a polynomial of five generators.
A direct generalization of our argument in \cite{GJR17P} only implies
that $F_g$ is a polynomial of nine generators.
The appearance of extra generators is similar to the simpler case of
the Gromov--Witten theory of an elliptic curve, where $F_g$ is a
quasi-modular form of $SL_2(\ZZ)$.
On the other hand, the corresponding twisted theory of $\cO(3)$ on
$\PP^2$ is a quasi-modular form for $\Gamma_0(3)$.
The ring of quasi-modular forms of $\Gamma_0(3)$ has more generators
than that of $SL_2(\ZZ)$!
In the case of the quintic, the presence of the extra four generators
increases the computational complexity significantly.
We could try to prove the genus three formula by brute force but the
proof would not be illuminating and is unlikely to generalize to
higher genus.
Our eventual goal is to reach to $g=51$ and beyond.
It is clear to us that we should first attack the set of physical
predictions for the structures of $F_g$.

To describe the main idea, recall that the general twisted theory
naturally depends on six equivariant parameters, five for the base
$\bP^4$ and one for the twist.
It is complicated to study the general twisted theory, and therefore
Zagier--Zinger \cite{ZaZi08} specialize the equivariant parameter to
$(\lambda, \zeta\lambda, \zeta^2\lambda, \zeta^3\lambda,
\zeta^4\lambda, 0)$, where $\zeta$ is a primitive fifth root of unity.
This theory is referred as \emph{formal quintic} (see \cite{LhPa18}).
They show that the twisted theory is generated by the five generators
predicted by physicists.
Unfortunately, in our work, the natural specialization of equivariant
parameters is $(0, 0,0,0,0, t)$.
The bulk of \cite{GJR17P} is to show that the corresponding twisted
theory for equivariant parameter $(0, 0,0,0,0, t)$ has four
\emph{extra generators}.
The main input of the current article is a comparison formula
expressing the reduced virtual cycle of the Log GLSM moduli space in
terms of its canonical (non-reduced) virtual cycle.
The advantage of non-reduced theory is that it admits a $(\CC^*)^6$
action with six equivariant parameters which we can specialize to
$(\lambda, \zeta\lambda, \zeta^2\lambda, \zeta^3\lambda,
\zeta^4\lambda, 0)$ as for the formal quintic.

\medskip

To state our precise results, we need to setup some notation.
Let $X_5$ be a quintic 3-fold.
Fix $(g, n)$ such that $2g - 2 + n > 0$, and fix ambient classes
$\gamma_1, \dotsc, \gamma_n \in H^*(X_5)$.\footnote{By a dimension
  consideration, insertions of primitive classes are not very
  interesting.}
Let
\begin{equation*}
  \Omega_{g, n} (\gamma_1, \dotsc, \gamma_n)
  := \sum_{\beta = 0}^\infty Q^{\beta} \rho_* \left(\prod_{i = 1}^n \ev_i^*(\gamma_i) \cap [\M_{g, n}(X_5, \beta)]^\vir\right),
\end{equation*}
where $\rho\colon \M_{g, n}(X_5, \beta) \to \M_{g, n}$ is the forgetful
map, be the generating series of Gromov--Witten classes defined by
$X_5$.
For $g \ge 2$, let
\begin{equation*}
  F_g(Q) := \int_{\M_g} \Omega_{g, 0}
\end{equation*}
be the corresponding numerical generating series.
In the cases $g = 0$, and $g = 1$ to avoid unstable terms, we need
markings (see below).

Let $\mathcal M_{g, n}(\PP^4, \cO(5), d)$ be the GLSM moduli space for
a quintic 3-fold, that is the moduli space of stable maps to $\PP^4$
with a $p$-field \cite{ChLi12}.
In \cite{CJR18P}, we construct a certain logarithmic compactification
$\M_{g, n}(\PP^4, \cO(5), d, \nu)$ (see Section~\ref{sec:loc} for more
details) where $\nu$ is a partition representing the contact order (or
relative condition).
Traditionally, we call the case $\nu=\emptyset$ the \emph{holomorphic
  theory}, and the case $\nu\neq \emptyset$ the \emph{meromorphic
  theory}.
The moduli space $\M_{g, n}(\PP^4, \cO(5), d, \nu)$ is a proper
Deligne--Mumford stack with a two-term perfect obstruction theory.
Hence, it admits a virtual fundamental cycle
$[\M_{g, n}(\PP^4, \cO(5), d, \nu)]^\vir$.
The cycle $[\M_{g, n}(\PP^4, \cO(5), d, \nu)]^\vir$ is equivariant for
both the $(\CC^*)^5$ action of $\PP^4$ and the $\CC^*$ action on the
$p$-field.
The holomorphic theory is very special in the sense, that in
\cite{CJR18P} we construct a \emph{reduced} virtual cycle
$[\M_{g, n}(\PP^4, \cO(5), d, \nu)]^\red$.
The main result of \cite{CJR18P} is that
$[\M_{g, n}(\PP^4, \cO(5), d, \nu)]^\red$ computes the GW-invariants
of quintic 3-folds.
The boundary of $\M_{g, n}(\PP^4, \cO(5), d, \nu)$ contains
\emph{rubber moduli spaces}
$\M_{g, n}^\sim(\PP^4, \cO(5), d, \mu, \nu)$ with their own virtual
fundamental cycles.
The key new higher genus information for quintic 3-folds are the
\emph{effective invariants}
\begin{equation*}
  c^\eff_{g, d}
  = \deg [\M_{g, n}^\sim(\PP^4, \cO(5), d, (2^{2g - 2 - 5d}), \emptyset)]^\red \in \QQ
\end{equation*}
for integral $g, d\geq 0$ such that $5d \le 2g - 2$.
In particular, the invariants $c^\eff_{g, d}$ are only defined when
$g \ge 1$, and when $d \le \frac{2g - 2}5$.
Furthermore, the $c^\eff_{g, d}$ are determined by the corresponding
low degree Gromov--Witten invariants.

The comparison formula of \cite{CJR18P} between the reduced and
non-reduced cycle yields a formula of the form (the precise result is
in Theorem~\ref{thm:compare})
\begin{equation*}
  \Omega_{g,n} 
  = \sum_{ \Gamma} c^\eff_{g_1,d_1} \cdots c^\eff_{g_k, d_k} \Omega_\Gamma,
\end{equation*}
where $\Gamma$ is a decorated bipartite graph, the $c^\eff_{g_i,d_i}$
are effective invariants associated to $\infty$ vertices and
$\Omega_{\Gamma}$ is the remaining contributions which can be
expressed in terms of non-reduced virtual cycles.
We can replace $c^\eff_{g, d}$ in the above formula by formal
parameters $c_{g, d}$ and denote the resulting generating series by
$\Omega^{\mathbf c}_{g,n}$.
We call $\Omega^{\mathbf c}_{g,n}$ \emph{the extended quintic family}.
We can then obtain the theory of $X_5$ by setting
$c_{g, d} = c^\eff_{g, d}$.
By setting $c_{g, d} = 0$, we obtain the theory of holomorphic GLSM.

\subsection{Graded finite generation and orbifold regularity}

Let us consider the $I$-function (or period integral of its mirror) of
the quintic 3-fold
\begin{equation*}
  I(q,z) = z\sum_{d\geq 0} q^d \frac{\prod^{5d}_{k=1}(5H+kz)}{\prod^d_{k=1}(H+kz)^5}
\end{equation*}
where $H$ is a formal variable satisfying $H^4=0$.
We separate $I(q, z)$ into components:
\begin{equation*}
  I(q,z)  = z I_0(q) \mathbf 1 + I_1(q) H  +z^{-1} I_2(q)H^2 + z^{-2} I_3(q) H^3
\end{equation*}

The genus zero mirror symmetry conjecture of quintic 3-folds can be
phrased as a relationship
\begin{equation*}
  J(Q)=\frac{I(q)}{I_0(q)}
\end{equation*}
between the $J$- and $I$-function, involving the mirror map
$Q = q e^{\tau_Q(q)}$, where
\begin{equation*}
  \tau_Q(q)  = \frac{I_1(q)}{I_0(q)}.
\end{equation*}
Now we introduce the following  degree $k$ ``basic'' generators
\begin{equation*}
  \X_k := \frac{d^k}{du^k} \left(\log \frac{I_0} L\right),\quad
  \Y_k:= \frac{d^k}{du^k} \left(\log \frac{I_0  I_{1,1}}{ L^2}\right),\quad
  \Z_k:=  \frac{d^k}{du^k} \left( \log (q^{\frac{1}{5}}L)\right),
\end{equation*}
where $I_{1,1} := 1+ q\frac{d}{dq} \tau_Q$,
$L := (1-5^5q)^{-\frac{1}{5}}$ and
$du := L \frac{dq}{q} = \frac{L}{I_{1, 1}} \frac{dQ}Q$.
We often abbreviate $\X = \X_1$, $\Y = \Y_1$ and $\Z = \Z_1$.

To simplify later formulae, we introduce following basis:
$$
 \phi_k = {I_{1,1}\cdots I_{k,k}}{  L^{-k}} H^k
$$
and the normalized Gromov--Witten classes
\begin{equation} \label{normalizedcohft}
  \bar \Omega^{\mathbf c}_{g,n}
  :=  {5^{g-1}(L/I_0)^{2g-2}}  \Omega^{\mathbf c}_{g,n}
\end{equation}
where $I_{2,2}=L^5 I_{0}^{-2} I_{1,1}^{-2}$, $I_{3,3}=I_{1,1}$,
$I_{4,4}=I_0$ \cite{ZaZi08}.

One can apply localization formula to the non-reduced virtual cycles
in the comparison formula to compute $\bar \Omega^{\mathbf c}_{g,n}$,
and specialize the equivariant parameter to the formal quintic
parameter
$(\lambda, \zeta \lambda, \zeta^2 \lambda, \zeta^3\lambda,
\zeta^4\lambda, 0)$.
The first main result of the article is a combinatorial structural
theorem that packages the localization contributions into a
Givental-style $R$-matrix action.

\begin{theorem}(Combinatorial Structural Theorem)
  \begin{equation*}
    \bar \Omega^{\mathbf c }_{g,n}
    = \lim_{\lambda \to 0} \sum_{\Gamma\in G_{g,n}^\infty} \frac{1}{|\Aut(\Gamma)|}  \Cont_\Gamma.
  \end{equation*}
  Here, $G_{g,n}^{\infty}$ is the set of genus $g$, $n$-marked stable
  graphs with the decorations:
  \begin{itemize}
  \item for each vertex $v$, we assign a label $0$ or $\infty$;
  \item for each flag $f=(v,e)$ or $f=(v,l)$ where $v$ labeled by $\infty$, we assign a degree $\delta_f \in \ZZ_{\geq 1}$.
  \end{itemize}
  For each $\Gamma \in G^\infty _{g,n}$, the contribution
  $\Cont_\Gamma$ is defined from the $R$- and $S$-matrices of the
  formal quintic theory in an explicit formula similar to that of
  Givental's $R$-matrix action (we refer to Section~\ref{NewRaction}
  for the precise formula).
\end{theorem}

We cannot directly apply the structural results of formal quintic
proven in \cite{LhPa18P}.
Nevertheless, after much computation, we can prove Yamaguchi--Yau's
prediction.

\begin{definition}
We introduce the ring of $6$-generators
$$
\widetilde \bR:={\QQ}[L^{-1},\X_1,\X_2, \X_3, \Y, Z]
$$
the degree defined as follows
$$
\deg L^{-1} = 1,\quad \deg \X_k =k,\quad \deg \Y =1,\quad \deg Z=0.
$$
We define Yamaguchi--Yau's finite generation ring\footnote{From the
  definitions one can check that it is a subring of $\widetilde \bR$.}
by
\begin{equation} \label{fivegenring}
  \bR := \QQ[\X_1, \X_2, \X_3, \Y] \otimes \spann_\QQ \{L^{-a} Z^b : b \le a \le 5b\}
\end{equation}
We denote by $\bR_k \subset \bR$ the linear subspace of degree $k$
elements.

This ring admits an additional structure: there exists a derivation
$\partial_u$ acting on this ring, such that the ring is closed under
$\partial_u$ and that $\partial_u$ increases the degree by $1$.
The explicit definition of $\partial_u$ is given in
Lemma~\ref{lem:du}.
\end{definition}
\begin{remark}
  By setting $Z=L^5$, we can regard
  $\bR\subset \QQ[L^{-1},\X_1,\X_2, \X_3, \Y, Z]$ as a subring of
  $\QQ[\X_1,\X_2, \X_3, \Y, L]$.
  Then we will lose the degree information.
\end{remark}

\begin{theorem}
  \label{thm:fgen}
The following ``graded finite generation properties" hold for the extended quintic family
\begin{description}
\item[(1)] $\bar \Omega^{\mathbf c}_{g,n}(\phi_{a_1}, \dotsc, \phi_{a_n}) \in  H^{2(3g-3+\sum_i a_i)}(\M_{g,n}, \QQ)\otimes \bR_{3g-3+\sum_i a_i}$\\
\item[(2)] $\bar{F}_{g,n} = \int_{\M_{g, n}} \bar \Omega_{g, n}^{\mathbf c}(\phi_1^{\otimes n}) \in \bR_{3g-3+n}$
\end{description}
\end{theorem}

Our graded finite generation theorem is much stronger than
Yamaguchi--Yau's original non-graded finite generation.
For example, a direct consequence is another key structural
prediction:
\begin{corollary}(Orbifold regularity)
  Suppose that $f_g=L^{3g-3} \bar F_g|_{\X=\Y_k=0}$.
  Then we can write
  \begin{equation*}
    f_g=  \sum^{3g-3}_{i = 0} a_{i,g} Z^i \quad  \text{ with } \quad Z=L^5 ,
  \end{equation*}
  where  $a_{i, g} =0$ for $i\leq \lceil \frac{3g-3}{5} \rceil$.
\end{corollary}

\begin{remark}
  $f_g$ represents the initial condition of the holomorphic anomaly
  equations.
  By using the boundary behavior of the large complex structure limit
  point, the conifold singularity and the orbifold regularity, it is
  expected to be a polynomial of $Z$ of degree $3g-3$. 
  At this moment, it is beyond our ability to calculate $f_g$ directly
  even at $g=3$.
  There are two additional B-model structural predictions for
  $f_g(Z)$.
  The \emph{orbifold regularity} claims the vanishing of the lower
  part of $f_g$.
  The other is the \emph{conifold gap condition}, which determines the
  upper $2g-2$ coefficients of of $f_g$.
  As we see, for each $g$ there are $3g-2$ of initial conditions.
  By using the orbifold regularity and the degree zero Gromov-Witten
  invariants, the number of initial conditions is reduced by
  $\lceil \frac{3g-3}{5} \rceil+1$.
  By using the conifold gap condition, it is further reduced to
  $\lfloor\frac{2g-2}{5}\rfloor$ many, the same as the number of
  effective invariants.
  It is natural to speculate that our approach also implies the
  conifold gap condition.
  We leave this for future research.
\end{remark}

\begin{remark}
  There is a different approach to the higher genus theory of quintic
  3-folds by Chang--Guo--Li--Li--Liu.
  Recently, they posted a series of articles \cite{CGLL18P, CGL18Pa,
    CGL18Pb}.
  Among other things, they proved the original Yamaguchi--Yau
  (non-graded) prediction independently.
\end{remark}

\begin{remark}
  It is nontrivial to show \cite{Mo11} that the five generators are
  algebraic independent and hence the expression of $\bar{F}_g$ is
  unique.
  This is important for the statement of the holomorphic anomaly
  equation for which we need to consider derivatives with respect to
  generators.
  On the other hand, our proof of Theorem~\ref{thm:fgen} gives a
  canonical expression of $\bar \Omega^{\mathbf c}_{g,n}$ and
  $\bar{F}_g$ in terms of the generators.
  Hence, we do not need the algebraic independence of five generators.
\end{remark}
    
\subsection{Holomorphic anomaly equations}

We now consider the holomorphic anomaly equations (HAE).
It is clear that the choice \eqref{fivegenring} of the generators of
$\bR$ is not canonical.
Let us make a choice.
\begin{definition}
  Suppose $S$ be a finite set that generates $\bR$.
  We pick a subset $S'$ of $S$ such that
  \begin{equation*}
    S'\cap \QQ[L]=\emptyset.
  \end{equation*}
  Let $\tR = \QQ[S'] \subset \bR$ be the subring generated by
  $S'$.
  We call $\tR$ a choice of \emph{non-holomorphic subring}, and we call the
  elements in the ring of the derivations of $\tR$
  \begin{equation*}
    \Der_{\tR} \subset \Der_{\bR}
  \end{equation*}
  the \emph{vector fields in the non-holomorphic direction}.
\end{definition}

\begin{remark}
  A natural choice of the non-holomorphic subring is
  \begin{equation} \label{tR}
    \tR:= \QQ[\X_1,\X_2,\X_3,\Y].
  \end{equation}
\end{remark}

Motivated by \cite{YaYa04}, we can pick another set of generators
\begin{equation}\label{generatorsSp}
S':=\{ \U := \X,\quad \V := \X+\Y, \quad \V_2:= \partial_u \U +\U^2 -\U\V,\quad \V_3:=(\partial_u+\V)\, \V_2 \}
\end{equation}
of $\tR$ defined in \eqref{tR},  such that $\QQ[S'] = \tR$.  Here $\partial_u:=\frac{d}{du}$.
\begin{theorem}
  \label{HAE2}
  We introduce the derivations $\partial_1, \partial_0 \in \Der_{\tR}$
  to be
  \begin{equation} \label{operatorpartial01}
    \partial_1 = \partial_\U, \qquad
    \partial_0 = \partial_{\V_1}-\U \, \partial_{\V_2} -\U^2 \,  \partial_{\V_3} .
  \end{equation}
  Then we have the following holomorphic anomaly equations conjectured
  in \cite{YaYa04}
  \begin{align*}
    -\partial_0 \bar F_g &= \frac{1}{2} \bar F_{g-1,2}  +\frac{1}{2} \sum_{g_1+g_2=g} \bar F_{g_1,1} \bar F_{g_2,1},  \\
    -\partial_1 \bar F_g &= 0.
  \end{align*}
  Note that we could consider $\bar F_{g,n}$ as polynomials of either
  the generators $\{\X_1,\X_2,\X_3,\Y,L\}$ or the generators
  $\{\V_1,\V_2,\V_3,\U,L\}$.
\end{theorem}

\begin{remark}
  The above holomorphic anomaly equation should be viewed as a higher
  genus generalization of the Picard--Fuchs equation.
  For example, for $g \ge 2$, it determines (see
  Remark~\ref{rmk:YYHAE}) $\bar F_g$ inductively up to a holomorphic
  function in $L$, which is referred to as the \emph{holomorphic
    anomaly} in the physics literature.
  Our Graded Finite Generation Theorem implies that the holomorphic
  anomaly is a polynomial of degree $3g-3$, whose lower part vanishes
  (orbifold regularity).
\end{remark}

\subsection{Plan of the paper}

The paper is organized as follows.
In Section~\ref{sec:loc}, we introduce the logarithmic
compactification of the GLSM moduli space and the comparison formula
between its reduced and nonreduced virtual cycle.
In Section~\ref{sec:rewrite}, we prove a geometric comparison between
the formal quintic and the true quintic theory.
The combinatorial structural theorem is proved in
Section~\ref{sec:Rmatrixaction}.
Graded finite generation and orbifold regularity are proven in
Section~\ref{sec:fgen}.
The holomorphic anomaly equation is proven in Section~\ref{sec:HAE}.
In the final section~\ref{sec:orbifoldregularity}, we prove a
technical result of the formal quintic theory which we used in the
proof of graded finite generation.
It implies a conjecture of Zagier--Zinger \cite{ZaZi08}.

\subsection{Acknowledgments}

Our program (including the results of the current paper) to the higher
genus theory was first announced in a workshop in Zurich in January
2018.
A special thanks to R.~Pandharipande for the opportunity to present
our program and his constant support.
We thank D.~Maulik, A.~Pixton, G.~Oberdieck and X.~Wang for
interesting discussions.
The third author would like to thank S.~Donaldson, H.~Ogouri and
C.~Vafa for valuable discussions regarding the program.

The first author was partially supported by supported by NSFC grants
11431001 and 11501013.
The second author was partially supported by an AMS/Simons Travel
Grant.
The third author was partially supported by NSF grant DMS 1405245 and
NSF FRG grant DMS 1159265.
Part of the work was done during last two authors' visit to the MSRI,
which was supported by NSF grant DMS-1440140.

\section{Logarithmic GLSM Moduli Space and its Virtual Cycles}
\label{sec:loc}

In this Section, we collect some results that will appear in
\cite{CJR18P, CJR19P}.

\subsection{Moduli space}

Let $X$ be a smooth projective variety with a line bundle $L$
admitting a section cutting out a smooth hypersurface $Y$.
We are mainly interested in the case that $X = \PP^4$ and
$L = \cO(5)$.

Given a map $f\colon C \to X$ from a log-smooth curve $C$, let $\PP$
be the projective bundle
$\PP(\omega_{\log} \otimes f^* L^\vee \oplus \cO)$ equipped with the
logarithmic structure pulled back from $C$ and the divisorial log
structure from the infinity section.
Let $\M_{g, n}(X, L, \beta, \nu)$ be the moduli space of maps
$f\colon C \to X$ of degree $\beta$ together with an
aligned\footnote{This imposes that the partial order on the
  degeneracies of the irreducible components is actually a total
  order. This is a variant of the logarithmic moduli space that
  behaves well for the localization.} log-section
$\eta\colon C \to \PP$ with contact order $\nu$ along the infinity
section, and with all markings mapping to the zero section, and such
that $\omega_C^{\log} \otimes f^* \cO(1) \otimes \eta^* \cO(0_\PP)$,
where $0_\PP$ is the zero section of $\PP$, is ample for all
sufficiently large $k$.
For the computation of the Gromov--Witten theory of $Y$ it suffices to
consider the case when $\nu$ is the empty partition but in the proof
of the Combinatorial Structural Theorem, we will also need to consider
more general $\nu$.

\begin{theorem}[\cite{CJR18P}]
  The space $\M_{g,n}(X, L, \beta, \nu)$ is a proper Deligne--Mumford
  stack.
\end{theorem}

There are evaluation maps
$\ev_i\colon \M_{g,n}(X, L, \beta, \nu) \to X$, and there is a
forgetful map
$p\colon \M_{g,n}(X, L, \beta, \nu) \to \M_{g,n + |\nu|}(X, \beta)$.

\subsection{Virtual cycles}

The moduli space $\M_{g,n}(X, L, \beta, \nu)$ has a canonical perfect
obstruction theory defining a virtual cycle
$[\M_{g,n}(X, L, \beta, \nu)]^\vir$.

The maximal degeneracy defines a line bundle $L_{\max}$ on
$\M_{g,n}(X, L, \beta, \nu)$.
In the case that $\nu = \emptyset$, in \cite{CJR18P}, a surjective
homomorphism of sheaves
\begin{equation*}
  \sigma\colon \mathrm{Ob}_{\M_{g,n}(X, L, \beta, \emptyset)} \to L_{\max}^\vee
\end{equation*}
is constructed using the section cutting out the hypersurface $Y$.
Using this cosection-like homomorphism, an alternative \emph{reduced}
virtual cycle $[\M_{g,n}(X, L, \beta, \emptyset)]^\red$ is
constructed, which agrees with
$[\M_{g,n}(X, L, \beta, \emptyset)]^\vir$ on the locus where the
log-section $\eta$ does not map any components of the source to the
infinity section.
In particular, the reduced virtual dimension is the same as the
ordinary virtual dimension.
The crucial property of the reduced virtual cycle is the following:
\begin{theorem}[\cite{CJR18P}]
  \label{thm:period}
  We have
  \begin{equation*}
    p_* [\M_{g, n}(Y, \beta)]^\vir = (-1)^{1 - g + \int_\beta c_1(L)} p_* [\M_{g,n}(X, L, \beta, \emptyset)]^\red,
  \end{equation*}
  where on both sides $p$ denotes the corresponding forgetful map to
  $\M_{g, n}(X, \beta)$.
\end{theorem}

\subsection{Localization formula}
\label{sec:loc:formula}

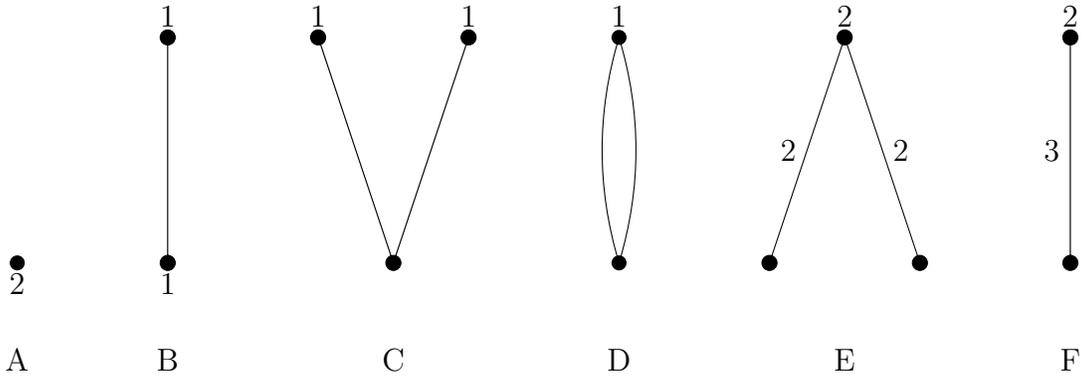
\begin{figure}
  \centering
  \begin{tikzpicture}
    \fill (0, 0) circle(1mm) node[below] {$2$};
    \draw (0, -1) node[below] {A};
    \draw[fill] (2, 0) circle(1mm) node[below] {$1$} -- (2, 3) circle(1mm) node[above] {$1$};
    \draw (2, -1) node[below] {B};
    \draw[fill] (4, 3) node[above] {$1$} circle(1mm) -- (5, 0) circle(1mm) -- (6, 3) circle(1mm) node[above] {$1$};
    \draw (5, -1) node[below] {C};
    \draw (8, 0) .. controls (7.7, 1) and (7.7, 2) .. (8, 3);
    \draw (8, 0) .. controls (8.3, 1) and (8.3, 2) .. (8, 3);
    \fill (8, 0) circle(1mm); \fill (8, 3) circle(1mm) node[above] {$1$};
    \draw (8, -1) node[below] {D};
    \draw[fill] (10, 0) circle(1mm) -- (10.5,1.5) node[left] {$2$} -- (11, 3) circle(1mm) node[above] {$2$};
    \draw[fill] (12, 0) circle(1mm) -- (11.5,1.5) node[right] {$2$} -- (11, 3);
    \draw (11, -1) node[below] {E};
    \draw[fill] (14, 0) circle(1mm) -- (14, 1.5) node[left] {$3$} -- (14, 3) circle(1mm) node[above] {$2$};
    \draw (14, -1) node[below] {F};
  \end{tikzpicture}
  \caption{The bipartite genus two graphs (for $\nu = \emptyset$). Each vertex is decorated by its genus when it is different from zero. Each edge is decorated by its degree when it is different from one. The curve class $\beta$ is distributed among the stable vertices $v \in V_0(\Gamma)$ (of which there is at most one).}
  \label{fig:g2graph}
\end{figure}
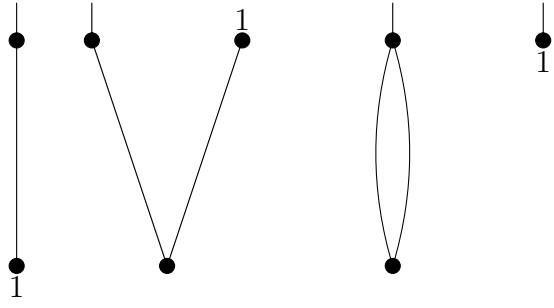
\begin{figure}
  \centering
  \begin{tikzpicture}
    \draw[fill] (0, 0) circle(1mm) node[below] {$1$} -- (0, 3) circle(1mm) -- (0, 3.5);
    \draw[fill] (2, 0) circle(1mm) -- (3, 3) circle(1mm) node[above] {$1$};
    \draw[fill] (1, 3.5) -- (1, 3) circle(1mm) -- (2, 0);
    \draw (5, 0) .. controls (4.7, 1) and (4.7, 2) .. (5, 3);
    \draw (5, 0) .. controls (5.3, 1) and (5.3, 2) .. (5, 3);
    \draw[fill] (5, 0) circle(1mm) (5, 3) circle(1mm) -- (5, 3.5);
    \draw[fill] (7, 3) circle(1mm) node[below] {$1$} -- (7, 3.5);
  \end{tikzpicture}
  \caption{The bipartite genus one graphs for the partition $\nu = (1)$}
  \label{fig:g1graph}
\end{figure}

The moduli space $\M_{g,n}(X, L, \beta, \nu)$ admits a $\CC^*$-action
defined via scaling the log-section $\eta$.
Both the ordinary and reduced perfect obstruction theories are
equivariant with respect to this torus action.
Hence, we can apply the virtual localization theorem \cite{GrPa99} to
compute their virtual classes.

We first consider the fixed loci of $\M_{g,n}(X, L, \beta, \nu)$.
They are classified by a set of decorated bipartite graphs $\Gamma$,
which we describe now.
Let $V(\Gamma)$ and $E(\Gamma)$ be the sets of vertices and edges of
$\Gamma$, respectively.
There are various decorations:
\begin{itemize}
\item the bipartite structure $V(\Gamma) = V_0(\Gamma) \sqcup V_\infty(\Gamma)$
\item a genus mapping $g\colon V(\Gamma) \to \ZZ_{\ge 0}$
\item a curve class mapping $\beta\colon V(\Gamma) \to \mathrm{Pic}(X)^\vee$
\item a fiber class mapping $\delta\colon E(\Gamma) \to \ZZ_{\ge 1}$
\item a distribution of markings
  $m\colon \{1, \dotsc, n\} \sqcup \nu \to V(\Gamma)$
\end{itemize}
The valence $n(v)$ of a vertex $v \in V(\Gamma)$ is the sum of the
number of edges at $v$ and the number of preimages under $m$.
We set for all $v \in V_\infty(\Gamma)$ that $\nu(v) = |m^{-1}(\nu)|$
and that $\mu(v)$ is the partition formed by $\delta(e)$ for all edges
$e$ at $v$.

We require the decorated graphs to statisfy the following conditions:
\begin{itemize}
\item $\sum_v g(v) + h^1(\Gamma) = g$
\item $\sum_v \beta(v) = \beta$
\item $\forall v \in V_\infty(\Gamma): |\mu(v)| - |\nu(v)| = 2g(v) - 2 + n(v) - \int_{\beta(v)} c_1(X)$
\item $m(\{1, \dotsc, n\}) \subseteq V_0(\Gamma)$, $m(\nu) \subseteq V_\infty(\Gamma)$
\item If for $v \in V(\Gamma)$, we have
  $(g(v), n(v), \beta(v)) = (0, 1, 0)$, then the unique incident edge
  $e$ has $\delta(e) > 1$.
\end{itemize}

In the case that $X = \PP^4$, $L = \cO(5)$, Figure~\ref{fig:g2graph}
and Figure~\ref{fig:g1graph} show all localization graphs for
$\M_2(X, L, \beta, \emptyset)$ and $\M_1(X, L, \beta, (1))$ which do
not have any vertices $v \in V_0(\Gamma)$ with
$(g(v), n(v)) = (0, 1)$.
These are all localization graphs for the analogous moduli space
defined using stable quotients.
The full set of stable maps localization graphs is obtained by
attaching additional genus zero vertices via degree-one edges to any
$v \in V_\infty(\Gamma)$ and distributing $\beta$ among them and the
original stable vertices in $V_0(\Gamma)$.

Given a decorated graph $\Gamma$, consider the fiber product
\begin{equation*}
  \xymatrix{
    \widetilde M_\Gamma \ar[r] \ar[d] & M_\Gamma = \prod_{v \in V(\Gamma)} \M_{g(v), n(v)}(X, \beta(v)) \ar[d] \\
    X^{|E(\Gamma)|} \ar[r]^\Delta & X^{|E(\Gamma)|} \times X^{|E(\Gamma)|}
    }
\end{equation*}
where for unstable $v \in V_0(\Gamma)$ in the sense that
$2g(v) - 2 + n(v) + 3 \int_{\beta(v)} c_1(X) \le 0$, we define the
corresponding factor in $M_\Gamma$ to be $X$.
There is a gluing map
\begin{equation*}
  \iota_\Gamma\colon \widetilde M_\Gamma \to \M_{g,n + |\nu|}(X, \beta).
\end{equation*}

We then have
\begin{equation}
  \label{eq:locvir}
  p_*([\M_{g,n}(X, L, \beta, \nu)]^\vir)
  = \sum_\Gamma \frac 1{|\Aut(\Gamma)|} \iota_{\Gamma*} \Delta^! (C_\Gamma^\vir)
\end{equation}
where
\begin{multline*}
  C_\Gamma^\vir = \prod_{v \in V_0(\Gamma)} \frac{e(-R\pi_{v*}(\omega_{\pi_v} \otimes f^* L^\vee) \otimes [1])}{\prod_{e\text{ at }v} (\frac{t - \ev_e^*(c_1(L))}{\delta(e)} - \psi_e)} \cap [\M_{g(v), n(v)}(X, \beta(v))]^\vir \\
  \prod_{v \in V_\infty(\Gamma)} p_*\left(\frac 1{-t - c_1(L_{\min})}
    [\M_{g(v)}^\sim(X, L, \beta(v), \mu(v), \nu(v))]^\vir\right) \prod_{e \in E(\Gamma)} \frac
  1{\delta(e) \prod_{i = 1}^{\delta(e)} \frac{t - \ev_e^*(c_1(L))}{\delta(e)}}
\end{multline*}
and
\begin{equation}
  \label{eq:locred}
  p_*([\M_{g,n}(X, L, \beta, \emptyset)]^\red)
  = \sum_\Gamma \frac 1{|\Aut(\Gamma)|} \iota_{\Gamma*} \Delta^! (C_\Gamma^\red)
\end{equation}
where
\begin{multline*}
  C_\Gamma^\red = \prod_{v \in V_0(\Gamma)} \frac{e(-R\pi_{v*}(\omega_{\pi_v} \otimes f^* L^\vee) \otimes [1])}{\prod_{e\text{ at }v} (\frac{t - \ev_e^*(c_1(L))}{\delta(e)} - \psi_e)} \cap [\M_{g(v), n(v)}(X, \beta(v))]^\vir \\
  \prod_{v \in V_\infty(\Gamma)} p_*\left(\frac t{-t - c_1(L_{\min})}
    [\M_{g(v)}^\sim(X, L, \beta(v), \mu(v), \emptyset)]^\red\right)
  \prod_{e \in E(\Gamma)} \frac 1{\delta(e) \prod_{i = 1}^{\delta(e)}
    \frac{i(t - \ev_e^*(c_1(L)))}{\delta(e)}}.
\end{multline*}
We have used the notations:
\begin{itemize}
\item $\M_{g(v)}^\sim(X, L, \beta(v), \mu(v), \nu(v))$ is a rubber
  moduli space defined in \cite{CJR18P}.
\item $L_{\min}$ is the line bundle on
$\M_{g(v)}^\sim(X, L, \beta(v), \mu(v), \nu(v))$ corresponding to the
minimal degeneracy.
\item $[\M_{g(v)}^\sim(X, L, \beta(v), \mu(v), \nu(v))]^\vir$ is a
  virtual class on $\M_{g(v)}^\sim(X, L, \beta(v), \mu(v), \nu(v))$.
\item $[\M_{g(v)}^\sim(X, L, \beta(v), \mu(v), \emptyset)]^\red$ is a
  virtual class on $\M_{g(v)}^\sim(X, L, \beta(v), \mu(v), \emptyset)$
  satisfying
  \begin{equation*}
    c_1(L_{\max}) [\M_{g(v)}^\sim(X, L, \beta(v), \mu(v), \emptyset)]^\red
    = [\M_{g(v)}^\sim(X, L, \beta(v), \mu(v), \emptyset)]^\vir.
  \end{equation*}
\end{itemize}

We also need to define
\begin{equation*}
  \frac{e(-R\pi_{v*}(\omega_{\pi_v} \otimes f^* L^\vee) \otimes [1])}{\prod_{e\text{ at }v} (\frac{t - \ev_e^*(c_1(L))}{\delta(e)} - \psi_e)} \cap [\M_{g(v), n(v)}(X, \beta(v))]^\vir
\end{equation*}
for unstable vertices $v$ at $0$: In the case that
$(g(v), n(v), \beta(v)) = (0, 2, 0)$, and that $v$ is connected to two
edges $e_1$ and $e_2$, we use
\begin{equation*}
  \frac{t - c_1(L)}{\frac{t - c_1(L)}{\delta(e_1)} + \frac{t - c_1(L)}{\delta(e_2)}},
\end{equation*}
in the case that $(g(v), n(v), \beta(v)) = (0, 2, 0)$, and that $v$ is
connected to one edge $e$ and has the $i$th marking, we use
$t - c_1(L)$, and finally, when $(g(v), n(v), \beta(v)) = (0, 1, 0)$
and $e$ is the unique edge at $v$, we use
\begin{equation*}
  \frac{(t - c_1(L))^2}{\delta(e)}.
\end{equation*}
The only case of unstable vertices $v$ at $\infty$ happens when
$(g(v), n(v), \beta(v)) = (0, 2, 0)$ and $\nu(v) \neq \emptyset$.
Then we define
\begin{equation*}
  p_*\left(\frac 1{-t - c_1(L_{\min})} [\M_{g(v)}^\sim(X, L, \beta(v), \mu(v), \nu(v))]^\vir\right)
  = 1.
\end{equation*}

\subsection{Comparison of virtual cycles}

The ordinary and reduced virtual cycle of the moduli space
$\M_{g,n}(X, L, \beta, \emptyset)$ only differ over the boundary
$\partial \M_{g,n}(X, L, \beta, \emptyset)$.
As in a similar situation in \cite{PaTh16}, the boundary contribution
can be made explicit.
\begin{theorem}[\cite{CJR18P}]
  There is a reduced virtual cycle
  $[\partial \M_{g,n}(X, L, \beta, \emptyset)]^\red$ such that
  \begin{equation*}
    [\M_{g,n}(X, L, \beta, \emptyset)]^\vir
    = [\M_{g,n}(X, L, \beta, \emptyset)]^\red + [\partial \M_{g,n}(X, L, \beta, \emptyset)]^\red.
  \end{equation*}
\end{theorem}

We make the comparison more explicit by decomposing the boundary
$\partial \M_{g,n}(X, L, \beta, \emptyset)$ according to ``tropical
data'', which is in correspondence to exactly the same decorated
graphs as those discussed in Section~\ref{sec:loc:formula}.

The graphs symbolize a decomposition of the source curve of the
log-section into a part mapping directly into the infinity section,
and one part not mapping directly into the infinity section.
The vertices $v \in V_\infty(\Gamma)$ (respectively,
$v \in V_0(\Gamma)$) make up the first (respectively, second) part.
There is exactly one graph, the graph consisting of a single vertex
$v \in V_0(\Gamma)$, that does not stand for any component of
$\partial \M_{g,n}(X, L, \beta, \emptyset)$.
It instead corresponds to $\M_{g,n}(X, L, \beta, \emptyset)$ itself.

Applying an additional comparison of connected and disconnected rubber
virtual cycles, \cite{CJR19P} arrives at:
\begin{theorem}
  \label{thm:compare}
  \begin{equation*}
    p_*([\M_{g, n}(X, L, \beta, \emptyset)]^\red)
    = \sum_\Gamma \frac 1{|\Aut(\Gamma)|} \iota_{\Gamma*} \Delta^!(C_\Gamma^\comp),
  \end{equation*}
  where
  \begin{multline*}
    C_\Gamma^\comp = \prod_{v \in V_0(\Gamma)} p_*([\M_{g(v), n(v)}(X, L, \beta(v), \nu(v))]^\vir) \\
    \prod_{v \in V_\infty(\Gamma)} p_*(-[\M^\sim_{g(v), n(v)}(X, L, \beta(v), \mu(v), \emptyset)]^\red) \prod_{e \in E(\Gamma)} \delta(e).
  \end{multline*}
\end{theorem}
\begin{remark}
  Note that
  $C_\Gamma^\comp = p_* [\M_{g,n}(X, L, \beta,
  \emptyset)]^\red$ in the case that $\Gamma$ is the graph consisting
  of a single vertex $v \in V_0(\Gamma)$.
  The sum of the contributions of the other graphs is
  $p_* [\partial \M_{g,n}(X, L, \beta, \emptyset)]^\red$.
\end{remark}
\begin{remark}
  \label{rmk:quintic}
  In the quintic case ($X = \PP^4$, $L = \cO(5)$), the classes
  $[\M^\sim_{g, n}(X, L, \beta, \mu, \emptyset)]^\red$ are determined
  from only $1 + \lfloor\frac{2g - 2}5\rfloor$ numbers in any genus:
  To see this, note first that $\mu$ is a partition of
  $2g - 2 + n - 5\beta$, and that the dimension of
  $[\M^\sim_{g, n}(\PP^4, \cO(5), \beta, \mu, \emptyset)]^\red$ is
  \begin{equation*}
    n - (2g - 2 - 5\beta).
  \end{equation*}  
  Therefore, in the case that all parts of $\mu$ are equal to $2$, the
  virtual dimension is zero, and we may define a constant
  $c_{g, \beta}^\eff$ for $5\beta \le 2g - 2$ by
  \begin{equation*}
    -[\M^\sim_{g, 2g - 2 - 5\beta}(\PP^4, \cO(5), \beta, (2^{2g - 2 - 5\beta}), \emptyset)]^\red
    = c_{g, \beta}^\eff [\mathrm{pt}].
  \end{equation*}

  Furthermore, in the case that some parts of $\mu$ are equal to $1$,
  the reduced virtual class is pulled back under the map forgetting
  those markings.
  Therefore, the virtual class vanishes unless all parts of $\mu$ are
  equal to $1$ or $2$, and when $\mu = (2^{2g - 2 - 5\beta}, 1^k)$, we
  have
  \begin{equation*}
    -[\M^\sim_{g, 2g - 2 - 5\beta}(\PP^4, \cO(5), \beta, (2^{2g - 2 - 5\beta}, 1^k), \emptyset)]^\red
    = c_{g, \beta}^\eff \pi_k^* [\mathrm{pt}],
  \end{equation*}
  where $\pi_k$ is the map forgetting the last $k$ markings.

  Therefore, $[\M^\sim_{g, n}(X, L, \beta, \mu, \emptyset)]^\red$ are
  determined from the numbers $c_{g, \beta}^\eff$ for
  $5\beta \le 2g - 2$.
  We call $c_{g, \beta}^\eff$ the \emph{effective constants}.
\end{remark}

\subsection{The extended quintic family}
\label{sec:loc:extended}

By Remark~\ref{rmk:quintic}, the comparison formula in
Theorem~\ref{thm:compare} involves effective constants
$c_{g, \beta}^\eff$.
Viewing the effective constants as parameters, we define the extended
quintic family of theories.
Given a choice $\mathbf c$ of constants $c_{g, \beta}$ for
$5\beta \le 2g - 2$, we formally define the extended quintic virtual
class
\begin{multline*}
  [Q_{g, n, \beta}^{\mathbf c}]^\vir
  := \sum_\Gamma \frac 1{|\Aut(\Gamma)|} \prod_{e \in E(\Gamma)} \delta(e) \\
  \iota_{\Gamma*} \Delta^!\left(\prod_{v \in V_0(\Gamma)} p_*([\M_{g(v), n(v)}(X, L, \beta(v), \nu(v))]^\vir) \prod_{v \in V_\infty(\Gamma)} c_{g(v), \beta(v), \mu(v)}\right),
\end{multline*}
where we set
\begin{equation*}
  c_{g, \beta, \mu} :=
  \begin{cases}
    0 & \text{if $\mu$ has a part of size $\ge 3$}, \\
    c_{g, \beta} \pi_{\ell(\mu) - (2g - 2 - 5\beta)}^* [\mathrm{pt}].
  \end{cases}
\end{equation*}
This is a cycle on $\M_{g, n}(\PP^4, \beta)$.

\section{From formal quintic to quintic}
\label{sec:rewrite}

The goal of this section is to prove a graph sum formula expressing
the extended quintic family in terms of the formal quintic.
We will prove a cycle-valued formula that will be the an essential
ingredient in the Combinatorial Structure Theorem for the quintic
family that we will prove in the following
Section~\ref{sec:Rmatrixaction}.

\subsection{Twisted virtual cycles}
\label{sec:twisted-vc}

Compared to the previous section, we now specialize to the case that
$X = \PP^4$ and $L = \cO(5)$.\footnote{The analysis of this and the
  following section can be performed analogously for the case where
  $X$ is any projective space and $L = \cO(k)$ for some $k > 0$.} In
this case, the localization formula of Section~\ref{sec:loc:formula}
involves the \emph{twisted virtual cycle}
\begin{equation*}
  e(-R\pi_*(\omega_\pi \otimes f^* \cO(-5)) \otimes [1])\cap [\M_{g, n}(\PP^4, \beta)]^\vir,
\end{equation*}
which we can rewrite via Serre duality as
\begin{multline}
  \label{eq:serre}
  e(-R\pi_*(\omega_\pi \otimes f^* \cO(-5)) \otimes [1])\cap [\M_{g, n}(\PP^4, \beta)]^\vir\\
  = (-1)^{1 - g + 5\beta} e(R\pi_*(f^* \cO(5)) \otimes [1])\cap [\M_{g, n}(\PP^4, \beta)]^\vir.
\end{multline}
In the following, we call the cohomological field theory defined using
the virtual cycle
\begin{equation*}
  e(R\pi_*(f^* \cO(5)) \otimes [1])\cap [\M_{g, n}(\PP^4, \beta)]^\vir
\end{equation*}
the \emph{$t$-twisted theory}.

This theory is closely related to the \emph{formal quintic} theory (we
also refer to it as the \emph{$\lambda$-twisted theory}) of
\cite{GuRo16P, GuRo17P, LhPa18, LhPa18P, ZaZi08} in genus zero.
To explain how they differ, it is useful to consider the \emph{fully
  equivariant twisted theory}, which is defined in the same way as the
$t$-twisted theory except that we use the $(\CC^*)^5$-equivariant
virtual cycle $[\M_{g, n}(\PP^4, \beta)]^\vir_{(\CC^*)^5}$ of
$\M_{g, n}(\PP^4, \beta)$ with respect to the diagonal
$(\CC^*)^5$-action on $\PP^4$.
Let $\lambda_i$ denote the corresponding additional equivariant
parameters.

Then the $t$-twisted theory can be obtained by specializing
$\lambda_i = 0$.
On the other hand, the $\lambda$-twisted theory is obtained by setting
$t = 0$, which is possible since $c_1(\cO(5))$ is invertible in
equivariant cohomology, as well as specializing
$\lambda_i = \zeta^i \lambda$ where $\lambda$ is one parameter and
$\zeta$ is a fifth root of unity.
We accordingly define the formal quintic virtual cycle:
\begin{equation*}
  [\M_{g, n}(\PP^4, \beta)]^{\vir, \lambda}
  = \left(e(R\pi_* f^* \cO(5) \otimes [1]) \cap [\M_{g, n}(\PP^4, \beta)]^\vir_{(\CC^*)^5}\right)\Big|_{t = 0, \lambda_i = \zeta^i \lambda}.
\end{equation*}

\begin{remark}
  In genus zero, the two twisted theories admit non-equivariant limits
  $t = 0$, $\lambda = 0$, which agree.
  However, in higher genus, it is not possible to take these limits,
  and taking the $t^0$-coefficient in the Laurent expansion of the
  $t$-twisted theory will give a different result to the
  $\lambda^0$-coefficient in the $\lambda$-twisted theory.
\end{remark}

\subsection{Full torus localization}
\label{sec:loc:full}

In addition to the $\CC^*$-action on
$\M_{g,n}(\PP^4, \cO(5), \beta, \nu)$ discussed in
Section~\ref{sec:loc:formula}, we may also use the diagonal
$(\CC^*)^5$-action on $\PP^4$ to define a $(\CC^*)^6$-action on
$\M_{g,n}(\PP^4, \cO(5), \beta, \nu)$.
The ordinary perfect obstruction theory of
$\M_{g,n}(\PP^4, \cO(5), \beta, \nu)$ is equivariant for the full
$(\CC^*)^6$-action, while (when defined), the reduced perfect
obstruction theory is only $\CC^*$-equivariant.

The localization formula of Section~\ref{sec:loc:formula} can clearly
be lifted to $(\CC^*)^6$-equivariant cohomology.
We note that this is not the same as the localization formula for the
$(\CC^*)^6$-action on $\M_{g,n}(\PP^4, \cO(5), \beta, \nu)$ because we
still use the fixed loci for the $\CC^*$-action and not the (smaller)
fixed loci of the $(\CC^*)^6$-action.

We now consider the specialization of equivariant parameters
\begin{equation}
  \label{eq:special}
  (\lambda_0, \lambda_1, \lambda_2, \lambda_3, \lambda_4, t) = (\lambda, \zeta\lambda, \zeta^2\lambda, \zeta^3\lambda, \zeta^4\lambda, 0).
\end{equation}
Under this specialization, the localization formula from
Section~\ref{sec:loc:formula} can be written as
\begin{equation*}
  p_*([\M_{g,n}(\PP^4, \cO(5), \beta, \nu)]^\vir)
  = \lim_{\lambda \to 0} \sum_\Gamma \frac 1{|\Aut(\Gamma)|} \iota_{\Gamma*} \Delta^! (C_\Gamma^\vir)
\end{equation*}
where
\begin{multline*}
  C_\Gamma^\vir = \prod_{v \in V_0(\Gamma)} \frac{(-1)^{1 - g(v) + 5\beta(v)}}{\prod_{e\text{ at }v} (\frac{-\ev_e^*(5H)}{\delta(e)} - \psi_e)} \cap [\M_{g(v), n(v)}(\PP^4, \beta(v))]^{\vir, \lambda} \\
  \prod_{v \in V_\infty(\Gamma)} p_*\left(\frac 1{-c_1(L_{\min})}
    [\M_{g(v)}^\sim(\PP^4, \cO(5), \beta(v), \mu(v), \nu(v))]^\vir_{(\CC^*)^5}\right)
  \prod_{e \in E(\Gamma)} \frac 1{\delta(e) \prod_{i = 1}^{\delta(e)}
    \frac{-i\ev_e^*(5H)}{\delta(e)}},
\end{multline*}
where and all other notation and exceptional cases are similar as in
Section~\ref{sec:loc:formula}.
\begin{remark}
  \label{rmk:invertible}
  It is not obvious that we could perform the formal quintic
  specialization.
  The next lemma addresses invertibility of $c_1(L_{\min})$.
  To see that $\frac{-\ev_e^*(5H)}{\delta(e)} - \psi_e$ is invertible
  note that on each fixed locus for localization on
  $\M_{g, n}(\PP^4, \beta)$, $H$ specializes to an element of the form
  $\zeta^i \lambda$, whereas $\psi_e$ is either nilpotent or an
  element of the form $(\zeta^j - \zeta^k)\lambda$ for $j \neq k$.
\end{remark}
\begin{lemma}[\cite{CJR19P}]
  $c_1(L_{\min})$ is invertible in the $(\CC^*)^5$-equivariant
  cohomology of
  $$\M_{g(v)}^\sim(\PP^4, \cO(5), \beta(v), \mu(v), \nu(v)).$$
\end{lemma}

\subsection{Tripartite graphs}

{\footnotesize
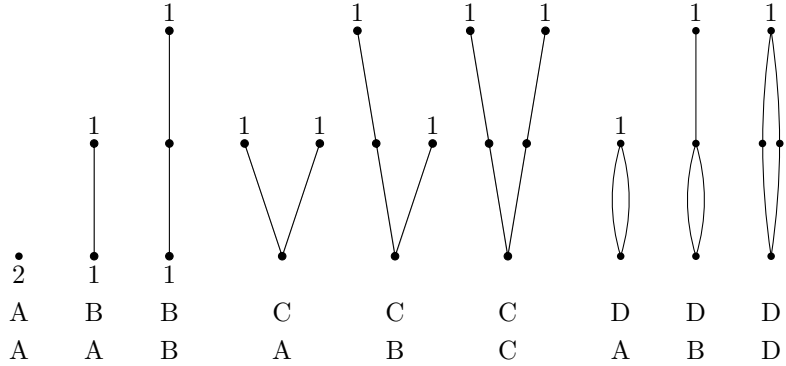
\begin{figure}
  \centering
  \begin{tikzpicture}[scale=0.5]
    \fill (0, 0) circle(1mm) node[below] {$2$};
    \draw (0, -1) node[below] {A} (0, -2) node[below] {A};
    \draw[fill] (2, 0) circle(1mm) node[below] {$1$} -- (2, 3) circle(1mm) node[above] {$1$};
    \draw (2, -1) node[below] {B} (2, -2) node[below] {A};
    \draw[fill] (4, 0) circle(1mm) node[below] {$1$} -- (4, 3) circle(1mm) -- (4, 6) circle(1mm) node[above] {$1$};
    \draw (4, -1) node[below] {B} (4, -2) node[below] {B};
    \draw[fill] (6, 3) node[above] {$1$} circle(1mm) -- (7, 0) circle(1mm) -- (8, 3) circle(1mm) node[above] {$1$};
    \draw (7, -1) node[below] {C} (7, -2) node[below] {A};
    \draw[fill] (9, 6) node[above] {$1$} circle(1mm) -- (9.5, 3) circle(1mm) -- (10, 0) circle(1mm) -- (11, 3) circle(1mm) node[above] {$1$};
    \draw (10, -1) node[below] {C} (10, -2) node[below] {B};
    \draw[fill] (12, 6) node[above] {$1$} circle(1mm) -- (12.5, 3) circle(1mm) -- (13, 0) circle(1mm) -- (13.5, 3) circle(1mm) -- (14, 6) circle(1mm) node[above] {$1$};
    \draw (13, -1) node[below] {C} (13, -2) node[below] {C};
    \draw (16, 0) .. controls (15.7, 1) and (15.7, 2) .. (16, 3);
    \draw (16, 0) .. controls (16.3, 1) and (16.3, 2) .. (16, 3);
    \fill (16, 0) circle(1mm); \fill (16, 3) circle(1mm) node[above] {$1$};
    \draw (16, -1) node[below] {D} (16, -2) node[below] {A};
    \draw (18, 0) .. controls (17.7, 1) and (17.7, 2) .. (18, 3);
    \draw (18, 0) .. controls (18.3, 1) and (18.3, 2) .. (18, 3) -- (18, 6);
    \fill (18, 0) circle(1mm); \fill (18, 3) circle(1mm); \fill (18, 6) circle(1mm) node[above] {$1$};
    \draw (18, -1) node[below] {D} (18, -2) node[below] {B};
    \draw (20, 0) .. controls (19.7, 2) and (19.7, 4) .. (20, 6);
    \draw (20, 0) .. controls (20.3, 2) and (20.3, 4) .. (20, 6);
    \fill (20, 0) circle(1mm); \fill (20, 6) circle(1mm) node[above] {$1$};
    \fill (19.77, 3) circle(1mm); \fill (20.23, 3) circle (1mm);
    \draw (20, -1) node[below] {D} (20, -2) node[below] {D};
  \end{tikzpicture}
  \caption{The tripartite graphs in genus two, part 1. The letters under each graph indicate which genus two bipartite graph is obtained when merging the upper two or the lower two level, respectively.}
  \label{fig:g2graph1}
\end{figure}
\begin{figure}
  \centering
  \begin{tikzpicture}[scale=0.5]
    \draw[fill] (22, 0) circle(1mm) -- (22.5,1.5) node[left] {$2$} -- (23, 3) circle(1mm) node[above] {$2$};
    \draw[fill] (24, 0) circle(1mm) -- (23.5,1.5) node[right] {$2$} -- (23, 3);
    \draw (23, -1) node[below] {E} (23, -2) node[below] {A};
    \draw[fill] (25, 0) circle(1mm) -- (25.5,1.5) node[left] {$2$} -- (26, 3) circle(1mm) node[left] {$1$} -- (26, 6) circle(1mm) node[above] {$1$};
    \draw[fill] (27, 0) circle(1mm) -- (26.5,1.5) node[right] {$2$} -- (26, 3);
    \draw (26, -1) node[below] {E} (26, -2) node[below] {B};
    \draw[fill] (28, 0) circle(1mm) -- (28.5,1.5) node[left] {$2$} -- (29, 3) circle(1mm) -- (30, 6) circle(1mm) node[above] {$1$};
    \draw[fill] (30, 0) circle(1mm) -- (29.5,1.5) node[right] {$2$} -- (29, 3) -- (28, 6) circle(1mm) node[above] {$1$};
    \draw (29, -1) node[below] {E} (29, -2) node[below] {C};
    \draw[fill] (31, 0) circle(1mm) -- (31.5,1.5) node[left] {$2$} -- (32, 3) circle(1mm) (32, 6) circle(1mm) node[above] {$1$};
    \draw[fill] (33, 0) circle(1mm) -- (32.5,1.5) node[right] {$2$} -- (32, 3);
    \draw (32, 3) .. controls (32.3, 4) and (32.3, 5) .. (32, 6);
    \draw (32, 3) .. controls (31.7, 4) and (31.7, 5) .. (32, 6);
    \draw (32, -1) node[below] {E} (32, -2) node[below] {D};
    \draw[fill] (34, 0) circle(1mm) -- (34.25,1.5) node[left] {$2$} -- (34.5, 3) circle(1mm) -- (34.75, 4.5) node[left] {$2$} -- (35, 6) circle(1mm) node[above] {$2$};
    \draw[fill] (36, 0) circle(1mm) -- (35.75, 1.5) node[right] {$2$} -- (35.5,3) circle(1mm) -- (35.25, 4.5) node[right] {$2$} -- (35, 6);
    \draw (35, -1) node[below] {E} (35, -2) node[below] {E};
    \draw[fill] (37, 0) circle(1mm) -- (37.5,1.5) node[left] {$2$} -- (38, 3) circle(1mm) -- (38, 4.5) node[left] {$3$} -- (38, 6) circle(1mm) node[above] {$2$};
    \draw[fill] (39, 0) circle(1mm) -- (38.5,1.5) node[right] {$2$} -- (38, 3);
    \draw (38, -1) node[below] {E} (38, -2) node[below] {F};
    \draw[fill] (40, 0) circle(1mm) -- (40, 1.5) node[left] {$3$} -- (40, 3) circle(1mm) node[above] {$2$};
    \draw (40, -1) node[below] {F} (40, -2) node[below] {A};
    \draw[fill] (42, 0) circle(1mm) -- (42, 1.5) node[left] {$3$} -- (42, 3) circle(1mm) node[left] {$1$} -- (42, 6) circle(1mm) node[above] {$1$};
    \draw (42, -1) node[below] {F} (42, -2) node[below] {B};
    \draw[fill] (45, 0) circle(1mm) -- (45, 1.5) node[left] {$3$} -- (45, 3) circle(1mm) -- (44, 6) circle(1mm) node[above] {$1$};
    \draw[fill] (45, 3) -- (46, 6) circle(1mm) node[above] {$1$};
    \draw (45, -1) node[below] {F} (45, -2) node[below] {C};
    \draw[fill] (48, 0) circle(1mm) -- (48, 1.5) node[left] {$3$} -- (48, 3) circle(1mm) (48, 6) circle(1mm) node[above] {$1$};
    \draw (48, 3) .. controls (48.3, 4) and (48.3, 5) .. (48, 6);
    \draw (48, 3) .. controls (47.7, 4) and (47.7, 5) .. (48, 6);
    \draw (48, -1) node[below] {F} (48, -2) node[below] {D};
    \draw[fill] (50, 0) circle(1mm) -- (50, 1.5) node[left] {$3$} -- (50, 3) circle(1mm) -- (50, 4.5) node[left] {$3$} -- (50, 6) circle(1mm) node[above] {$2$};
    \draw (50, -1) node[below] {F} (50, -2) node[below] {F};
  \end{tikzpicture}
  \caption{The tripartite graphs in genus two, part 2.}
  \label{fig:g2graph2}
\end{figure}
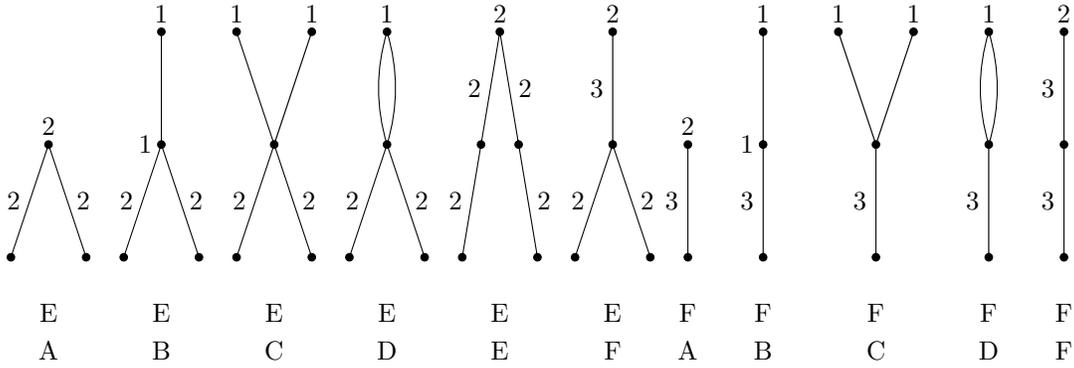
}

We now combine the comparison of Theorem~\ref{thm:compare}, and more
generally, the definition of the extended quintic family in
Section~\ref{sec:loc:extended}, with the localization formula of
Section~\ref{sec:loc:full} to the computation of
$[Q_{g, n, \beta}^{\mathbf c}]^\vir$.

The comparison formula is already a sum over bipartite graphs.
Applying localization to
$p_*([\M_{g(v), n(v)}(\PP^4, \cO(5), \beta(v), \nu(v))]^\vir)$ for
each $v \in V_0(\Gamma)$, we arrive at a sum over graphs with three
different types of vertices.

The decorations of such a graph $\Gamma$ consist of
\begin{itemize}
\item the tripartite structure $V(\Gamma) = V_l(\Gamma) \sqcup V_m(\Gamma) \sqcup V_u(\Gamma)$
\item a genus mapping $g\colon V(\Gamma) \to \ZZ_{\ge 0}$
\item a curve class mapping $\beta\colon V(\Gamma) \to \mathrm{Pic}(X)^\vee$
\item a fiber class mapping $\delta\colon E(\Gamma) \to \ZZ_{\ge 1}$
\item a distribution of markings
  $m\colon \{1, \dotsc, n\} \to V(\Gamma)$
\end{itemize}
When $v \in V_l(\Gamma)$ (respectively, $v \in V_m(\Gamma)$,
$v \in V_u(\Gamma)$), we say that $v$ is in the lower (respectively,
middle, upper) level.
Edges must always connect vertices of different levels.
We can therefore decompose
$E(\Gamma) = E_{lm}(\Gamma) \sqcup E_{lu}(\Gamma) \sqcup
E_{mu}(\Gamma)$ according to which levels are connected by an edge.
For $v \in V_m(\Gamma) \sqcup V_u(\Gamma)$, we set $\mu(v)$
(respectively, $\nu(v)$) to be the partitions formed by the
$\delta(e)$ for all $e$ at $v$ connecting to a lower (respectively,
upper) level.
Note that if $v \in V_u(\Gamma)$, then $\nu(v) = \emptyset$.
They need to satisfy the constraints:
\begin{itemize}
\item $E_{lu}(\Gamma) = \emptyset$
\item $\sum_v g(v) + h^1(\Gamma) = g$
\item $\sum_v \beta(v) = \beta$
\item
  $\forall v \in V_m(\Gamma) \sqcup V_u(\Gamma): |\mu(v)| - |\nu(v)| =
  2g(v) - 2 + n(v) - \int_{\beta(v)} c_1(X)$
\item If for $v \in V(\Gamma)$, we have
  $(g(v), n(v), \beta(v)) = (0, 1, 0)$, then the unique incident edge
  $e$ has $\delta(e) > 1$.
\end{itemize}
We have listed all tripartite graphs in the case that
$(g, n) = (2, 0)$ (and we use a stable quotient theory) in
Figure~\ref{fig:g2graph1} and \ref{fig:g2graph2}.

We define $M_\Gamma$, $\widetilde M_\Gamma$ and
\begin{equation*}
  \iota_\Gamma\colon \widetilde M_\Gamma \to \M_{g, n}(\PP^4, \beta)
\end{equation*}
analogously as in Section~\ref{sec:loc:formula}.
There is an additional case of an unstable vertex $v$, that is, when
$v \in V_m(\Gamma)$ and $(g(v), n(v), \beta(v)) = (0, 2, 0)$.

We can now write $p_*([\M_{g,n}(\PP^4, \cO(5), \beta, \emptyset)]^\red)$ as
a sum over decorated tripartite graphs:
\begin{equation*}
  [Q_{g, n, \beta}^{\mathbf c}]^\vir
  = \lim_{\lambda \to 0} \sum_\Gamma \frac 1{|\Aut(\Gamma)|} \iota_{\Gamma*} \Delta^! (C_\Gamma^3),
\end{equation*}
where
\begin{multline*}
  C_\Gamma^3 = \prod_{v \in V_l(\Gamma)} \frac{(-1)^{1 - g(v) + 5\beta(v)}}{\prod_{e\text{ at }v} (\frac{-\ev_e^*(5H)}{\delta(e)} - \psi_e)} \cap [\M_{g(v), n(v)}(\PP^4, \beta(v))]^{\vir, \lambda} \\
  \prod_{v \in V_m(\Gamma)} p_*\left(\frac 1{-c_1(L_{\min})} [\M_{g(v)}^\sim(\PP^4, \cO(5), \beta(v), \mu(v), \nu(v))]^\vir_{(\CC^*)^5}\right) \\
  \prod_{v \in V_u(\Gamma)} c_{g(v), \beta(v), \mu(v)}
  \prod_{e \in E_{lm}(\Gamma)} \frac 1{\delta(e) \prod_{i = 1}^{\delta(e)}
    \frac{-i\ev_e^*(5H)}{\delta(e)}}
  \prod_{e \in E_{mu}(\Gamma)} \delta(e).
\end{multline*}
Here the notations are similar to those in
Section~\ref{sec:loc:formula}, as are the contributions of unstable
vertices $v \in V_l(\Gamma)$ and $v \in V_m(\Gamma)$.

{\footnotesize
\begin{figure}
  \centering
  \begin{tikzpicture}[scale=0.5]
    \draw[fill] (0, -0.5) -- (0, 0) circle(1mm) node[above] {$1$};
    \draw[fill] (2, 2.5) --  (2, 3) circle(1mm) node[above] {$1$};

    \draw[fill] (5, -0.5) -- (5, 0) circle(1mm) node[above] {$2$};
    \draw[fill] (7, -0.5) -- (7, 0) circle(1mm) node[left] {$1$} -- (7, 3) circle(1mm) node[above] {$1$};
    \draw[fill] (10, -0.5) -- (10, 0) circle(1mm) -- (11, 3) circle(1mm) node[above] {$1$};
    \draw[fill] (10, 0) -- (9, 3) circle(1mm) node[above] {$1$};
    \draw (13, 0) .. controls (12.7, 1) and (12.7, 2) .. (13, 3);
    \draw (13, 0) .. controls (13.3, 1) and (13.3, 2) .. (13, 3);
    \draw[fill] (13, -0.5) -- (13, 0) circle(1mm) (13, 3) circle(1mm) node[above] {$1$};
    \draw[fill] (15, 2.5) --  (15, 3) circle(1mm) node[above] {$2$};
  \end{tikzpicture}
  \caption{Bipartite graphs for $(g, \mu) = (1, (1))$, $(2, (3))$}
  \label{fig:comp1}
\end{figure}
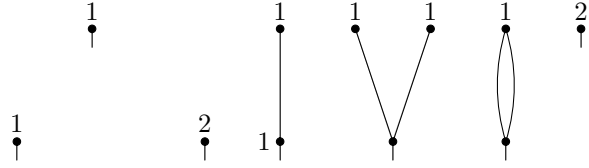
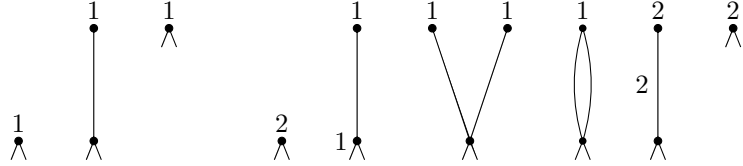
\begin{figure}
  \centering
  \begin{tikzpicture}[scale=0.5]
    \draw[fill] (4.8, -0.5) -- (5, 0) circle(1mm) node[above] {$1$};
    \draw (5.2, -0.5) -- (5, 0);
    \draw[fill] (6.8, -0.5) -- (7, 0) circle(1mm) -- (7, 3) circle(1mm) node[above] {$1$};
    \draw (7.2, -0.5) -- (7, 0);
    \draw[fill] (8.8, 2.5) -- (9, 3) circle(1mm) node[above] {$1$};
    \draw (9.2, 2.5) -- (9, 3);

    \draw[fill] (11.8, -0.5) -- (12, 0) circle(1mm) node[above] {$2$};
    \draw (12.2, -0.5) -- (12, 0);
    \draw[fill] (13.8, -0.5) -- (14, 0) circle(1mm) node[left] {$1$} -- (14, 3) circle(1mm) node[above] {$1$};
    \draw (14.2, -0.5) -- (14, 0);
    \draw[fill] (16.8, -0.5) -- (17, 0) circle(1mm) -- (18, 3) circle(1mm) node[above] {$1$};
    \draw[fill] (17.2, -0.5) -- (17, 0) -- (16, 3) circle(1mm) node[above] {$1$};
    \draw (20, 0) .. controls (19.7, 1) and (19.7, 2) .. (20, 3);
    \draw (20, 0) .. controls (20.3, 1) and (20.3, 2) .. (20, 3);
    \fill (20, 0) circle(1mm) (20, 3) circle(1mm) node[above] {$1$};
    \draw (19.8, -0.5) -- (20, 0) -- (20.2, -0.5);
    \draw[fill] (21.8, -0.5) -- (22, 0) circle(1mm) -- (22, 1.5) node[left] {$2$} -- (22, 3) circle(1mm) node[above] {$2$};
    \draw (22.2, -0.5) -- (22, 0);
    \draw[fill] (23.8, 2.5) -- (24, 3) circle(1mm) node[above] {$2$};
    \draw (24.2, 2.5) -- (24, 3);
  \end{tikzpicture}
  \caption{Bipartite graphs for $(g, \mu) = (1, (1, 1))$, $(2, (2, 2))$}
  \label{fig:comp2}
\end{figure}
}

We now rewrite the formula as a sum over bipartite graphs by merging
the middle and upper levels.
The result is a sum over decorated graphs as in
Section~\ref{sec:loc:formula}:
\begin{equation}
  \label{eq:quinticlambda}
  [Q_{g, n, \beta}^{\mathbf c}]^\vir
  = \lim_{\lambda \to 0}\sum_\Gamma \frac 1{|\Aut(\Gamma)|} \iota_{\Gamma*} \Delta^! (C_\Gamma^2),
\end{equation}
where
\begin{multline*}
  C_\Gamma^2 = \prod_{v \in V_0(\Gamma)} \frac{(-1)^{1 - g(v) + 5\beta(v)}}{\prod_{e\text{ at }v} (\frac{-\ev_e^*(5H)}{\delta(e)} - \psi_e)} \cap [\M_{g(v), n(v)}(\PP^4, \beta(v))]^{\vir, \lambda} \\
  \prod_{v \in V_\infty(\Gamma)} C_{g(v), \beta(v), \mu(v)}
  \prod_{e \in E(\Gamma)} \frac 1{\delta(e) \prod_{i = 1}^{\delta(e)}
    \frac{-i\ev_e^*(5H)}{\delta(e)}}
\end{multline*}
for classes $C_{g, \beta, \mu} \in H_*(\M_{g, \mu}(\PP^4, \beta))$
that themselves can be written as a sum over a class of bipartite
decorated graphs $\Gamma$ that we will describe now.

The decorations are given by:
\begin{itemize}
\item the bipartite structure $V(\Gamma) = V_0(\Gamma) \sqcup V_\infty(\Gamma)$
\item a genus mapping $g\colon V(\Gamma) \to \ZZ_{\ge 0}$
\item a curve class mapping $\beta\colon V(\Gamma) \to \mathrm{Pic}(X)^\vee$
\item a fiber class mapping $\delta\colon E(\Gamma) \to \ZZ_{\ge 1}$
\item a distribution of markings $m\colon \mu \to V_0(\Gamma)$
\end{itemize}
We set for all $v \in V_0(\Gamma)$ that $\mu(v) = |m^{-1}(\mu)|$ and
that $\nu(v)$ is the partition formed by $\delta(e)$ for all edges $e$
at $v$.
For all $v \in V_\infty(\Gamma)$, the partition formed by $\delta(e)$
for all edges $e$ at $v$ is denoted by $\mu(v)$, while we set that
$\nu(v) = \emptyset$.

We require the decorated graphs to satisfy the following conditions:
\begin{itemize}
\item $\sum_v g(v) + h^1(\Gamma) = g$
\item $\sum_v \beta(v) = \beta$
\item
  $\forall v \in V_0(\Gamma): |\mu(v)| - |\nu(v)| = 2g(v) - 2 + n(v) -
  5 \beta(v)$
\item There are no vertices $v \in V(\Gamma)$ such that
  $(g(v), n(v), \beta(v)) = (0, 1, 0)$.
\end{itemize}
We note that the conditions imply that $\mu$ is a partition of
$2g - 2 - 5\beta$.

Then, we can define $C_{g, \beta, \mu}$ via
\begin{equation}
  \label{eq:lambdaloc}
  C_{g, \beta, \mu}
  = \sum_\Gamma \frac 1{|\Aut(\Gamma)|} \iota_{\Gamma*} \Delta^! (C_\Gamma^{mu}),
\end{equation}
where
\begin{multline*}
  C_\Gamma^{mu} =
  \prod_{v \in V_0(\Gamma)} p_*\left(\frac 1{-c_1(L_{\min})} [\M_{g(v)}^\sim(\PP^4, \cO(5), \beta(v), \mu(v), \nu(v))]^\vir_{(\CC^*)^5}\right) \\
  \prod_{v \in V_\infty(\Gamma)} c_{g(v), \beta(v), \mu(v)}
  \prod_{e \in E(\Gamma)} \delta(e).
\end{multline*}

All in all, we have found a formula \eqref{eq:quinticlambda}
expressing the extended quintic virtual cycle in terms of the formal
quintic and cycles \eqref{eq:lambdaloc} that depend on the chosen
element of the extended quintic family.
\begin{remark}
  \label{rmk:fullquintic}
  We could further extend the quintic family by allowing the
  $C_{g, \beta, \mu}$ to be arbitrary cycles subject to the constraint
  that $C_{g, \beta, \mu}$ is only nonzero if $\mu$ is a partition of
  $2g - 2 - 5\beta$.
\end{remark}

\section{The combinatorial structure theorem}
\label{sec:Rmatrixaction}

In Section~\ref{sec:loc:extended}, we introduced the extended quintic
family.
It contains the quintic and holomorphic GLSM theories.
In this section, we prove a structural theorem for the extended
quintic family, which sets the stage for the proof of the main
theorems.

\subsection{Formal quintic CohFT}
\label{sec:twisted-CohFT}

In Section~\ref{sec:twisted-vc} we introduced the $\lambda$-twisted
virtual cycle $[\M_{g, n}(\PP^4, \beta)]^{\vir, \lambda}$.
We now define the corresponding cohomological field theory (CohFT).
The state space is $\H = H^*(\PP^4)$, the pairing is given by
\begin{equation*}
  (\gamma_1, \gamma_2)^\lambda
  = \Big(\int_{\PP^4, (\CC^*)^5} e(\cO(5)) \gamma_1\gamma_2\Big) \Big|_{\lambda_i = \zeta^i \lambda},
\end{equation*}
and the unit is $\mathbf 1$, the unit in cohomology.
We now define the $\lambda$-twisted cohomological field theory
$\Omega^\lambda$ by
\begin{equation*}
  \Omega_{g, n}^\lambda(\gamma_1, \dotsc, \gamma_n)
  = \sum_{\beta = 0}^\infty q^\beta \rho_*\left(\prod_{i = 1}^n \ev_i^*(\gamma_i) \cap [\M_{g, n}(\PP^4, \beta)]^{\vir, \lambda}\right),
\end{equation*}
where $\rho\colon \M_{g, n}(\PP^4, \beta) \to \M_{g, n}$ is the
forgetful map.
More generally, for any (formal) $\tau \in \H$, we may define the
$\tau$-shifted $\lambda$-twisted cohomological field theory
$\Omega^\lambda$ by
\begin{equation*}
  \Omega_{g, n}^{\lambda, \tau}(\gamma_1, \dotsc, \gamma_n)
  = \sum_{\beta, k = 0}^\infty \frac{q^\beta}{k!} \rho_*\pi_*^k\left(\prod_{i = 1}^n \ev_i^*(\gamma_i) \prod_{i = n + 1}^{n + k} \ev_i^*(\tau) \cap [\M_{g, n + k}(\PP^4, \beta)]^{\vir, \lambda}\right),
\end{equation*}
where
$\pi^k\colon \M_{g, n + k}(\PP^4, \beta) \to \M_{g, n}(\PP^4, \beta)$
is the forgetful map.

The ($\tau$-shifted) $\lambda$-twisted cohomological field theory is
semi-simple.
Hence there exists a basis $\{e_\alpha\}$ of idempotents, and the
Givental--Teleman reconstruction theorem \cite{Te12, PPZ15} applies.
This theorem determines $\Omega^\lambda$ from its degree-zero part
(topological field theory) $\omega$, and an $R$-matrix
\begin{equation*}
  R_\tau(z) \in \End(\H) \otimes \QQ(\lambda)[[q, z]]
\end{equation*}
satisfying the symplectic condition
\begin{equation*}
  R_\tau(z) R^*_\tau(-z) = \mathrm{Id}.
\end{equation*}

To motivate the extended $R$-matrix action, we recall the statement of
the Givental--Teleman theorem.
Let $G_{g,n}$ be the set of stable graphs of genus $g$ with vertices
and $n$ legs.
For any $\Gamma$ in $G_{g,n}^{\infty}$ or $G_{g, n}$, there is a
corresponding gluing map
\begin{equation*}
  \iota_\Gamma\colon \M_\Gamma := \prod_v \M_{g(v), n(v)} \to \M_{g, n}.
\end{equation*}

We then, for each $\Gamma \in G_{g, n}$, define the contribution
\begin{equation*}
  \Cont_\Gamma\colon \H^{\otimes n} \to \M_\Gamma
\end{equation*}
by the following contraction of tensors
\begin{enumerate}
\item at each flag $f=(v, l)$ with leg insertion $\gamma$, we place
  $R^{-1}_\tau(z_f)\gamma$;
\item at each edge $e$ connecting two vertices $v_1,v_2$, we place
  \begin{equation*}
    \frac{\sum_\alpha e_\alpha \otimes e^\alpha - R^{-1}_\tau(z_{f_1}) e_\alpha \otimes R^{-1}_\tau(z_{f_2}) e^\alpha}{z_{f_1} + z_{f_2}}
  \end{equation*}
  where $f_1:=(v_1,e), f_2:=(v_2,e)$ and $\{e^\alpha\}$ is the dual
  basis to $\{e_\alpha\}$, and
\item at each vertex $v$ with $m$ flags $f_1, \dotsc, f_m$, we place the map
  \begin{multline*}
    \gamma_1(z_{f_1}) \otimes \dotsb \otimes \gamma_m(z_{f_m}) \mapsto
    T\omega_{g(v), m} (\gamma_1(\bar\psi_1)\otimes \cdots \otimes \gamma_m(\bar\psi_m) ) \\
    := \sum_{k = 0}^\infty \frac 1{k!} \pi_{k*}\omega_{g(v), m + k} \left(\gamma_1(\bar\psi_1)\otimes \cdots \otimes \gamma_m(\bar\psi_m) \otimes \bigotimes_{i = 1}^k T(\psi_{m + i})\right),
  \end{multline*}
  where
  \begin{equation*}
    T(z) = z(\mathbf 1 - R^{-1}_\tau(z) \mathbf 1),
  \end{equation*}
  and the $\bar\psi$ are the (ancestor) psi classes.
\end{enumerate}
\begin{theorem}[Givental--Teleman]
  \begin{equation*}
    \Omega_{g, n}^\lambda
    = \sum_{\Gamma\in G_{g,n}} \frac{1}{|\Aut(\Gamma)|}  \iota_{\Gamma*} \Cont_\Gamma
  \end{equation*}
\end{theorem}

For later reference, we recall the definition of the twisted
$J$-function and the (inverse) $S$-matrix.
We allow for a more general shift $\bt(z) \in \H[[t]]$:
\begin{equation*}
  J_\bt(z)
  := z\mathbf 1 + \bt(-z) + \sum_{\beta, k = 0}^\infty \frac{q^\beta}{k!} \ev_{1*} \left(\frac 1{z - \psi_1} \prod_{i = 2}^{k + 1} \ev_i^*(\bt(\psi_i)) \cap [\M_{0, k + 1}(\PP^4, \beta)]^{\vir, \lambda}\right)
\end{equation*}
\begin{equation*}
  S^{-1}_\bt(z) \gamma
  := \gamma + \sum_{\beta, k = 0}^\infty \frac{q^\beta}{k!} \ev_{1*} \left(\frac 1{-z - \psi_1} \ev_2^*(\gamma) \prod_{i = 3}^{k + 2} \ev_i^*(\bt(\psi_i)) \cap [\M_{0, k + 2}(\PP^4, \beta)]^{\vir, \lambda}\right)
\end{equation*}

\subsection{The extended quintic family as a generalized $R$-matrix action}
\label{NewRaction}

We use the extended quintic virtual cycle
$[Q_{g, n, \beta}^{\mathbf c}]^\vir$ to define ($\bt$-shifted)
extended quintic classes for any $\gamma_1, \dotsc, \gamma_n \in \H$.
We use a slightly artificial definition in order to avoid making
assumptions on the effective constants.
We define
\begin{multline*}
  \Omega_{g, n}^{\mathbf c, \bt} (\gamma_1, \dotsc, \gamma_n) \\
  := (I_0)^{2g - 2 + n} \sum_{\beta, k = 0}^\infty \frac{q^\beta}{k!} (-1)^{1 - g + 5\beta} \rho_* \pi_*^k\left(\prod_{i = 1}^n \ev_i^*(\gamma_i) \prod_{i = n + 1}^{n + k} \ev_i^*(\bt(\psi_i)) \cap [Q_{g, n + k, \beta}^{\mathbf c}]^\vir\right),
\end{multline*}
where we use the shift
\begin{equation*}
  \bt = z (1 - I_0(q)) + I_1(q) H,
\end{equation*}
and where $I_0(q)$ and $I_1(q)$ are as in the introduction.
The choice is motivated by the Wall-Crossing Theorem \cite{CiKi16P}
(see also \cite{CJR17P1}).
The extended quintic classes are related to the quintic CohFT via
\begin{equation*}
  \Omega_{g, n}(\gamma_1, \dotsc, \gamma_n)
  = \Omega_{g, n}^{\mathbf c^\eff, \bt} (\gamma_1, \dotsc, \gamma_n)
\end{equation*}
under the mirror transformation between $q$ and $Q$.

Although the extended quintic classes do not form a CohFT, in this
section, we will express them in terms of a generalized $R$-matrix
action.
In the following sections, we will mostly work with the related shift
\begin{equation*}
  \tau = \frac{I_1(q)}{I_0(q)} H.
\end{equation*}

Before stating the generalized $R$-matrix action, we need to setup
some notation.
First, let $G_{g,n}^{\infty}$ be the set of stable graphs in
$G_{g, n}$ with the decorations:
\begin{itemize}
\item for each vertex $v$, we assign a label $0$ or $\infty$;
\item for each flag $f=(v,e)$ or $f=(v,l)$ where $v$ labeled by $\infty$, we assign a degree $\delta_f \in \ZZ_{\geq 1}$.
\end{itemize}
Then, if $v$ is a vertex of a graph $\Gamma \in G^\infty_{g, n}$, and
$\gamma \in \H$, we write
\begin{equation*}
  \sR_v^{-1}(z_f)\gamma :=
  \begin{cases}
    R^{-1}_\tau(z_f) \gamma & \text{if $v$ is labeled by $0$ }, \\
    S^{-1}_\tau(z_f) \, \gamma  & \text{if $v$ is labeled by $\infty$ }.
  \end{cases}
\end{equation*}

To define the generalized $R$-matrix action, define for each
$\Gamma \in G^\infty _{g,n}$ a contribution
$$
\Cont_\Gamma\colon  \H^{\otimes n} \rightarrow \M_\Gamma
$$
by the following construction:
\begin{enumerate}
\item at each flag $f=(v,l)$ with leg insertion $\gamma$, we place
  $\sR_v^{-1}(z_f)\gamma$;
\item at each edge $e$ connecting two vertices $v_1,v_2$, we place
  \begin{equation*}
    \mathscr V_{f_1,f_2}(z_{f_1} ,z_{f_2}):=\frac{\sum_\alpha  \delta_{v_1v_2} e_\alpha \otimes e^\alpha- \sR_{v_1}^{-1}(z_{f_1} )e_\alpha \otimes   \sR_{v_2}^{-1}(z_{f_2} )e^\alpha}{z_{f_1} +z_{f_2} }
  \end{equation*}
  where $f_1:=(v_1,e), f_2:=(v_2,e)$ and
  \begin{equation*}
    \delta_{v_1,v_2} =
    \begin{cases}
      1 & \text{ $v_1, v_2$ have both label $0$}, \\
      0  & \text{ otherwise };
    \end{cases}
  \end{equation*}
\item at each vertex $v$ with $m$ flags $f_1, \dotsc, f_m$, we place
  the map
  \begin{multline*}
    \gamma_1(z_{f_1})\otimes \cdots \otimes\gamma_m(z_{f_m}) \mapsto \\
    \begin{cases}
      T\omega^{\lambda}_{g(v), m} ( \gamma_1(\bar\psi_1)\otimes \cdots
      \otimes\gamma_m(\bar\psi_m) ) & \text{ $v$ is labeled by $0$ },\\
      (I_0)^{2g(v) - 2 + \ell(\mu)} J \Omega^{\infty, \mathbf c}_{g(v), \mu} ( \gamma_1(5H/\mu_1)\otimes \cdots
      \otimes\gamma_m(5H/\mu_m)) & \text{ $v$ is labeled by $\infty$ } .
    \end{cases}
  \end{multline*}
  where $\mu$ is the partition of length $m$ formed by $\delta_f$ for
  all flags $f$, where
  \begin{equation*}
    \Omega^{\infty, \mathbf c}_{g, \mu}\colon \H^{\otimes \ell(\mu)} \to H^*(\M_{g, \ell(\mu)})
  \end{equation*}
  is a multilinear form determined from the effective invariants such
  that
  \begin{equation*}
    \Omega^{\infty, \mathbf c}_{g, \mu} = 0 \quad \text{for} \quad {2g-2} + \ell(\mu) - 5\,|\mu| < 0,
  \end{equation*}
  and where
  \begin{equation*}
    J \Omega^{\infty, \mathbf c}_{g, \mu}(\gamma_1\otimes \cdots
    \otimes\gamma_m)
    = \sideset{}{'}\sum_{\mu'} \frac 1{|\Aut(\mu')|} \Omega^{\infty, \mathbf
      c}_{g,\mu + \mu'} \left(\gamma_1 \otimes \cdots
    \otimes\gamma_m \otimes \bigotimes_{i=1}^{\ell(\mu')}
    J_\bt(-5H/\mu'_i)\right),
  \end{equation*}
  in which $\sum'$ indicates that the sum is only over all partitions
  $\mu'$ with no part equal to one.
\end{enumerate}
\begin{remark}
  By Remark~\ref{rmk:invertible}, the substitution
  $z_i = \frac{5H}{\mu_i}$ is well-defined.
  Also, note that the sum in the definition of
  $J \Omega^{\infty, \mathbf c}$ has only finitely many nonzero terms.
\end{remark}
\begin{remark}
  The formal definition of $\Omega^{\infty, \mathbf c}$ is in
  Section~\ref{sec:rmatrix1}.
\end{remark}
\begin{theorem}
  \label{thm:Raction}
  \begin{equation*}
    \Omega_{g, n}^{\mathbf c, \bt}
    = \lim_{\lambda \to 0} \sum_{\Gamma\in G_{g,n}^\infty} \frac{1}{|\Aut(\Gamma)|}  \iota_{\Gamma*}(\Cont_\Gamma)
  \end{equation*}
\end{theorem}
\begin{definition}
  We define a $\lambda$-equivariant version of
  $\Omega_{g, n}^{\mathbf c, \bt}$ by
  \begin{equation*}
    \Omega_{g, n}^{\mathbf c, \bt, \lambda}
    = \sum_{\Gamma\in G_{g,n}^\infty} \frac{1}{|\Aut(\Gamma)|}  \iota_{\Gamma*}(\Cont_\Gamma)
  \end{equation*}
  so that
  $\Omega_{g, n}^{\mathbf c, \bt} = \lim\limits_{\lambda \to 0}
  \Omega_{g, n}^{\mathbf c, \bt, \lambda}$.
\end{definition}
The proof of this theorem will occupy Section~\ref{sec:rmatrix}.

\subsection{Proof of Theorem~\ref{thm:Raction}}
\label{sec:rmatrix}

\subsubsection{Step 1}
\label{sec:rmatrix1}

For elements $\gamma_1, \dotsc, \gamma_n \in \H$, we apply
\eqref{eq:quinticlambda} to compute
$\Omega_{g, n}^{\mathbf c, \bt}(\gamma_1, \dotsc, \gamma_n)$.
We note that the bivalent graphs $\Gamma$ that
\eqref{eq:quinticlambda} sums over may not be stable graphs in
$G_{g, n}$, but may include unstable vertices of genus zero, with at
most two markings.
The conditions from Section~\ref{sec:loc:formula} imply that any such
unstable vertex $v$ satisfies $v \in V_0(\Gamma)$.

Contracting the unstable vertices produces a stable graph $\Gamma'$
from $\Gamma$.
We label each vertex in $\Gamma'$ according to whether the
corresponding vertex in $\Gamma$ lies in $V_0(\Gamma)$ or
$V_\infty(\Gamma)$.
Note that the flags $f$ of $\Gamma'$ at vertices labeled by $\infty$
are in bijective correspondence to the edges of $\Gamma$ that are not
connected to a genus zero vertex of valence one.
Using the fiber class mapping
$\delta\colon E(\Gamma) \to \ZZ_{\ge 1}$, we define the degree
$\delta_f$.
In this way, we assign to each decorated bivalent graph, a decorated
dual graph in $G_{g, n}^\infty$ with the special property that there
is no edge connecting two vertices of label $0$.
We note that $S^{-1}$ is the generating series of unstable vertices of
genus zero with one edge and one marking, $J$ is the generating series
of unstable vertices with a single edge, and
\begin{multline*}
  V_\bt(z_1, z_2)
  := \frac{\sum_\alpha e^\alpha \otimes e^\alpha}{-z_1 - z_2} \\
  + \sum_{\beta = 0}^\infty \sum_{\alpha_1, \alpha_2}
  \frac{q^\beta}{k!} \int_{[\M_{0, 2}(\PP^4, \beta)]^{\vir, \lambda}}
  \left(\frac{\ev_1^*(e_{\alpha_1})}{-z_1 - \psi_1}
    \frac{\ev_2^*(e_{\alpha_2})}{-z_2 - \psi_2} \prod_{i = 3}^{k + 2} \ev_i^*(\bt(\psi_i)) \right) e^{\alpha_1}
  \otimes e^{\alpha_2}
\end{multline*}
is the generating series of unstable vertices with two edges.

Reorganizing \eqref{eq:quinticlambda} according to the decorated dual
graph, we see that
\begin{equation*}
  \Omega_{g, n}^{\mathbf c, \bt} (\gamma_1, \dotsc, \gamma_n)
  = \lim_{\lambda \to 0} \sum_{\Gamma \in G_{g, n}^\infty} \frac 1{|\Aut(\Gamma)|} \iota_{\Gamma*}(\Cont_\Gamma^1),
\end{equation*}
where the contribution $\Cont_\Gamma^1$ is defined by the following:
\begin{enumerate}
\item at each flag $f = (v, l)$ with insertion $\gamma$, we place
  $\soR_v^{-1}(z_f)\gamma$ where
  \begin{equation*}
    \soR_v^{-1}(z)\gamma :=
    \begin{cases}
      \gamma & \text{if }v\text{ is labeled by }0, \\
      S^{-1}_\bt(z) \gamma & \text{if }v\text{ is labeled by }\infty
    \end{cases}
  \end{equation*}
\item at each edge $e$ connecting two vertices $v_1$ and $v_2$, we
  place
  \begin{equation*}
    \begin{cases}
      0, & \text{if $v_1$ and $v_2$ are labelled by $0$}, \\
      \sum_\alpha\frac{-e_\alpha \otimes e^\alpha}{z_{f_1} + z_{f_2}}, & \text{if exactly one of $v_1$ and $v_2$ is labelled by $0$}, \\
      V_\bt(z_{f_1}, z_{f_2}), & \text{if $v_1$ and $v_2$ are labelled by $\infty$}
    \end{cases}
  \end{equation*}
  where $f_1:=(v_1,e), f_2:=(v_2,e)$, and
\item at each vertex $v$ with $m$ flags $f_1, \dotsc, f_m$, we place
  the map
  \begin{equation*}
    \gamma_1(z_{f_1}) \otimes \dotsb \otimes \gamma_m(z_{f_m})
    \mapsto
    \begin{cases}
      \Omega_{g(v), m}^{\lambda, \bt}(\gamma_1(\psi_1), \dotsc, \gamma_m(\psi_m)) & \text{if }v\text{ is labeled by }0 \\
      J\Omega_{g(v), \mu}^\infty (\gamma_1(z_1), \dotsc, \gamma_m(z_m)|_{z_i=5H/\mu_i}) & \text{if }v\text{ is labeled by }\infty
    \end{cases},
  \end{equation*}
  where we set
  \begin{multline*}
    \Omega_{g, n}^{\lambda, \bt}(\gamma_1(\psi_1), \dotsc, \gamma_n(\psi_n)) \\
    = (I_0)^{2g - 2 + n} \sum_{\beta, k = 0}^\infty \frac{q^\beta}{k!} \rho_*\pi_*^k\left(\prod_{i = 1}^n \ev_i^*(\gamma_i(\psi_i)) \prod_{i = n + 1}^{n + k} \ev_i^*(\bt(\psi_i)) \cap [\M_{g, n + k}(\PP^4, \beta)]^{\vir, \lambda}\right),
  \end{multline*}
  where
  \begin{equation*}
    \Omega_{g, \mu}^\infty(\gamma_1, \dotsc, \gamma_m)
    = \sum_{\beta \ge 0} (-1)^{1 - g + 5\beta} q^\beta \left(C_{g, \beta, \mu} \prod_{e = 1}^{\ell(\mu)} \frac 1{\prod_{i = 2}^{\mu_e} \frac{-5i H}{\mu_e}}, \gamma_1 \cdot \dotsb \cdot \gamma_m\right)^\lambda,
  \end{equation*}
  and where
  \begin{multline*}
    J \Omega^{\infty, \mathbf c}_{g, \mu}(\gamma_1\otimes \cdots
    \otimes\gamma_m) \\
    = \sum_{\mu'} \frac 1{|\Aut(\mu')|} \Omega^{\infty, \mathbf
      c}_{g,\mu + \mu'} (\gamma_1 \otimes \cdots
    \otimes\gamma_m \otimes \prod_{i=1}^{\ell(\mu')}
    (\delta_{1\mu'_{m + i}} z_{m + i} + J_\bt(-z_{m + i}))|_{z_{m + i} = 5H/\mu'_i}).
  \end{multline*}
  Here, we set $C_{g, \beta, \mu} = 0$, when
  $2g - 2 + \ell(\mu) - 5\beta \neq |\mu|$.
  Note that $\Omega_{g, \mu}^\infty$ has absorbed a factor of $t - 5H$
  from the unstable vertex contributions and the factor relating the
  ordinary and twisted pairing of $\PP^4$.
\end{enumerate}

\subsubsection{Step 2}

We simplify the expression from Step 1 slightly, using the following
three facts from Gromov--Witten theory.

First, the dilaton equation implies (see
\cite[Corollary~1.5]{CiKi16P})
\begin{equation*}
  \Omega_{g, n}^{\lambda, \bt}(\gamma_1, \dotsc, \gamma_n)
  = I_0^{-(2g - 2 + n)} \Omega_{g, n}^\lambda(\gamma_1, \dotsc, \gamma_n),
\end{equation*}
\begin{equation*}
  S^{-1}_\bt(z) = S^{-1}_\tau(z).
\end{equation*}

Second, there is the identity
\begin{equation*}
  V_\bt(z_1, z_2)
  = \frac{-\sum_\alpha S^{-1}_\bt(z_1) e_\alpha \otimes S^{-1}_\bt(z_2) e^\alpha}{z_1 + z_2}
  = \frac{-\sum_\alpha S^{-1}_\tau(z_1) e_\alpha \otimes S^{-1}_\tau(z_2) e^\alpha}{z_1 + z_2}.
\end{equation*}
Third, using the ancestor-descendent comparison, we may replace the
insertion
\begin{equation*}
  \frac\gamma{-z - \psi}
\end{equation*}
involving a descendent psi-class $\psi$ by the insertion
\begin{equation*}
  \frac{S^{-1}_\bt(z)\gamma}{-z - \bar\psi} = \frac{S^{-1}_\tau(z)\gamma}{-z - \bar\psi}
\end{equation*}
involving the corresponding ancestor psi-class $\bar\psi$.

We therefore may redefine $\Cont_\Gamma^1$ as follows:
\begin{enumerate}
\item at each flag $f = (v, l)$ with insertion $\gamma$, we place
  $\soR_v^{-1}(z_f)\gamma$.
\item at each edge $e$ connecting two vertices $v_1$ and $v_2$, we
  place
  \begin{equation*}
    \frac{-\soR_{v_1}^{-1}(z_{f_1}) e_\alpha \otimes \soR_{v_2}^{-1}(z_{f_2})e^\alpha}{z_{f_1} + z_{f_2}}
  \end{equation*}
\item at each vertex $v$ with $m$ flags $f_1, \dotsc, f_m$, we place
  the map
  \begin{multline*}
    \gamma_1(z_{f_1}) \otimes \dotsb \otimes \gamma_m(z_{f_m})
    \mapsto \\
    \begin{cases}
      \Omega_{g(v), m}^\lambda (\gamma_1(\bar\psi_1), \dotsc, \gamma_m(\bar\psi_m)) & \text{if }v\text{ is labeled by }0 \\
      I_0^{2g(v) - 2 + |\ell(\mu)|} J\Omega_{g(v), \mu}^\infty (\gamma_1(z_1), \dotsc, \gamma_m(z_m)) & \text{if }v\text{ is labeled by }\infty
    \end{cases}.
  \end{multline*}
\end{enumerate}

\subsubsection{Step 3}

We finally apply the Givental--Teleman classification to the
computation of $\Omega_{g(v), m}^\lambda$ at every vertex $v$ of
$\Gamma$ which has label $0$.
For every such vertex $v$, we will have a sum over dual graphs
$\Gamma_v$.
Replacing $v$ in $\Gamma$ by $\Gamma_v$ for all $v$ results in a new
dual graph $\Gamma'$.
We can extend the decoration of $\Gamma$ to the one of $\Gamma'$ by
requiring that all vertices in $\Gamma_v$ should be labeled $0$.
Therefore, $\Gamma'$ is an element of $G_{g, n}^\infty$ with no
further condition on the labeling.

We can therefore write
\begin{equation*}
  \Omega_{g, n}^{\mathbf c, \bt} (\gamma_1, \dotsc, \gamma_n)
  = \lim_{\lambda \to 0} \sum_{\Gamma' \in G_{g, n}^\infty} \frac 1{|\Aut(\Gamma')|} \iota_\Gamma(\Cont_{\Gamma'}^2),
\end{equation*}
where $\Cont_{\Gamma'}^2$ is defined in the following way:
\begin{enumerate}
\item at each flag $f = (v, l)$ with insertion $\gamma$, we place
  $\sR_v^{-1}(z_f)\gamma$.
\item at each edge $e$ connecting two vertices $v_1$ and $v_2$, we
  place
  \begin{equation*}
    \frac{\sum_\alpha \delta_{v_1 v_2} e_\alpha \otimes e^\alpha - \sR_{v_1}^{-1}(z_{f_1}) e_\alpha \otimes \sR_{v_2}^{-1}(z_{f_2})e^\alpha}{z_{f_1} + z_{f_2}}
  \end{equation*}
\item at each vertex $v$ with $m$ flags $f_1, \dotsc, f_m$, we place
  the map
  \begin{multline*}
    \gamma_1(z_{f_1}) \otimes \dotsb \otimes \gamma_m(z_{f_m})
    \mapsto \\
    \begin{cases}
      T\Omega_{g(v), m}^\lambda (\gamma_1(\bar\psi_1), \dotsc, \gamma_m(\bar\psi_m)) & \text{if }v\text{ is labeled by }0 \\
      I_0^{2g(v) - 2 + \ell(\mu)} J\Omega_{g(v), \mu}^{\infty, \mathbf c} (\gamma_1(z_1), \dotsc, \gamma_m(z_m)) & \text{if }v\text{ is labeled by }\infty
    \end{cases}.
  \end{multline*}
\end{enumerate}
This is almost the same as the definition of the generalized
$R$-matrix action.
The only difference lies in the definition of $J\Omega^{\infty, \mathbf c}$
To complete the proof of Theorem~\ref{thm:Raction}, it suffices to
prove the following Lemma.
\begin{lemma}
  \label{lem:mirror}
  Setting $H_\mu = -5H/\mu$, we have
  \begin{equation*}
    J_\bt(H_\mu) = H_\mu
  \end{equation*}
  for $\mu \le 5$.
  More generally, for any positive integer $\mu$, the power series
  $J_\tau(H_\mu)$ is a polynomial of degree
  $\lfloor(\mu - 1)/5\rfloor$.
\end{lemma}
\begin{proof}
  By the genus zero mirror theorem
  \begin{equation*}
    J_\bt(z)
    = z\sum_{\beta = 0}^\infty q^\beta \frac{\prod_{i = 1}^{5\beta} (5H + iz)}{\prod_{i = 1}^\beta ((H + iz)^5 - \lambda^5)}.
  \end{equation*}
  Therefore, we have
  \begin{equation*}
    J_\bt(H_\delta)
    = H_\delta \sum_{\beta = 0}^{\lfloor(\mu-1)/5\rfloor} q^\beta \frac{\prod_{i = 1}^{5\beta} (5\mu - 5i)}{\prod_{i = 1}^\beta ((\mu - 5i)^5 - \mu^5)}.
  \end{equation*}
\end{proof}
\begin{corollary}
  For any partition $\mu$ of size $m$ and any
  $\gamma_1, \dotsc, \gamma_m$, the class
  \begin{equation*}
    J\Omega_{g, \mu}^{\infty, \mathbf c}(\gamma_1, \dotsc, \gamma_m)
  \end{equation*}
  is a polynomial of $q$ of degree
  $\lfloor (2g - 2 + \ell(\mu) - |\mu|)/5\rfloor$, and of codimension
  \begin{equation*}
    3g - 3 + \sum_i \deg(\gamma_i).
  \end{equation*}
\end{corollary}
\begin{proof}
  Under the mirror map, we can simplify
  \begin{equation*}
    J\Omega_{g, \mu}^{\infty, \mathbf c} (\gamma_1, \dotsc, \gamma_m)
    = (C_{g, \mu}(q), \gamma_1 \cdot \dotsb \cdot \gamma_m)^\lambda,
  \end{equation*}
  where
  \begin{multline*}
    C_{g, \mu}(q)
    = \sum_{\beta \ge 0} q^\beta (-1)^{1 - g + 5\beta} \prod_{e = 1}^{\ell(\mu)} \frac 1{\prod_{i = 2}^{\mu_e} i H_{\mu_e}} \\
    \sideset{}{'}\sum_{\mu'} \frac 1{|\Aut(\mu')|} \prod_{e = 1}^{\ell(\mu')} \frac 1{\prod_{i = 2}^{\mu'_e} i H_{\mu'_e}} C_{g, \beta, \mu + \mu'} \prod_{i = 1}^{\ell(\mu')} J_\bt(H_{\mu_i'}).
  \end{multline*}

  Recall that $C_{g, \beta, \mu + \mu'}$ vanishes unless
  \begin{equation*}
    5\beta = 2g - 2 + \ell(\mu) + \ell(\mu') - |\mu| - |\mu'|.
  \end{equation*}
  Therefore, the summands vanishes unless
  \begin{equation*}
    \beta
    \le (2g - 2 + \ell(\mu) + \ell(\mu') - |\mu| - |\mu'|)/5
    \le (2g - 2 + \ell(\mu) - |\mu|)/5.
  \end{equation*}
\end{proof}

\section{Graded finite generation and orbifold regularity}
\label{sec:fgen}

As we mentioned in the introduction, a fundamental structural
prediction is the finite generation property of $F_g$ of
Yamaguchi--Yau \cite{YaYa04}, which says that $F_g$ is a polynomial of
five generators.
This can be thought as a higher dimensional generalization of the
statement that $F_g$ is a quasi-modular form for an elliptic curve.
Because of this, the finite generation property is often loosely
referred to as the ``modularity'' of the Gromov--Witten theory of
Calabi--Yau manifolds.
Since we only have to compute finitely many coefficients of a
polynomial, this key property immediately reduces the infinite number
of degreewise computations to a finite computation!
By introducing a grading, we prove a stronger version of finite
generation which implies orbifold regularity.
The main theorem of this section is the following:

\begin{theorem}
  \label{thm:fgen2}
The following ``finite generation properties" hold for the extended quintic family
\begin{description}
\item[(1)] $\bar \Omega^{\mathbf c, \lambda}_{g,n}(\phi_{a_1}, \dotsc, \phi_{a_n}) \in   \sbigoplus_{k\geq 0} \lambda^{5k}  \Big(H^{3g-3+\sum_i a_i-5k}(\M_{g,n}, \QQ)\otimes \bR_{3g-3+\sum_i a_i}\Big)$
\item[(2)] $\bar{F}_{g,n}= 5^{g-1}(L/I_0)^{2g-2} I_{1,1}^n  \big(Q\frac{d}{dQ} \big)^nF_{g}\in \bR_{3g-3+n}$
\end{description}
\end{theorem}
In the non-equivariant limit $\lambda \to 0$, we recover
Theorem~\ref{thm:fgen}.

\subsection{CohFT of the $\lambda$-twisted invariants}

Recall the $\lambda$-twisted invariants defined in
Section~\ref{sec:twisted-CohFT} for the mirror shift
\begin{equation*}
  \tau = \frac{I_1(q)}{I_0(q)} H.
\end{equation*}
Note that the $\lambda$-twisted theory is semi-simple.
We will work with five bases of the state space $H^*(\PP^4)$:
\begin{enumerate}
\item The flat basis $\{H^i\}_{i = 0, \dotsc, 4}$
\item The normalized flat basis $\{\phi_i\}_{i = 0, \dotsc, 4}$ (see the introduction)
\item An alternative normalized flat basis $\{\bar\phi_i\}_{i = 0, \dotsc, 4}$, where
  \begin{equation*}
    \bar \phi_k =  \frac{ I_{0} }{ L} \phi_k = \frac{I_0 I_{1, 1} \dotsb I_{k, k}}{L^{k+1}} H^k
  \end{equation*}
\item The canonical basis $\{ e_\alpha \}$ (of idempotents)
\item The normalized canonical basis $\{ \bar e_\alpha \}$, where
  \begin{equation*}
    \bar e_\alpha:= \Delta_\alpha^{1/2}  e_\alpha,
  \end{equation*}
  and
  \begin{equation*}
    \Delta_\alpha:= (e_\alpha,  e_\alpha)^{-1} = (\zeta^\alpha \lambda)^3\frac{I_0^2}{L^2}.
  \end{equation*}
\end{enumerate}

We now collect basic formulae for the $\lambda$-twisted invariants
(see \cite[Section~6]{GuRo17P}, and also \cite[Section~7]{GuRo16P}):
\begin{itemize}
\item Canonical basis and normalized canonical basis
\begin{align*}
  e_\alpha  =\frac{1}{5 }\sum_i (\zeta^{\alpha} \lambda)^{-i}  \phi_i, \qquad
  \bar e_\alpha  =  \frac{1}{ {5}}\sum_i (\zeta^\alpha  \lambda)^{-i+\frac{3}{2}}  \bar \phi_i .
\end{align*}
\item Canonical coordinates
  \begin{equation}
    \label{eq:cancoords}
    u_\alpha =\zeta^\alpha \lambda \int_0^q (L(q)-1) \frac{dq}{q}
  \end{equation}
\item The pairing $(,)^\lambda$ in the bases $\{\phi_i\}$ and
  $\{\bar\phi_i\}$ is given by the matrices
  \begin{equation*}
 \frac{5 L^2}{I_0^2}    \begin{pmatrix}
      & & & 1& \\
      & & 1 & & \\
      & 1& & & \\
   1 & & & & \\
      & & & &  \lambda^5  
    \end{pmatrix}, \qquad
    \begin{pmatrix}
      & & & 5 & \\
      & & 5 & & \\
      & 5 & & & \\
      5 & & & & \\
      & & & & 5 \lambda^5
    \end{pmatrix},
  \end{equation*}
  respectively.
\item The matrix of quantum multiplication by
  $\dot\tau =H+q\frac{d}{dq} (\frac{I_1}{I_0})$ in the flat basis
  $\{H^i\}_{i=0}^4$ is
\begin{equation} \label{quantumproduct}
   \dot\tau * = A: =\begin{pmatrix}
  & I_{1,1} \\
 &    & I_{2,2} \\
    &  &  & I_{3,3} \\
      &  &   &  & I_{4,4} \\
    I_{5,5}\lambda^5  &   &   &   &
   \end{pmatrix}^T.
\end{equation}
\item Topological field theory {\small
\begin{align*}
 \qquad \quad
\omega_{g,n}(e_{\alpha_1},\cdots,e_{\alpha_n}) = &\ \sum_\alpha \Delta_\alpha^{g-1} \prod_i \delta_{\alpha, \alpha_i}  ,\\
\omega_{g,n}(\phi_{a_1},\cdots,    \phi_{a_n}) = &\ \sum_\alpha \zeta^{\alpha \sum_i a_i}\, \lambda^{3g-3+\sum_i a_i} (I_0/L)^{2g-2} .
\end{align*}}
\end{itemize}

We will also consider the following dual bases:
\begin{enumerate}
\item The dual flat basis $\{\frac 15 H^{3 - i}\}_{i = 0, \dotsc, 4}$,
  where we note that
  \begin{equation*}
    H^{-1} = H^4 H^{-5} = H^4 \lambda^{-5}
  \end{equation*}
\item The alternative dual normalized flat basis
  $\{\phi^i\}_{i = 0, \dotsc, 4}$, where
  \begin{equation*}
    \bar \phi^i = \frac 15 \bar \phi_{3 - i},
  \end{equation*}
  and $\bar \phi_{-1} = \lambda^{-5} \bar \phi_4$.
\end{enumerate}

Note that if $k = a + 5b \in \ZZ$ for $a \in \{0, 1, 2, 3, 4\}$, we
have $H^k = \lambda^b H^a$.
We analogously define $\phi_k = \lambda^b \phi_a$ and
$\bar \phi_k = \lambda^b \bar\phi_a$.
This is consistent with the above definition of $\bar\phi_{-1}$.

\subsection{Differential equations for $S$-matrix and $R$-matrix}
\label{sec:QDERS}

Recall the $S$-matrix is the solution of the quantum differential
equation (QDE) \cite{Du94, LePa04P}
\begin{equation*}
  z \,d S_\bt(z)  = d\bt *_{\bt}  S_\bt(z).
\end{equation*}
By using the divisor equation, the QDE at $\bt=\tau:=I_1/I_0$ is
equivalent to the following
\begin{equation}\label{QDE}
  zD  S_\tau(z) +    S_\tau(z) \cdot H= A  \cdot S_\tau(z),
\end{equation}
where $A = \dot \tau *$ is the quantum product matrix
\eqref{quantumproduct}.
We will omit the subscript $\tau$ in $S_\tau$ in the rest of the
paper.

\begin{lemma}
We consider the entries of the $S$-matrix and $R$-matrix 
$$
 {S}_{i\bar \alpha}(z):= (\bar e_\alpha,  S^*(z) H^i)^\lambda,\quad
 {R}_{i \bar \alpha}(z):= (\bar e_\alpha,  R^*(z) H^i)^\lambda .
$$
The entries satisfy the following quantum differential equation
\begin{align}
\big(z\,q\frac{d}{dq}+ h_\alpha \big)   {S}_{i\bar \alpha}(z) = &\, \sum_j {A_i}^j \cdot  {S_j}_{\bar \alpha} (z)   \label{QDEforS}
\\
\big(z\,q\frac{d}{dq}+ L_\alpha \big)\,  {R}_{i\bar \alpha}(z) = &\,  \sum_j {A_i}^j \cdot {R_j}_{\bar \alpha}(z)   \label{QDEforR}
\end{align}
where $h_\alpha = \zeta^\alpha \lambda$,
$L_\alpha = \zeta^\alpha \lambda L$ and
${A_i}^{j} = \frac{1}{5} (A\, H^i, H^{3-j})^\lambda$ are the entries
of the matrix $A$ \eqref{quantumproduct}.
\end{lemma}
\begin{proof}
The equation for the $S_{i\bar\alpha}$ is just the component form of the equation \eqref{QDE}.

  By the following Birkhoff factorization
  $$
  {S}_{i\bar \alpha}(z) C_\alpha(z)^{-1} = e^{u_\alpha /z} {R}_{i \bar \alpha}(z)
  $$
  and \eqref{eq:cancoords}, the equation for the $R$-matrix
  follows.
  Here by the results of \cite{Gi01a, CoGi07}
  \begin{equation} \label{constantR} 
    C_\alpha(z) = \,  e^{ \sum_{k>0,\beta\neq \alpha} \frac{B_{2k}}{2k({2k-1})}\frac{z^{2k-1}}{(\zeta^\alpha \lambda - \zeta^\beta\lambda)^{2k-1}} } \cdot e^{ \sum_{k>0} \frac{B_{2k}}{2k({2k-1})}\frac{z^{2k-1}}{(-5 h_\alpha)^{2k-1}} .}
  \end{equation}
\end{proof}

We denote by $\bar R^*$ the matrix representation of $R^*$ under the
normalized flat basis $\{\bar \phi_i\}$, in other words
$$
\bar {R}_{  i} ^{ j}(z):= (\bar\phi^j,  R^*(z) \bar\phi_i)^{\lambda},
$$
and we also define a $\bar S^*(z)$ matrix by
$$
\bar {S}_{  i} ^{ j}(z):={\textstyle \frac{1}{5} }(H^{3-j},  S^*(z) \bar\phi_i)^{\lambda}  ,\qquad
(\bar {S}_\delta)_{ i} ^{ j}:=L^\delta \cdot {\textstyle \frac{1}{5} }   (H^{3-j} ,  S^*(H_\delta) \bar\phi_i) ^{\lambda},
$$
where $H_\delta = - 5H/\delta$.

\begin{lemma}
  \label{cor:QDEforR}
  We have the following properties
\begin{align}
(a)&\qquad    \big(z D_{\C} +\Amone \big) \, \bar R^*(z)   =  \bar R^*(z) \Amone   \label{QDEnormalize}\\
 &\qquad \big(H_\delta D_{\C} + 5H \Z \big) \,  \bar {S}_\delta   =  \bar S_\delta \Amone   \label{QDESnormalize}
\\
(b)&\qquad [z^k] {\bar R_{ i }}^{j}(z) = 0 \quad \text{ if  } \quad  k-i+j \notin 5 \ZZ\\
&\qquad \qquad \!\! (\bar S_\delta)_{i }^{j} = 0 \quad \text{ if  } \quad  i\neq j  \label{vanishingcondS}
\end{align}
where $\Amone$ is defined by $\Amone \cdot \bar\phi_i:=\bar\phi_{i+1}$,
 $[z^k]f(z)$ is the coefficient of $z^k$ in a formal series $f(z)$ of $z$, and for any $M$
$$
D_{\C} M :=\partial_u M -M \cdot \C,\qquad \C=
\diag( \X,\Y,-\Y,-\X,0)
$$

\end{lemma}
\begin{proof}
  Note that we have
  \begin{equation*}
    \frac{1}{5} H^{3-j} = \frac 15 \sum_\alpha (\zeta^\alpha \lambda)^{\frac 32 - j}   \bar e_\alpha|_{q = 0},\qquad
    \bar \phi^j =  \frac{1}{5} \sum_\alpha (\zeta^\alpha \lambda)^{\frac 32 - j} \bar e_\alpha,\qquad
  \end{equation*}

  Letting $ {S }_i^j(z):=\frac{1}{5}(H^{3-j}, S^*(z) H^i)^\lambda$ and
  $ {R }_{ i}^{ \bar j}(z):=(\bar\phi^j, R^*(z) H^i)^\lambda$, we have
  \begin{align*}
    z\, \frac{1}{L} q\frac{d}{dq}  {S }_{ i}^{ j}(z)+  \frac{1}{L}  {S}_{i}^{ {j-1}}(z)   =& \    \frac{ I_{i+1\,,i+1} }{L}   {S }_{i+1}^{   j}(z), \\
    z\, \frac{1}{L} q\frac{d}{dq}  {R }_{ i}^{ \bar j}(z)+    {R}_{i}^{ \overline {j-1}}(z)   =& \    \frac{ I_{i+1\,,i+1} }{L}   {R }_{i+1}^{ \overline j}(z).
  \end{align*}
  By changing the basis on the right hand side to the alternative normalized flat
  basis as well, we arrive at
  \begin{align*}
    & z\big(\partial_u-  \partial_u \log {\textstyle \frac{I_{0}\cdots I_{i,i}}{   L^{i+1 }} } \big) \,{\bar S_{{i} }}^{  j}(z)+  L^{-1} {\bar S_{{i}}}^{  {j-1}}(z)   =   {\bar S_{ {i+1}}}^{ j}(z)
    \\
    & z\big(\partial_u- \partial_u \log {\textstyle \frac{I_{0}\cdots I_{i,i}}{  L^{i+1 }} } \big) \,{\bar R_{  {i} }}^{ j}(z)+   {\bar R_{  {i}}}^{  {j-1}}(z)   =   {\bar R_{ {i+1}}}^{  j}(z)
  \end{align*}
  The second equality gives \eqref{QDEnormalize}. 
  The first equality gives
  $$
  z D_{\C} \, \bar S^*(z)+   H \cdot L^{-1} \bar S^*(z)   =  \bar S^*(z) \Amone  .
  $$
  By setting $z = H_\delta = -5H/\delta$ and by using
  $\partial_u \log L +({5 L})^{-1}= \Z$ we obtain
  \eqref{QDESnormalize}.

  For (b), the following $i=0$ case will be proved in
  Theorem~\ref{lem:R0}
  $$
  [z^k] {\bar R_0}^j(z) = 0 \quad \text{ if  } \quad  k+j \notin 5 \ZZ .
  $$
  Then by using Part (a) we can deduce (b) for any $i>0$ inductively.

  We apply Lemma~\ref{lem:mirror} to compute the entries
  $(S_\delta)_0^j$.
  We note that by the genus zero mirror theorem
  \begin{equation}
    \label{eq:JS}
    J_\bt(H_\delta)
    = H_\delta \frac{S^*(H_\delta) \mathbf 1}{I_0}
    = H_\delta \frac{S^*(H_\delta) \bar\phi_0}L.
  \end{equation}
  So by Lemma~\ref{lem:mirror}, $(\bar S_\delta)_0^j = 0$, unless
  $j = 0$.
  As for $\bar R$, we can now deduce Part (b) from Part (a)
  inductively.
\end{proof}

\subsection{Finite generation for $S$-matrix and $R$-matrix}

As we will see soon, Lemma~\ref{cor:QDEforR} gives us a canonical way
to write each entry of the $R$-matrix as an element in the ring of six
generators.
We first note the following:
\begin{lemma}
  \label{lem:du}
  There is a derivation $\partial_u$ on the ring $\bR$ of six
  generators compatible with the derivation $\frac qL \frac d{dq}$.
  Explicitly:
  \begin{align*}
    \partial_u \X_1 = & \ \X_2, \qquad
    \partial_u \X_2 = \X_3, \qquad
    \partial_u L^{-a} Z^b = \frac 15 L^{-a} Z^b (5b - a) L^{-1} (Z - 1), \\
    \partial_u \Y_1=&\, -3 \X_2-\Y^2-\X^2-\frac{3}{5}  L^{-2} Z(Z-1)\\
    \partial_u \X_3 =&\,  -4  \X  \X_3-3  \X_2^2-6  \X^2  \X_2- \X^4- \frac{3}{5}L^{-2} Z(Z-1)\big( \X_1^2+\X_2\big) \\&\quad- \frac{3}{25}L^{-3} Z(Z-1)(8Z-3) \X_1+
 \frac{1}{625} L^{-4} \cdot  Z \big( -396 Z^3+714 Z^2 -341 Z+23 \big)
  \end{align*}
\end{lemma}
\begin{lemma}
  \label{lem:fgforR}
  Let $\bar R_k:=[z^k] \bar R(z)$. We have $R_0=\mathrm{Id}$.
  For each $k>0$, all entries of the $\bar R_k$-matrix
  can be canonically expressed
  as an element of the degree $k$ subspace $\bR_k$ of the ring $\bR$:
  $$
  [z^k] {\bar R_{ i }}^{ j}(z)  \in  \lambda^{i - j - k} \bR_k
  $$
\end{lemma}
\begin{proof}
  The statement for $[z^k] {\bar R_0}^{ j}(z)$ follows from
  Lemma~\ref{lem:R0} (1) and Lemma \ref{degreeboundrk}, where we
  regard $Z=L^5$ such that
  $$
  [z^k] {\bar R_0}^{ j}(z) \in  \lambda^{-j -k} L^{-k}  \mathbb Q[Z]_{k/5 \leq d\leq k}  \subset \bR_k
  $$
  The general case of the lemma follows then from
  Lemma~\ref{cor:QDEforR} inductively, where we use Lemma~\ref{lem:du}
  to canonically define $\partial_u$.
\end{proof}

We move on to proving finite generation results for the specialized
$S$-matrices.
\begin{lemma}
  Following \cite{ZaZi08}, let
  \begin{equation} \label{tIfunction}
    \tilde I(q,z)
    = z\sum_{d\geq 0} q^d \frac{\prod^{5d-1}_{k=0}(5H+kz)}{\prod^d_{k=1}\big( (H+kz)^5 -\lambda^5 \big)}.
  \end{equation}
  Then, we have
  \begin{equation} \label{S4}
    H^4 \cdot \tilde I(z) =  I_0 \cdot S^*(z) \bar \phi_4,\qquad  I(z) = \big(H + zq\frac{d}{dq}\big) \tilde I(z).
  \end{equation}
\end{lemma}
\begin{proof}
  This follows from \cite[Section~2]{ZaZi08}.
\end{proof}

\begin{definition}
For any $A\subset   {\QQ}[L^{-1},L,\X_1,\X_2, \X_3, \Y]$, 
we introduce
\begin{align*}  
A^{\text{reg}} :=&\  A \cap  {\QQ}[L,\X_1,\X_2, \X_3, \Y] .
\end{align*}
We also introduce
$$
\overline\bR:=\bR[L^{-1}],\qquad  \overline\bR_k:=\{ f\in \bR[L^{-1}] \subset \widetilde \bR \ : \  \deg f = k \}
$$
\end{definition}
\begin{remark}
  Clearly,
  $$
  \bR = \overline\bR \cap  \widetilde \bR^{\text{reg}}.
  $$
\end{remark}
\begin{lemma}
  Recall $ H_\delta := -\frac{5 H }{\delta}$ and let (see \eqref{vanishingcondS})
$$
S_{\delta;j}:=  (S_{\delta})_j^j= L^{\delta} \cdot  \frac{I_0\cdots I_{j,j}}{5 L^{j + 1}} (H^{3-j}, S^*(H_\delta)  H^j)^\lambda,
$$
then by regarding $Z=L^5$ we have
for $i=0,1,2,3,4$,  
 \begin{equation*}
   S_{\delta;i}  \  \in \  \big( L^{\overline{\delta-1}}   \, \overline\bR_i \big)^{\text{reg}} \otimes \mathbb Q[Z]_{\leq \lfloor \frac{\delta-1}{5}  \rfloor}
 \end{equation*}
 where $\overline i = i - 5\lfloor i/5\rfloor$ denotes the remainder of
 $i$ under division by $5$, and $\QQ[Z]_{\leq d}$ is the set of degree
 $\leq d$ polynomials in $Z$.  
\end{lemma}
\begin{proof}
  Indeed we have the following stronger result:
  \begin{align*}
    S_{\delta;4} \ \in &\  L^{\bar\delta}\cdot \QQ[Z]_{\leq \lfloor\frac{\delta}{5} \rfloor}  ,\qquad \qquad 
                         S_{\delta;0} \ \in \ L^{\overline {\delta-1}}\cdot  \QQ[Z]_{\leq \lfloor\frac{\delta-1}{5} \rfloor}  , \\
    S_{\delta;1} \  \in &\     \big( L^{\overline{\delta-1}} \QQ [L^{-1},\X,\Z_1 ]_1\big)^{\text{reg}}  \otimes \QQ[Z]_{\leq \lfloor\frac{\delta-1}{5} \rfloor}   ,\\
    S_{\delta;2}\ \in &\   \big( L^{\overline{\delta-1}}  \QQ [L^{-1},\X,\X_2,\Y,\Z_1,\Z_2]_2\big)^{\text{reg}}  \otimes \QQ[Z]_{\leq \lfloor\frac{\delta-1}{5} \rfloor}    , \\
    S_{\delta;3} \  \in  &\   \big( L^{\overline{\delta-1}} \QQ [L^{-1},\X,\X_2,\X_3,\Z_1,\Z_2,\Z_3]_3\big)^{\text{reg}}  \otimes \QQ[Z]_{\leq \lfloor\frac{\delta-1}{5} \rfloor}   
  \end{align*}
  The statement for $S_{\delta;4} $ and $S_{\delta;0} $ follows from
  the oscillatory integral analysis in
  Section~\ref{sec:orbifoldregularity}, and from \eqref{S4} and
  \eqref{eq:JS}, respectively.
  The other properties follow from a direct computation via
  \eqref{QDESnormalize}.
  
  The fact that the function $S_{\delta; 3}$ does not involve $\Y$ is
  not important here.
  It will follow from the first equation in Example~\ref{ex:PDEforR}.
  (Note although stated for $R$ there, the PDE holds also for
  $S_\delta$.)
\end{proof}

\subsection{Proof of Theorem \ref{thm:fgen2}}  \label{proofofFG}

We carefully look at the generalized $\sR$-matrix action from
Section~\ref{NewRaction}.
We first consider the contribution $ {\mathscr{V}}_{f_1, f_2}$
of an edge $e$
connecting $v_1$ and $v_2$ with corresponding flags $f_1$, $f_2$.
We distinguish three cases:
\begin{itemize}
\item If both $v_1$ and $v_2$ are labeled by $0$, the contribution is
  of the form
  \begin{align*}
& \frac 15 \sum_{k,l\geq 0} \sum_{a+b=k+l+1, \atop b>l,\ i \in \ZZ_5} (-1)^{a+l} \bar R_{a;3-i}  \bar  R_{b;i}  \   \bar\phi_{3-i-a} \bar\psi_{f_1} ^k  \otimes   \bar\phi_{i-b}  \bar\psi_{f_2} ^l 
\\& \in \frac{I_0^2}{L^2} \sbigoplus_{k \geq 0} \Big( \bR_{k+1} \otimes (\sH^{\otimes 2}[\lambda^{-5}])_{3 - k - 1} \otimes H^{k}(\M_{g_{v_1},n_{v_1}} \times \M_{g_{v_2},n_{v_2}})\Big),
  \end{align*}
  where $\bar R_{a;i} = [z^a] \bar R_{i}^{i - a}$, and where
  $(\sH^{\otimes 2}[\lambda^{-5}])_k$ denotes the subspace of elements
  of cohomological degree $k$.
  The fact that only powers of $\lambda$ divisible by 5 appear is due
  to the fact that the classes under consideration are symmetric in
  the equivariant parameters.
\item If $v_1$ is labeled by $0$ and $v_2$ is labeled by $\infty$,
  the contribution is of the form
 \begin{align*}
\qquad
 &\frac 15 L^{- \delta_{f_2} } \sum_{k\geq 0} \sum_{a\leq k,i \in \ZZ_5} (-1)^{a} \bar R_{a;3-i}   \bar  S_{\delta_{f_2};i} \  \bar\phi_{3-i-a} \bar\psi_{f_1}^{k} \otimes   H^{i} \! \cdot \!H_{\delta_{f_2}}^{a-k-1}
 \\  &\in \frac{I_0}{L} L^{- \delta_{f_2} } \sbigoplus_{j} \Big(  \big( L^{\overline{\delta_{f_2}-1}}\overline\bR_{j   + 1}  \big)^{\text{reg}}\otimes \mathbb Q[Z]_{\leq \lfloor(\delta_{f_2}-1)/5\rfloor}\otimes \phi_{3 - j - 1} \otimes H^j_{\CC^*}(\PP^4 \times \M_{g_{v_1},n_{v_1}})\Big)
 \end{align*}
\item If both $v_1$ and $v_2$ are labeled by $\infty$, the contribution is
  of the form
  \begin{align*}
    &\frac 15 L^{- \delta_{f_1} - \delta_{f_2} } \sum_{i \in \ZZ_5} (H_{\delta_{f_1}} \otimes 1+1\otimes H_{\delta_{f_2}})^{-1}(\bar  S_{\delta_{f_1};3-i}  \bar  S_{\delta_{f_2};i} )H^{3-i} \otimes H^{i} 
    \\ \in&   L^{- \delta_{f_1} - \delta_{f_2} } \big(L^{\overline{\delta_{f_1}-1}+\overline{\delta_{f_2}-1}} \overline\bR_{3} \big)^{\text{reg} }\otimes\mathbb Q[Z]_{\leq \lfloor(\delta_{f_1}-1)/5\rfloor+\lfloor(\delta_{f_2}-1)/5\rfloor} \otimes H^{3 - 1}_{\CC^*}(\PP^4 \times \PP^4)
  \end{align*}
  Note that for $i=4$, we rewrite $L \cdot L = L^{-3} Z \in \bR_3$.
\end{itemize}
Note that it suffices to only consider edge contributions not
involving $\bar\psi$-classes because each power of $\bar\psi$
increases both cohomological degree and degree in $\bR$ by one, which
is compatible with Theorem~\ref{thm:fgen2}.
We may therefore split the contribution of each edge $e = (v_1, v_2)$
into two half-edge contributions according to a choice
\begin{equation*}
  \deg \gamma_{(e,v_1)} + \deg \gamma_{(e,v_2)} = 2,\qquad
  k_{(e,v_1)} + k_{(e,v_2)}=3.
\end{equation*}
The contribution of a half-edge $h = (e, v)$ with $v$ of label 0 lies in
\begin{equation*}
  \frac{I_0}L \bR_{k_h - \deg \gamma_h} \otimes (\sH[\lambda^{-5}])_{\deg \gamma_h},
\end{equation*}
while the contribution of a half-edge $h = (e, v)$ with $v$ of label
$\infty$ lies in
\begin{equation*}
  L^{-\delta_h} H^{\deg \gamma_h} (L^{\overline{\delta_h - 1}} \overline\bR_{k_h})^{\text{reg}} \otimes \QQ[Z]_{\leq \lfloor(\delta_h-1)/5\rfloor}.
\end{equation*}

We next consider the contribution of a leg $l$ with insertion $\phi_i$ at
a vertex $v$:
\begin{itemize}
\item If $v$ has label $0$, the contribution is
  \begin{equation*}
    \sum_{k \ge 0} \bar R_{k;i} \phi_{i - k} \bar\psi_l^k
    \in \sum_{k \ge 0} \bR_k \phi_{i - k} \bar\psi_l^k.
  \end{equation*}
\item If $v$ has label $\infty$ and $l$ has degree $\delta$, the
  contribution is
  \begin{equation*}
    \frac{L}{I_0} L^{-\delta} (\bar S_\delta)_i^i H^i
    \in \frac{L}{I_0} L^{-\delta} H^i \big(L^{\overline{\delta-1}} \overline\bR_i \big)^{\text{reg} }\otimes \QQ[Z]_{\leq \lfloor(\delta-1)/5\rfloor}.
  \end{equation*}
\end{itemize}
We now consider the vertex contributions:
\begin{itemize}
\item 
  The contribution at each vertex of label $0$ with $m$ half-edges is
  \begin{align*}
    &T \omega_{g(v), m}^\lambda\! \Big(\!\!\otimes_{i=1}^m   \phi_{a_i} \! \Big) \\
    &\in \left(\frac{I_0}L\right)^{2g(v) - 2} \sbigoplus_{k \in \ZZ} \lambda^{5k}  \bR_{3g(v)-3+\sum_i a_i-5k} \otimes H^{3g(v)-3+\sum_i a_i-5k}(\M_{g(v),m}).
  \end{align*}
 This is because the left hand side is a sum of terms of the
 form
 \begin{align*}
   \qquad\qquad ( \pi_k)_*\omega_{g(v),m+k}^\lambda &\ \Big(\otimes_{i=1}^{m}  \phi_{a_i}   , \  \otimes_{j=1}^{k}  (\bar R_{{l_j}})_{0}^{b_j}  \phi_{b_j} \bar\psi_{m+j}^{l_j+1}\Big) \\
   &\qquad \in  \left(\frac{I_0}L\right)^{2g(v) - 2} \lambda^{3g(v)-3+\sum_i a_i -\sum_j l_j} \bR_{\sum_j l_{j}} \otimes H^{\sum_j  l_j}(\M_{g(v),m})
 \end{align*}
 where $3g(v)-3+\sum_i a_i-\sum_j l_j$ is divisible by $5$ because
 both $b_j + l_j$ and $\sum_i a_i + \sum_j b_j$ must be divisible by
 $5$ in order for the term to be nonzero.
 After multiplying by $1 = L^{-5k} Z^k \in \bR_{5k}$, we may regard
 the degree in $\bR$ to be $3g - 3 + \sum_i a_i$.
\item The contribution of a vertex $v$ of label $\infty$ with $m$
  half-edges is
  \begin{multline*}
    J \Omega_{g(v), \mu}^{\infty, \mathbf c}(H^{a_1}, \dotsc, H^{a_m}) \\
    \in \left(\frac{I_0}L\right)^{2g(v) - 2 + m} L^{2g(v) - 2 + m} \sbigoplus_{k \in \ZZ} \lambda^{5k} \QQ[q]_{\le d_v} \otimes H^{3g(v) - 3 + \sum_i a_i - 5k}(\M_{g(v), m})
  \end{multline*}
  where $\nu_i = \mu_i - 1$ and $d_v := \lfloor \frac{a_v}{5} \rfloor$
  with $a_v:=2g(v) - 2 - |\nu|$.
\end{itemize}

We now consider the total contribution of a graph $\Gamma$ to the
computation of
\begin{equation*}
  \bar \Omega^{\mathbf c, \lambda}_{g,n}(\phi_{a_1}, \dotsc, \phi_{a_n}).
\end{equation*}
We first consider the power of $I_0/L$.
We have a factor $(I_0/L)^{2g(v) - 2 + |E_v|}$ where $E_v$ is the set
of edges at $v$ for every vertex $v$.
If $v$ has label $0$, this factor comes from both the vertex and the
edges.
If $v$ has label $\infty$, this factor comes from both the vertex and
the legs.
Combining these factors for all vertices $v$ gives the desired global
factor $(I_0/L)^{2g(v) - 2}$.

The overall cohomological degree is
\begin{equation*}
  |E(\Gamma)| + \sum_v (3g(v) - 3) + \sum_i a_i + 2|E(\Gamma)|
  = 3 g - 3 + \sum_i a_i,
\end{equation*}
as expected.

We next observe that the total contribution is regular in $L$.
This is because negative powers of $L$ only appear in the half-edge
factors, and the $q$-polynomial in the contribution of a vertex $v$
with label $\infty$.
These negative powers are absorbed by the power of $L$ in the vertex
term:
\begin{equation*}
  L^{-|\mu|} L^{2g(v) - 2 + m} \QQ[q]_{\le d_v}
  = L^{a_v} \QQ[q]_{\le d_v}
  = L^{\overline{a_v}} \QQ[Z]_{\le d_v}
\end{equation*}

Combining at every vertex $v$ with label $\infty$, the factors
involving $L$, $q$ and $Z$ gives an element of
\begin{multline*}
  L^{-|\mu| + \sum_i \overline{\nu_i} + 2g(v) - 2 + \ell(\mu)} \otimes \QQ[Z]_{\le \lfloor\nu_i/5\rfloor} \otimes \QQ[q]_{d_v} \\
  = L^{-(3g(v) - 3)} Z^{-\sum_i \lfloor\nu_i/5\rfloor} Z^{g(v) - 1} \otimes \QQ[Z]_{\le \lfloor\nu_i/5\rfloor} \otimes \QQ[q]_{d_v} \\
  \subset L^{-(3g(v) - 3)} \QQ[Z, Z^{-1}]_{\le g(v) - 1}
  \subset \overline\bR_{3g(v) - 3}
\end{multline*}

We therefore obtain an element of $\bR$ of degree
\begin{equation*}
  \sum_v (3g(v) - 3) + \sum_i a_i + 3|E(\Gamma)|
  = 3 g - 3 + \sum_i a_i.
\end{equation*}
This completes the proof of Theorem~\ref{thm:fgen2}.

\subsection{Yamaguchi--Yau's prediction}

In this subsection, we prove the second (numerical) part of
Theorem~\ref{thm:fgen}, which implies Yamaguchi--Yau's prediction.

For Calabi--Yau $3$-folds, all nonzero-degree Gromov--Witten
invariants vanish unless all insertions are divisor classes.
Let
$$
\bar F_{g,n}^{\mathbf c}
:= \int_{\M_{g,n}}\bar \Omega_{g,n}^{\mathbf c}(\phi_1^{\otimes n})
= 5^{g-1}\frac{L^{2g-2}}{I_0^{2g-2}} F_{g,n}^{\mathbf c}(\phi_1^{\otimes n}).
$$
Notice that by the divisor equation, we have
\begin{equation} \label{divisoreq}
\bar F_{g,n+1} =  (\partial_u - n(\Y-\X) +(2g-2)\X )\bar F_{g,n}.
\end{equation}
\begin{corollary}
We have
\begin{equation}
\bar F_{g,n}      \in \bR_{3g-3+n}
\end{equation}
\end{corollary}

\begin{example} We have the following initial data  and low genus formulae
\begin{align*}
\bar F_{0,3} = &\,\, 1 \qquad
\bar F_{0,4} = \,\, \X-3\Y  \qquad
\bar F_{1,1} = \,\, -\frac{59}{6}\X -\frac{1}{2}\Y -\frac{125}{12} \Z \\
\bar F_{1,2} = &\,\, -\frac{25}{3} \X_2+\frac{28}{3} \X(\Y-\X)+\Y^2 +\frac{125}{2}(\Y-\X) \Z-\frac{205}{24} \Z_2\\
\bar F_{2,0} = &\,\,  5\cdot \Big(
{\frac {70\,\X_{{3}}}{9}}+{\frac {575\,\X\X_{{2}}}{18}}+\frac{5 \Y\X_{{2}}}{6}+{\frac {557\,{\X}^{3}}{72}}-{\frac {629\,\Y{\X
}^{2}}{72}}-{\frac {23\,{\Y}^{2}\X}{24}}-\frac{{\Y}^{3}}{24}\\&+{
\frac {625\,\Z\X_{{2}}}{36}}-{
\frac {175\,\Z\Y\X}{9}}+{\frac {1441\,\Z_{{2
}}\X}{48}}-{\frac {25\,\Z({\X}^{2}+{\Y}^{2})}{24}}\\
& -{\frac {3125\,{\Z}^{2}(\X+\Y)}{288}} +{\frac {41\,\Z_{{2}}\Y}{48}}-{\frac {625\,{\Z}^{3}}{144}}+{
\frac {2233\,\Z\Z_{{2}}}{128}}+{\frac {547\,\Z_{{3}}}{72}}\Big)
 \end{align*}
where we have used the genus one \cite{Zi08} and genus two mirror theorems \cite{GJR17P}.
\end{example}

\subsection{Grading and orbifold regularity}
 
In this section, we will explain how the ``orbifold regularity"  follows from the ``graded finite generation property". 

The original ``orbifold regularity" is a global property for the
Gromov--Witten potential.
In the paper \cite{BCOV93}, the authors constructed a non-holomorphic
completion $\mathcal F_g(q,\bar q)$ of the Gromov--Witten potential
$F_g(q)$ for each $g$, such that
$$
\lim_{\bar q \rightarrow 0}  \mathcal F_g(q,\bar q) =   F_g(Q(q)) .
$$
Further, they conjecture that there exist an  ``LG" theory with the potential function $F_g^{\text{orb}}(\tau^{\text{orb}})$ such that
$$
\lim_{\bar q \rightarrow \infty}  \mathcal F_g(q,\bar q) =   F_g^{\text{orb}}(\tau^{\text{orb}}) |_{\tau^{\text{orb}}=\tau^{\text{orb}}(q^{-1/5}) }.
$$
Nowadays it is known that such an LG theory is the FJRW theory.
Then the ``orbifold regularity" in  physics can be translated as follows:
\begin{theorem}
  Let
  $$
  \frac{L^{2g-2}}{I_0^{2g-2}}  F_g = P_g(\X_1,\X_2,\X_3,\Y,L^{-1},Z)
  $$
  be the polynomial of the six generators.
  Then the limit
  $$
  \lim_{L\rightarrow 0} P_g(\X_1,\X_2,\X_3,\Y,L^{-1},L^5)
  $$
  exists.
\end{theorem}
\begin{proof}
  The theorem follows from the degree bounds in the definition of
  $\bR$ (see Equation~\eqref{fivegenring}).
\end{proof}

\section{Holomorphic anomaly equations (HAE)}
\label{sec:HAE}

\subsection{Derivations acting on the $\sR$-matrix} \label{sec:PDEforR}

With the finite generation property for the $R$-matrix (see
Lemma~\ref{lem:fgforR}), we can consider vector fields in $\Der_\bR$
acting on the $R$-matrix.
The key to the proof of the holomorphic anomaly equations of
Theorem~\ref{HAE2}, are the following differential equations satisfied
by the generalized $\sR$-matrix:
\begin{lemma}
  \label{lem:PDEsforR}
  Let $\bR' \subset \bR$ be a non-holomorphic subring.
  Then, there exists a unique strictly lower triangular (see below)
  linear map\footnote{We sometimes abuse notation and will identify
    $\End\sH$ with
    $\End \spann_\QQ\{\bar \phi_0,\cdots,\bar \phi_4\}$.}
  $\A\colon \Der_{\tR} \rightarrow \End \sH\otimes \bR[z]$, such that
  for any $\D\in \Der_{\tR}$, we have
  \begin{equation} \label{PDEforR}
     \D  \bar R^{-1}(z) = \bar R^{-1}(z) \cdot \A_\D(z)  \quad \text{  and  }\quad     \D  {\bar S}_\delta = {\bar S}_\delta \cdot \A_\D(-H_\delta),
  \end{equation}
  where $\bar R^{-1}(z) = \bar R^*(-z)$  (see Section~\ref{sec:QDERS}), and
  where $H_\delta = -5H/\delta$.
  Furthermore, for any $\D$, $\D_1$ and $\D_2$, we have the following
  properties:
  \begin{enumerate}
  \item $\A_\D$ is strictly lower triangular:
    $\A_\D \bar \phi_j \in \spann_\QQ\{\bar \phi_0,\cdots,\bar
    \phi_{j-1}\}[z]$;
  \item $\A_\D$ is skew adjoint: $\A_\D  + \A_\D^*|_{z\mapsto -z} =0$;
  \item $\A_{[\D_1,\D_2]} =  [\A_{\D_1}, \A_{\D_2}]+\D_1 \A_{\D_2}- \D_2 \A_{\D_1}$
  \end{enumerate}
\end{lemma}
\begin{remark}
  The proof of the above Lemma will give an algorithm for computing
  the map $\A$ via the QDE.
  Furthermore, it also works for the fully equivariant (without
  specialization) twisted theory.
  In a sequel \cite{GJR19P}, we will use it in the proof of the LG/CY
  correspondence.
\end{remark}
\begin{remark}
  The map can be extended to the ring $\Der_{\bR}$, if we enlarge the
  codomain of the map and if we only consider the $R$-matrix.
  To be precise, there exist a map
  $${\A}\colon \Der_{\bR} \rightarrow  \End \sH \otimes \bR[[z]]$$
  such that the first equality in \eqref{PDEforR} holds.   
\end{remark}
\begin{corollary} \label{cor:PDEforV}
  There exists a map
  $
  \B\colon \Der_{\tR} \rightarrow   \sH[z_1]\otimes  \sH[z_2]
  $
  such that for any $\D\in \Der_{\tR}$
  \begin{equation}
    \begin{aligned}
      \label{PDEforV}
          \D\Big[\frac{\sum_\alpha e_\alpha \otimes e^\alpha - \bar R^{-1}(z) e_\alpha \otimes \bar R^{-1}(z) e^\alpha}{z_1 + z_2}\Big]
          &=\ -  (\bar R^{-1}(z_1) \otimes \bar R^{-1}(z_2)) \B_\D(z_1, z_2), \\
          \D\Big[\frac{\sum_\alpha -\bar R^{-1}(z_1) e_\alpha \otimes \bar S_{\delta_2} e^\alpha}{z_1 - H_{\delta_2}}\Big]
          &=\ -  (\bar R^{-1}(z_1) \otimes \bar S _{\delta_2}) \B_\D(z_1, -H_{\delta_2}), \\
          \D\Big[\frac{\sum_\alpha -\bar S_{\delta_1}  e_\alpha \otimes \bar S_{\delta_2} e^\alpha}{-H_{\delta_1} - H_{\delta_2}}\Big]
          &=\ -   (\bar S_{\delta_1} \otimes \bar S_{\delta_2}) \B_\D(-H_{\delta_1}, -H_{\delta_2}).
    \end{aligned}
  \end{equation}
\end{corollary}
\begin{proof}
  We can define the map as follows:
  \begin{equation} \label{BinA}
    \B_\D(z_1,z_2):= \sum_\alpha \frac{ \A_\D(z_1) e^\alpha\otimes  e_\alpha+e^\alpha\otimes \A_\D(z_2) e_\alpha }{z_1+z_2}
  \end{equation}
  Note that the skew-adjoint property makes \eqref{BinA} a polynomial
  of $z_1$ and $z_2$.
\end{proof}
\begin{proof}[Proof of Lemma~\ref{lem:PDEsforR}]
The differential operator $\partial_u= \frac{d}{d u}$ can be considered as the following vector field
$$
D_ u =  \sum_{G\in\{\X_1,\X_2,\X_3,\Y,L\}}  (\partial_uG)\cdot \partial_G = \Y_2 \partial_{\Y}+\sum_{i=1}^3 \X_{i+1} \partial_{\X_i}+ \frac 15 (L^5 - 1) \partial_L \in \Der_{\bR},
$$
where $\Y_2, \X_4 \in \bR$ are given by
\begin{align*}
\Y_2=&\, -3 \X_2-\Y^2-\X^2-\frac{15}{4}\Z_2, \\
\X_4 =&\,  \textstyle -4  \X  \X_3-3  \X_2^2-6  \X^2  \X_2- \X^4-\frac{15}{4} (\Z_2  \X^2+\Z_2  \X_2+\Z_3  \X)\\
&\textstyle \quad -\frac{23}{24} \Z_4+\frac{29}{9} \Z_1^2 \Z_2-\frac{65}{72} \Z_1 \Z_3-\frac{3}{4} \Z_2^2.
\end{align*}
We obtain for $\D \in \Der_{\tR} $
$$
[\D, \partial_u] = [\D, \Y_2 \partial_{\Y}+\X_2\partial_{\X_1}+\X_3\partial_{\X_2}+\X_4\partial_{\X_3}] \in \Der_{\tR}.
$$

Now we inductively construct the map $\A$ and prove that it satisfies
Property~(1).
First, by the first statement in Lemma~\ref{lem:R0}, we have
\begin{equation*}
  \D \bar R^{-1}(z)\bar \phi_0 = 0,
\end{equation*}
so that we can define $\A_\D  \bar \phi_0 = 0$.

Assume that for any $k\leq i$, we have defined $\A_\D$ such that
$\D R^{-1}(z) \bar \phi_k = R^{-1}(z) \A_\D \bar \phi_k$ and that
Property~(1) holds.
By using the QDE \eqref{QDEnormalize} for $R^*(z) = R^{-1}(-z)$, we
first get for $k< i$:
\begin{align*}
  \big(-z D_{\C} +\Amone \big) \, \big[\bar R^{-1}(z) \bar\phi_k  \big] &=   \bar R^{-1}(z) \Amone  \bar\phi_k  \\
  \big(  H_\delta D_{\C} + 5H \Z \big) \, \big[\bar S_\delta \bar\phi_k  \big] &=  \bar S_\delta \Amone   \bar\phi_k 
\end{align*}
The next equation gives
\begin{align*}
 \D {\bar R}^{-1}(z) \bar \phi_{i+1}
=  & \
-z \D   \big[D_{\C} {\bar R}^{-1}(z) \bar \phi_{i}  \big] + \Amone  \D \big[  {\bar R}^{-1}(z) \bar \phi_{i} \big]  \\
=  & \ -z \big( [\D, \partial_u]  \big) {\bar R}^{-1}(z) \bar \phi_{i}+z  {\bar R}^{-1}(z)  \big(\D \C +  [\A_\D ,\C]  \big) \bar \phi_{i} \\
&\qquad + \big(  -z D_{\C}   +  \Amone   \big)  \bar R^{-1}(z)  \A_\D   \bar \phi_i \\
=  & \ 
 {\bar R}^{-1}(z)  \big[ z\big( - \A_{  [\D, \partial_u] }+{ \D \C }+ [\A_\D ,\C] \big)+ (-z\partial_u+\Amone)\A_\D \big] \bar \phi_i  \\
  \D {\bar S}_\delta\bar \phi_{i+1}
=  & \
H_\delta \D   \big[D_{\C} {\bar S}_\delta\bar \phi_{i}  \big] +  5\Z \cdot  H  \D \big[ {\bar S}_\delta\bar \phi_{i} \big]  \\
=  & \ H_\delta \big( [\D, \partial_u]  \big) {\bar S}_\delta\bar \phi_{i}-H_\delta  {\bar S}_\delta \big(\D \C +  [\A_\D ,\C]  \big) \bar \phi_{i} \\
&\qquad + \big(H_\delta D_{\C}   +   5\Z \cdot H   \big)  {\bar S}_\delta  \A_\D   \bar \phi_i \\
=  & \ 
 {\bar S}_\delta \big[   z  \big( - \A_{  [\D, \partial_u] }+{ \D \C }+ [\A_\D ,\C] \big)+ (-z \partial_u+ \Amone )\A_\D \big]_{z=-H_\delta} \bar \phi_i  
\end{align*}
where in the last equality we have again used the QDE
\eqref{QDEnormalize} and the induction assumption.
We define
  \begin{equation*}
    \A_\D(z) \bar \phi_{i+1}:= \big[ z (-\A_{  [\D, \partial_u]}(z) + \D \C + [\A_\D(z) ,\C])+  (-z\partial_u+\Amone)\A_\D(z) \big] \bar \phi_i 
  \end{equation*}
and hence finish the inductive definition.

Property (2) follows from the symplectic property of $R(z)$.
Property (3) follows from a direct computation via \eqref{PDEforR}. 
\end{proof}

\subsection{Explicit formulae of PDEs for the $\sR$-matrix} \label{sec:exPDEforR}

In this section, we set $\tR=\QQ[\X_1,\X_2,\X_3,\Y]$.
  We introduce the following derivations, which generate $\Der_{\tR}$:\footnote{One can directly check that for $\partial_0$ and $\partial_1$, the definition matches with the one given in \eqref{operatorpartial01}.}
\begin{align*}
  \partial_1 :=-\partial_\Y  + \partial_{\X} - ( \X-\Y)\,\partial_{\X_2}-\big(2\X\Y +4\X_2+\frac{15}{4} \Z_2\big)  \,\partial_{\X_3}\\
  \partial_2 :=  \partial_{\X_2} -2\X\,\partial_{\X_3},\qquad
  \partial_3 := \partial_{\X_3},\qquad
  \partial_0 := \partial_\Y
\end{align*}
By a direct calculation via the algorithm in the proof of Lemma \ref{lem:PDEsforR}, we have
\begin{lemma}
For $\D={g_0 \partial_0+g_1 \partial_1+g_2 \partial_2+g_3 \partial_3} \in \Der_{\tR}$,\footnote{In terms of Yamaguchi-Yau's generators, one has
\begin{align*}
\D = &\, g_1 \,\partial_\U+g_0\,\partial_{\V_1}+  (g_2-  \U\cdot g_0 ) \,\partial_{\V_2}
+(   g_3+2\, \U \cdot g_2- \U^2 \cdot g_0 )\, \partial_{\V_3}
\end{align*}
}
the maps $\A$ and $\B$ are given by
\begin{align*}
\A _\D = &\,\,  (g_0-g_1) \, \Lambda_0 \psi+  g_1 \,\Lambda_1 \psi  - g_2\, \Lambda_2\psi^2 + g_3 \, \Lambda_3\psi^3 \in  \End \sH[[ \psi]]   ,
\\
  \B_\D =&\,(g_0-g_1) \,  {\textstyle \frac{1}{5}}  \bar\phi_1\otimes  \bar\phi_1 + g_1\, \B_2 + g_2 \,\B_1 +g_3 \,\B_0 \in \sH[[ \psi_1]]\otimes \sH[[\psi_2]]  ,
\end{align*}
where we define the following lower triangular maps (where we, by
abuse of notation, set $\bar \phi_{k}:=0$ for $k<0$)
$$
\forall j,\quad \Lambda_0(\bar \phi_j) :=\delta_{j,2}\bar \phi_1,\quad \Lambda_1(\bar \phi_j):=\bar \phi_{j-1} , \quad \Lambda_2(\bar \phi_j):=(-1)^j \bar \phi_{j-2} ,\quad \Lambda_3(\bar \phi_j):=\bar \phi_{j-3}  ;
$$
and we define certain ``diagonal" classes in $ \sH[z_1]\otimes \sH[z_2]$:
\begin{align*}
\B_{0} :=&\,  \  \   {\textstyle  \frac{1}{5}  }\big( z_1^2  \bar\phi_0 \otimes   \bar\phi_0 - z_1 \bar\phi_0 \otimes  z_2 \bar\phi_0+  \bar\phi_0 \otimes z_2^2  \bar\phi_0 \big) \\
\B_1 :=&\,  - {\textstyle  \frac{1}{5}  }(z_1-z_2) (1\otimes  \bar\phi_1- \bar\phi_1\otimes 1) ,\\
\B_{2} :=&\,  \  \    {\textstyle  \frac{1}{5}  } \big(  \bar\phi_0\otimes  \bar\phi_2+ \bar\phi_1 \otimes  \bar\phi_1+ \bar\phi_2\otimes   \bar\phi_0 \big)
\end{align*}
\end{lemma}

\begin{example} \label{ex:PDEforR} We have the following explicit forms of \eqref{PDEforR}
\begin{align*}
\partial_\Y R^{-1}(z) \bar \phi_j =\,& \delta_{j,2}\cdot zR ^{-1}(z) \bar \phi_1\\
\partial_{\X} R^{-1}(z) \bar \phi_j =\,&  z R ^{-1}(z)  \bar \phi_{j-1}  +\delta_{j,2}\cdot(\Y-\X) z^2R ^{-1}(z)  \bar \phi_0 \qquad\qquad \text{  ($\bar \phi_{-1}:=0$) }
\\
&\quad +\delta_{j,3}\cdot \Big(- (\Y-\X) z^2R ^{-1}(z) \bar \phi_1 +  \big(  2\X^2+4\X_2 +\frac{15}{4}\Z_2\big) z^3 R^{-1}(z)\bar \phi_0\Big)
\\
\partial_{\X_2} R^{-1}(z) \bar \phi_j =\,& -\delta_{j,2}\cdot z^2R ^{-1}(z) \bar \phi_0 +\delta_{j,3}\cdot ( z^2R ^{-1}(z) \bar \phi_1+2 \X\, z^3R ^{-1}(z) \bar \phi_0)\\
\partial_{\X_3} R^{-1}(z) \bar \phi_j =\,&  \delta_{j,3}\cdot z^3R ^{-1}(z) \bar \phi_0
\end{align*}
and we have the following explicit forms of \eqref{PDEforV}
\begin{align*}
-\partial_{\Y} V(z_1,z_2)  =\,& R ^{-1}(z_1) \otimes R ^{-1}(z_2) \Big({\textstyle  \frac{1}{5}  } \phi_1\otimes \phi_1 \Big)\\
-\partial_{\X} V(z_1,z_2)  =\,&   R ^{-1}(z_1) \otimes R ^{-1}(z_2) \Big(  \B_{2}  + (\Y-\X)\B_{1} + \big(  2\X^2+4\X_2 +\frac{15}{4}\Z_2\big)   \B_{0}\Big)
\\
-\partial_{\X_2} V(z_1,z_2)  =\,& R ^{-1}(z_1) \otimes R ^{-1}(z_2) \Big( \B_{1}+2\X \B_{0}  \Big)
\\
-\partial_{\X_3} V(z_1,z_2)  =\,& R ^{-1}(z_1) \otimes R ^{-1}(z_2) \Big(  \B_{0}  \Big).
\end{align*}
where $V(z_1,z_2):=\frac{\sum_\alpha e_\alpha \otimes e^\alpha - \bar R^{-1}(z) e_\alpha \otimes \bar R^{-1}(z) e^\alpha}{z_1 + z_2}$.
If we replace $R$ by $S$ and set $z=5H/\delta$, the same equalities
hold for the specialized $S$-matrix.
\end{example}

\begin{corollary}
  For the operators $\partial_0, \partial_1 \in \Der_{\tR}$ defined in
  Equation~\eqref{operatorpartial01}, under Yamaguchi-Yau's generators, we simply
  have
  \begin{align}
    \B_{\partial_0}  = & \  {\textstyle  \frac{1}{5}  }   \bar\phi_1 \otimes   \bar\phi_1 , \label{eq:YY0} \\
    \B_{\partial_1}  = & \  {\textstyle  \frac{1}{5}  }  \bar\phi_0\otimes  \bar\phi_2+ {\textstyle  \frac{1}{5}  }  \bar\phi_2\otimes   \bar\phi_0. \label{eq:YY1}
  \end{align}
  In particular, $\B_{\partial_0}$ and $\B_{\partial_1}$ do not depend
  on $z_1$ and $z_2$.
\end{corollary}

\begin{remark}[Yamaguchi--Yau's generators] Indeed,  
  by using the generators defined in \eqref{generatorsSp} the QDE for $\sR$ can
  be written in a much simpler form:
  \begin{align*}
    \bar R _{k+1;1}= \,& (\partial_u - \U) \bar R_{k;0} +\bar R _{k+1;0} \\
    \bar R _{k+1;2} = \,&  (\partial_u +\U- \V) \bar R _{k;1} + (\partial_u - \U) \bar R_{k;0} +\bar R _{k+1;0}\\
    = \,&  (\partial_u^2 -\V \cdot \partial_u - \V_2) \bar R_{k-1;0}   + (2 \partial_u - \V) \bar R_{k;0} +\bar R _{k+1;0}\\
    \bar R _{k+1;3}= \,& \bar R _{k+1;2} + (\partial_u+\U)  \bar R _{k;2}  \\
    \bar R _{k+1;4} = \,&\bar R _{k+1;0} - \partial_u  \bar R _{k;4}
  \end{align*}
  Then by using the differential relations
  \begin{equation}
    \label{eq:DV}
    \begin{aligned}
      &  \partial_u \V = -2 \V_2-\V^2-\frac{15}{4} \Z_2, \qquad  \partial_u \V_2 = \V_3 - \V \V_2 ,\\
      &  \partial_u \V_3 = \V_2^2-\frac{23}{24}\Z_4+\frac{29}{9}\Z^2\Z_2-\frac{65}{72}\Z\Z_3-\frac{3}{4}\Z_2^2 .
    \end{aligned}
  \end{equation}
  the PDE for $R$ with respect to the differential operator
  $\partial_1=\partial_\U$ can be written in a much simpler way.
\end{remark}

\subsection{Proof of the holomorphic anomaly equations}

In the following section, we prove the holomorphic anomaly equations
of Theorem~\ref{HAE2}.
They will be a direct consequence of:
\begin{theorem}
  \label{thm:HAE}
  For any $\D \in \{\partial_0, \partial_1\}$, we have
  \begin{equation*}
    -\D \bar \Omega^{\mathbf c}_g  =
    \frac{1}{2} r_* \bar \Omega^{\mathbf c}_{g-1,2}\big(5\B_\D \big) + \frac{1}{2} \sum_{g_1+g_2=g} s_* (\bar  \Omega^{\mathbf c}_{g_1,1}\otimes \bar  \Omega_{g_2, 1}^{\mathbf c}) \big(5\B_\D\big),
  \end{equation*}
  where we recall that
  $\bar \Omega^{\mathbf c}_{g,n} := {5^{g-1}(L/I_0)^{2g-2}}
  \Omega^{\mathbf c}_{g,n}$, and where
  $r\colon \M_{g - 1, 2} \to \M_g$ and
  $s\colon \M_{g_1, 1} \times \M_{g_2, 1} \to \M_g$ are the gluing
  maps.
\end{theorem}
\begin{proof}
  This is similar to \cite{LhPa18} and \cite{LhPa18P}.

  We consider the action of $\D$ on the contribution of a decorated
  graph $\Gamma$ in the generalized $\sR$-matrix action for
  $\bar \Omega^{\mathbf c}_g$.
  Since we have no markings, only the edge contribution $\mathscr V$
  depends on non-holomorphic generators (see Section \ref{proofofFG}).
  Therefore the contribution of $\Gamma$ to
  $\D \bar \Omega^{\mathbf c}_g$ naturally splits into a sum over
  edges $e$ of $\Gamma$.
  We can consider the contribution to $2\D \bar \Omega^{\mathbf c}_g$
  instead as split into a sum over half-edges $h$ of $\Gamma$.
  When the edge $e$ is non-disconnecting, the half-edge $h$ determines
  a decorated genus $g - 1$ graph with two (ordered) legs.
  This corresponds to the first term in the HAE.
  When $e$ is disconnecting, the half-edge $h$ determines a decorated
  genus $g_1$ graph with a leg corresponding to $h$ and a decorated
  genus $g_2$ graph with one leg, where $g_1 + g_2 = g$.
  This corresponds to the second term in the HAE.

  The holomorphic anomaly equation therefore follows from
  Corollary~\ref{cor:PDEforV}.
\end{proof}

\subsection{Examples of HAEs}

\begin{theorem}
  Recall
  $$
  \bar F^{\bf c}_{g,n}  := \int_{\M_{g,n}} \bar \Omega^{\mathbf c}_{g,n}(\phi_{1}, \dotsc, \phi_{n}) .
  $$
  The following HAEs for the extended quintic family hold {\footnotesize
    \begin{align*}
      -\partial_0 \bar F^{\bf c}_{g}  =\, \,\,\frac{1}{2} \bar F^{\bf c}_{g-1,2} +\frac{1}{2} \sum_{g_1+g_2=g}  \bar F^{\bf c}_{g_1,1} \cdot \bar F^{\bf c}_{g_2,1}, \quad \qquad
      -\partial_{1 } \bar F^{\bf c}_{g}   =\,   0 .
    \end{align*}}
\end{theorem}
\begin{proof}
  The first HAE directly follows from Theorem~\ref{thm:HAE} and
  \eqref{eq:YY0}.
  For the second HAE, in addition to Theorem~\ref{thm:HAE} and
  \eqref{eq:YY1}, we need to use that by pullback and dimension
  considerations both $\Omega^{\mathbf c}_{g,1}(\phi_2)$ and
  $\Omega^{\mathbf c}_{g,2}(\phi_0, \phi_2)$ vanish.
\end{proof}
\begin{remark}
  \label{rmk:YYHAE}
  Under the Yamaguchi--Yau generators $\{\U,\V_1.\V_2,\V_3,L\}$ (see
  \eqref{generatorsSp}), the divisor equation \eqref{divisoreq}
  becomes
  \begin{equation}
    \bar F_{g,n+1} =  (\partial_u - n(\V-2\,\U) +(2g-2)\,\U )\bar F_{g,n}.
  \end{equation}
Then we have
\begin{align*}
 & - \big( \partial_{\V_1}-\U \, \partial_{\V_2} -\U^2 \,  \partial_{\V_3}  \big) \bar F_{g}  \\
  &\qquad = \, \,\,\frac{1}{2}  \Big(\partial_u^2   +\big( 2(2g-1)\,\U-  \V\big) \partial_u +(2g-2) \big(\V_2 + (2g-1)\U^2\big) \Big) \bar F_{g-1} \\
  &\qquad \qquad +\frac{1}{2} \sum_{g_1+g_2=g}   (\partial_u   +(2g_1-2)\,\U ) \bar F_{g_1} \cdot   (\partial_u  +(2g_2-2)\,\U )\bar F _{g_2}
\end{align*}
Since $\partial_\U \bar F_{g} =0$ and $\partial_u \V_i \in \mathbb Q[\V_1,\V_2,\V_3,L]$, we have $\partial_\U (\partial_u\bar F_{g})=0$. Hence, we obtain
\begin{align*}
   \partial_{\V_1} \bar F_{g}  =& \, - \frac{1}{2}  \Big(\partial_u^2    -  \V  \partial_u +(2g-2) \V_2  \Big) \bar F_{g-1} 
- \frac{1}{2}\sum_{g_1+g_2=g}   \partial_u   \bar F_{g_1} \cdot  \partial_u  \bar F _{g_2}\\
\partial_{\V_2} \bar F_{g}  =& \, \,\, (2g-1)\,  \partial_u   \bar F_{g-1} 
   +\sum_{g_1+g_2=g}   \partial_u   \bar F_{g_1} \cdot   (2g_2-2)\, \bar F _{g_2}\\
\partial_{\V_3} \bar F_{g}  =& \, \,\, (g-1)   (2g-1)  \bar F_{g-1} 
   +\sum_{g_1+g_2=g}  2(g_1-1)(g_2-1)\, \bar F_{g_1} \cdot    \bar F _{g_2}
\end{align*}
We can see that up to a polynomial of $L$ (which we consider as a global meromorphic function in the moduli space),  $\bar F_{g}$ is determined from the initial data by the holomorphic anomaly equations.
\end{remark}

\begin{example}
We have the following HAEs for the genus two cases
 {\footnotesize
 }
\begin{align*}
 \quad  
  \partial_\Y \bar F_{2,0} =& \,-\frac{1}{2} \bar F_{1,2}  - \frac{1}{2}\bar F_{1,1} ^2\\
  \qquad  \quad =&  \, -\frac{1}{2} \Big( {\small \text{$
{\frac {5^6\Z^{2}}{144}}\!+ {\frac {1750\X\Z}{9}}\!+{
\frac {125 \Y\Z}{6}}\! +{\frac {3145\,{\X}^{2}}{36}}\!+{\frac {115\X\Y}{6}}+\frac{5}{4}{\Y}^{2}\!-{\frac {205\,{\Z_2}}{24}}\!-{\frac {25\,{\X_2}
}{3}}$}}\Big)
 \end{align*}
 and
 $$
   \Big( -\partial_\Y  + \partial_{\X} - ( \X-\Y)\,\partial_{\X_2}-\big(2\X\Y +4\X_2+\frac{15}{4} \Z_2\big)  \,\partial_{\X_3} \Big) \bar F_{2,0} =  \  0
 $$
This matches with the  following mirror formula proved in \cite{GJR17P} {\small
\begin{align*}
 F^{GW}_2(\tau(q)) = \, & \frac{I^2_0 }{L^2} \cdot \Big(
{\frac {70\,\X_{{3}}}{9}}+{\frac {575\,\X\X_{{2}}}{18}}+\frac{5 \Y\X_{{2}}}{6}+{\frac {557\,{\X}^{3}}{72}}-{\frac {629\,\Y{\X
}^{2}}{72}}-{\frac {23\,{\Y}^{2}\X}{24}}-\frac{{\Y}^{3}}{24}\\&+{
\frac {625\,\Z\X_{{2}}}{36}}-{
\frac {175\,\Z\Y\X}{9}}+{\frac {1441\,\Z_{{2
}}\X}{48}}-{\frac {25\,\Z({\X}^{2}+{\Y}^{2})}{24}}-{\frac {3125\,{\Z}^{2}(\X+\Y)}{288}}\\
&  +{\frac {41\,\Z_{{2}}\Y}{48}}-{\frac {625\,{\Z}^{3}}{144}}+{
\frac {2233\,\Z\Z_{{2}}}{128}}+{\frac {547\,\Z_{{3}}}{72}} \Big).
 \end{align*} 
 }
 \end{example}

\section{A technical result of the formal quintic theory}
\label{sec:orbifoldregularity}

In this Section we prove a somewhat technical result about the formal
quintic theory whose results we used to prove Lemma~\ref{lem:fgforR}.
As a by-product, we prove a conjecture raised in \cite{ZaZi08}.

For future applications, we consider more generally the
$(\CC^*)^m$-equivariant $\mathcal O(m)$-twisted GW theory of
$\PP^{m-1}$, and we use $\lambda_0, \dotsc, \lambda_{m - 1}$ to denote
the equivariant parameters.
Following Givental (see \cite{Gi96,CoGi07}), for any
$\alpha = 0,\cdots, m - 1$, let $ \I_\alpha$ be the following
oscillatory integrals in the Landau--Ginzburg (LG) model
\begin{equation*}
  \I_\alpha(q,z):=\int_{\gamma_\alpha \subset (\CC^*)^m}  e^{W(x_0, \dotsc, x_{m - 1})/z}  \frac{d x_0 \wedge \cdots \wedge d x_{m - 1}}{x_0 \cdots x_{m - 1}},
\end{equation*}
where $\gamma_\alpha$ are the $m$ Lefschetz thimbles and
$W$ is the LG potential
$$
W(x_0,\cdots,x_{m-1})=\sum_{i=0}^{m-1}  (x_i {+} \lambda_i \ln x_i)  { -} \big( {q}^{-1} {\prod_{i=0}^{m-1} x_i} \big)^{\frac{1}{m}}.
$$

\begin{theorem}
  \label{thm:R0}
  When $\lambda_i =\zeta_m^i \cdot \lambda $, where $\zeta_m$ is a
  primitive $m$th root of unity, we have the following asymptotic
  expansion
  \begin{equation}  \label{asymptI}
    \I_\alpha(q,z) \asymp e^{u_\alpha/z} (-2\pi  z)^{\frac{m}{2}}  \Big( 1+  \sum_{k>0} r_{k  \alpha} \cdot (-z)^k\Big)
  \end{equation}
  as $z\rightarrow 0^{-}$, such that $L_\alpha^k \cdot r_{k \alpha}$
  where
  $L_\alpha = \zeta^\alpha \lambda L = \zeta^\alpha \lambda (1 - m^m
  q)^{-1/m}$ satisfies
  \begin{enumerate}
  \item $L_\alpha^k \cdot r_{k  \alpha}  \in L^k \QQ[L] \cap \QQ[L^m]$ (Regularity)
  \item $\deg_{L^m}(L_\alpha^k \cdot r_{k  \alpha}) = k$ (Degree bound)
  \item $(mL_\alpha)^k \cdot r_{k  \alpha}  \in \QQ[m, L^m]$ (Polynomiality in $m$)
  \end{enumerate}
\end{theorem}
This theorem will follow from Proposition~\ref{prop:osc},
Lemma~\ref{lem:Lm} and Lemma~\ref{lem:extra}.

\begin{remark}
  When $\lambda_i =\zeta_m^i \cdot \lambda $, the oscillatory integral
  $ \I_\alpha(q,z) $ is related to the $\tilde I$-function defined in
  \cite{ZaZi08} (where it is denoted by
  $z \mathcal F_{-1}(h_\alpha/z, q)$, and defined in Equation (17)) by
  $$
  (-2\pi  z)^{-\frac{m}{2}}  e^{-h_\alpha \log q/z}  C_\alpha(z) \cdot \I_\alpha(q,z)  = \tilde I_\alpha (q,z):=  z\sum_{d\geq 0} q^d \frac{\prod^{md-1}_{k=0}(m h_\alpha +kz)}{\prod^d_{k=1}\big( ( h_\alpha +kz)^m -\lambda^m \big)}
  $$
  where $h_\alpha = \zeta_m^\alpha \cdot \lambda$ is the restriction
  of the hyperplane class
  $H \in H_{(\mathbb C^*)^m}^2(\mathbb P^{m-1})$ at the $\alpha$th
  fixed point, and where $C_\alpha(z)$ is defined as in
  \eqref{constantR}.
\end{remark}

Then as a direct consequence of the above theorem, we prove a
conjecture originally proposed in \cite[Equation (39)]{ZaZi08}.
We now present the conjecture in a slightly stronger form:
\begin{corollary}
  When $\lambda_i =\zeta_m^i \cdot \lambda $, we have the following
  asymptotic expansion
  $$
  z \, q\frac{d}{dq} \I_\alpha(q,z) \asymp e^{u_\alpha/z} (-2\pi  z)^{\frac{m}{2}} \cdot L_\alpha \Big( 1+  \sum_{k>0} (R_k)_{0  \bar \alpha} \cdot (-z)^l\Big)
  $$
  as $z\rightarrow 0^{-}$, such that
  \begin{equation*}
    (m L_\alpha)^k (R_k)_{0  \bar \alpha} \in L^k \QQ[L] \cap \QQ[L^m]
  \end{equation*}
  is a polynomial of $L^m$ of degree $k$ with coefficients in
  $\QQ[m]$.
\end{corollary}
\begin{proof}
  This follows from Theorem~\ref{thm:R0} since, as we will see below,
  $q \frac d{dq} u_\alpha = L_\alpha$ and
  $q \frac d{dq} L_\alpha = -\frac{L^m}m L_\alpha$.
\end{proof}

In particular, when $m=5$ the oscillatory integral is related to our
quintic $I$-function by (recall the definition of $\tilde I(z)$ in
\eqref{tIfunction})
$$
C_\alpha^{-1}(z) e^{ h_\alpha \log q / z} \tilde I_\alpha(q,z) =   \I_\alpha(q,z),
$$
so that
$$
C_\alpha^{-1}(z) e^{ h_\alpha \log q / z} I_\alpha(q,z) = q \frac{d}{dq} \I_\alpha(q,z),
$$
and we have that
  $$
[z^l] R_{4\bar \alpha}(z)= r_{\alpha,0;l}(L_\alpha), \quad
[z^l] R_{0\bar \alpha}(z)= r_{\alpha,1;l}(L_\alpha), 
  $$
are regular at $L_\alpha=0$.
In the case $m = 5$, the corollary specializes to:
\begin{corollary}
  \label{lem:R0}
  For the entries in the first column of the $R$-matrix, we have
\begin{align*}
(1)&\qquad [z^k] {\bar R_0}^j(z) \in \QQ[L] \cap   L^{-k}\QQ[L^5]
\\
(2)&\qquad 
[z^k] {\bar R_0}^j(z) = 0 \quad \text{ if  } \quad  k+j \notin 5 \ZZ .
\end{align*}
\end{corollary}

\subsection{Computing the asymptotic expansion by using stationary phase method}

In this subsection, we study the oscillatory integral using the
stationary phase method in order to prove the regularity in
Theorem~\ref{thm:R0}.

We first note that the $m$ critical points of $W$ can be computed
directly
$$
(x_i)_\alpha = L_\alpha -\lambda_i,\qquad y_\alpha = m L_\alpha
$$
with critical value $u_\alpha  = \int L_\alpha \frac{dq}{q}$, where for $\alpha=0,\cdots,m-1$
$$
L_\alpha = \lambda_\alpha+ O(q)
$$
are the $m$ solutions of the indicial equation
$\prod_i(L-\lambda_i)=(mL)^{m} q$.

\begin{proposition}
  \label{prop:osc}
  The oscillatory integral has an asymptotic expansion of the form
  \begin{align*}
    \I_\alpha(q,z)  \asymp e^{u_\alpha /z} (-2\pi  z)^{\frac{m}{2}}  \Psi_{-1 \bar\alpha}  \cdot \Big( 1+  \sum_{l>0} r_l \cdot (-z)^l\Big)
  \end{align*}
  as $z\rightarrow 0$ from the negative real axis, where $Q$ is the
  Hessian bilinear form of $W$, 
  \begin{equation} \label{Psim1}
    \Psi_{-1 \bar\alpha}  := \left( m^{-1}{\textstyle \sum_{j=0}^{m-1} (-1)^{m - j} (m - j)\, e_{m - j} L_\alpha^j} \right)^{-1/{2}}
  \end{equation}
  with $e_j$ being the $j$th elementary symmetric polynomial in the
  $\lambda_i$, and where $r_l$ are rational functions of $L_\alpha$.
  Furthermore, $r_l(L_\alpha)$ is regular at $L_\alpha = 0$.
\end{proposition}
\begin{proof}
  We can Taylor-expand
  \begin{equation*}
    W\big((x_0)_\alpha + \xi_0, \dotsc, (x_{m - 1})_\alpha + \xi_{m - 1} \big)
    = u_\alpha + \frac 12 Q(\vec\xi) + W'(\xi_0, \dotsc, \xi_{m - 1}),
  \end{equation*}
  where $W'$ consists of terms of order $\ge 3$ in the $\xi_i$.
  With this, we can write
  \begin{align*}
    \I_\alpha(q,z)
    \asymp & \ \frac{e^{u_\alpha/z}}{\prod_i (x_i)_\alpha} \int_{\gamma_\alpha' \subset (\CC^*)^m} e^{\frac{1}{2z}Q(\vec\xi)} \frac{e^{W'(\xi_0, \dotsc, \xi_{m - 1})/z}}{\prod_i (1 + (x_i)^{-1}_\alpha \xi_i)} d\xi_0 \wedge \cdots \wedge d\xi_{m - 1}\\
   = & \ \frac{e^{u_\alpha/z}}{\prod_i (x_i)_\alpha} (-z )^{\frac{m}{2}}  \int_{\gamma_\alpha'}   e^{-\frac{1}{2}Q(\vec\xi)}  d\xi_0\cdots d\xi_{m - 1}\cdot  \Big( 1+  \sum_{l>0} r_l \cdot (-z)^l\Big)
,
  \end{align*}
  where $\gamma_\alpha'$ is $\gamma_\alpha$ shifted by
  $((x_0)_\alpha, \dotsc, (x_{m - 1})_\alpha)$.
   By Wick's (or Isserlis') theorem, for each $l>0$, $r_k$ is polynomial in
  $( Q^{-1})_{ij}$, $(x_i)^{-1}_\alpha$ and the higher order derivatives
  $W_I:={\textstyle \frac{\partial^{|I|} W}{\partial x_{i_1} \cdots
      \partial x_{i_{|I|}} }
    \big|_{x_i=(x_i)_\alpha}}$
  It therefore remains to show that
  \begin{equation}
    \label{eq:Psi-1}
    \Psi_{-1, \bar\alpha}  = (2\pi )^{-\frac{m}{2}}  {  \prod_i (x_i)^{-1}_\alpha \int_{\gamma_\alpha'}   e^{-\frac{1}{2}Q(\vec\xi)}  d\xi_0\cdots d\xi_{m - 1} },
  \end{equation}
  and that $\det(Q)^{-1}$, $(x_i)^{-1}_\alpha$ and the higher order
  derivatives of $W$ at $(x_i)_\alpha$ are rational functions in
  $L_\alpha$ well-defined at $L_\alpha = 0$.

  First, note that
  \begin{equation*}
    (x_i)^{-1}_\alpha = \frac 1{L_\alpha - \lambda_i}
  \end{equation*}
  is a rational function of $L_\alpha$ regular at $L_\alpha$.
  Therefore, also the derivatives
  \begin{align*}
    \frac{d^k W}{d x_i^k} \Big((x_0)_\alpha, \dotsc, (x_{m - 1})_\alpha\Big)
    =& (x_i)^{-k}_\alpha\left(\lambda_i (-1)^{k -1} (k - 1)! - m L_\alpha \prod_{i = 0}^{k - 1} \left(\frac 1m - i\right)\right), \\
    \frac{d^k W}{d x_0^{k_0} \dotsb d x_{m - 1}^{k_{m - 1}}} \Big((x_0)_\alpha, \dotsc, (x_{m - 1})_\alpha\Big)
    =& -m L_\alpha \prod_i (x_i)^{-k_i}_\alpha \prod_{j = 0}^{m - 1} \prod_{i = 0}^{k_j - 1} \left(\frac 1m - i\right)
  \end{align*}
  for $k = \sum_i k_i \ge 2$ are rational functions of $L_\alpha$
  regular at $L_\alpha = 0$.
  
  Note that
  \begin{equation*}
    (2\pi )^{-\frac{m}{2}}  {  \prod_i (x_i)^{-1}_\alpha \int_{\gamma_\alpha'}   e^{-\frac{1}{2}Q(\vec\xi)}  d\xi_0\cdots d\xi_{m - 1} } 
    =\prod_i (x_i)^{-1}_\alpha \cdot  ( \det Q )^{-1/2},
  \end{equation*}
  and that
  \begin{equation*}
    \det Q = \frac{\sum_{j=0}^{m-1} (-1)^{m - j} (m - j)\, e_{m - j} L_\alpha^j}{m \prod_i (L_\alpha-\lambda_i)^2},
  \end{equation*}
  so that $\det(Q^{-1})$ is rational in $L_\alpha$, well-defined at
  $L_\alpha = 0$, and that \eqref{eq:Psi-1} holds.
\end{proof}

\begin{corollary}
  For any $k>0$, the $k$th derivative of $\I_\alpha$ has an asymptotic
  expansion
  \begin{equation*}
    \Big( z q\frac{d}{dq} \Big)^k \I_\alpha(q,z)  \asymp e^{u_\alpha /z}  (-2\pi  z)^{\frac{m}{2}}  \Psi_{k-1, \bar\alpha}  \cdot \Big( 1+  \sum_{l>0} r_{\alpha,k;l} \cdot (-z)^l\Big),
  \end{equation*}
  were
  \begin{equation*}
    \Psi_{k-1, \bar\alpha} = L_\alpha^k  \Psi_{-1, \bar\alpha}
  \end{equation*}
  such that for any $k,l >0$
  $$
  r_{\alpha,k;l} (L_\alpha)  \text{ is a rational function of $L_\alpha$  which  is regular at $L_\alpha=0$}.
  $$
\end{corollary}
\begin{proof}
  This follows from the following two observations:
  First, we have
  \begin{equation*}
    z q\frac{d}{dq} e^{u_\alpha/z}\Psi_{k-1, \bar\alpha}
    = e^{u_\alpha/z}\Psi_{k, \bar\alpha} \big(1+ z L_\alpha^{-1} q\frac{d}{dq} \log \Psi_{k-1, \bar\alpha} + z L_\alpha^{-1} q\frac{d}{dq}\big)
  \end{equation*}
  Second, the operator $ L_\alpha^{-1} q\frac{d}{dq} $ preserves
  regularity, in the sense that for any rational function $f(x)$ that
  is regular at $x=0$, the function
  \begin{equation*}
  L_\alpha^{-1}  q\frac{d}{dq}  f(L_\alpha) =     \frac{ -\prod_i (L_\alpha-\lambda_i)}{\sum_{j=0}^{m-1} (-1)^{m - j} (m - j)\, e_{m - j} L_\alpha^j}  f'(L_\alpha),
  \end{equation*}
  is regular at $L_\alpha=0$.
\end{proof}

By the above results, we can set $L_\alpha = 0$ in $\I_\alpha$ and its
derivatives.
We have the following explicit formulae which will be crucial in the
proof of the LG/CY correspondence\cite{GJR19P}:
\begin{lemma}
  The rational functions $r_{\alpha,k;l} $ take the following values
  at $L_\alpha=0$,
  \begin{equation*}
    \Big( 1+  \sum_{l>0} r_{\alpha,k;l} \cdot z^l\Big)\Big|_{L_\alpha=0} = \exp \bigg[ {\sum_{i=0}^{m-1} \sum_{j = 1}^{\infty}\frac{(-1)^{j + 1} B_{j+1}(\frac{k}{m})}{j(j+1)}\frac{z^{j}}{\lambda_i^{j}}}\bigg],
  \end{equation*}
  where $B_j(x)$ is the $j$th Bernoulli polynomial.
  In particular, if we set $\lambda_i= \zeta^i \lambda$, we have
  \begin{equation} \label{initialdata}
    \Big( 1+  \sum_{l>0} r_{\alpha,k;l} \cdot z^l\Big)\Big|_{L_\alpha=0} = \exp \bigg[ {m \sum_{j = 1}^{\infty}\frac{(-1)^{mj + 1} B_{mj+1}(\frac{k}{m})}{mj(mj+1)} \frac{z^{mj}}{\lambda^{mj}}}\bigg].
  \end{equation}
\end{lemma}
\begin{proof}
  For any $\alpha$, if $L_\alpha=0$, we have $q^{-1}=0$, and
  furthermore that $q^{-\frac 1m} L_\alpha^{-1}$ becomes
  \begin{equation*}
    m \prod_{i = 0}^{m - 1} (-\lambda_i)^{-\frac 1m}.
  \end{equation*}
  The $k$th derivative $L_\alpha^{-k} (z q \frac d{dq})^k \I_\alpha$
  therefore is
  \begin{equation*}
    \int_{\gamma_\alpha \subset (\CC^*)^m}  \prod_{i = 0}^{m - 1} (-\lambda_i)^{-\frac km} (x_i)^{\frac km} e^{\sum_{i=0}^{m-1}  z^{-1} (x_i {+} \lambda_i \ln x_i)}  \frac{d x_0 \wedge \cdots \wedge d x_{m - 1}}{x_0 \cdots x_{m - 1}}
  \end{equation*}
  at $L_\alpha = 0$, which agrees with the asymptotic expansion of the
  product of Gamma functions
  \begin{equation*}
    \prod_{i = 0}^{m - 1} (-\lambda_i)^{-\frac km} (-z)^{\frac{\lambda_i}z + \frac km} \Gamma\left(\frac{\lambda_i}z + \frac km\right).
  \end{equation*}
  By a direct computation using the asymptotic expansion \cite{He95}
  (see also \cite[(1.8)]{Ne13})
  \begin{equation*}
    \ln \Gamma(z + a) \asymp \left(z + a - \frac 12\right) \ln(z) - z + \frac 12 \ln(2\pi) + \sum_{j = 1}^\infty \frac{(-1)^{j + 1} B_{j + 1}(a)}{j(j + 1) z^j},
  \end{equation*}
  we conclude that $L_\alpha^{-k} (z q \frac d{dq})^k \I_\alpha$ has
  the asymptotic expansion
  \begin{equation*}
    \prod_i e^{\frac{-\lambda_i + \lambda_i \ln(-\lambda_i)}z} \lambda_i^{\frac km} \sqrt{-2\pi z/(-\lambda_i)} \exp\left[\sum_{j = 1}^\infty \frac{(-1)^{j + 1} B_{j + 1}(\frac km)}{j(j + 1)} \frac{z^j}{\lambda_i^j}\right].
  \end{equation*}
  Unwrapping the definition of $r_{\alpha, k; l}$, we conclude the
  lemma.
\end{proof}

\begin{lemma}
  \label{lem:Lm}
If we take $\lambda_i = \zeta_m^i \lambda$, then
$$
L_\alpha^k \cdot r_k(L_\alpha) \in \QQ[L_\alpha^m/\lambda^m] 
$$
\end{lemma}
\begin{proof}
  Notice that the coefficient of $L_\alpha$ in $r_l$ are symmetric
  polynomials in $\lambda_i$, and that
  \begin{equation*}
    \det Q^{-1}
    = \frac{(L_\alpha^m - \lambda^m)^2}{\lambda^m}
    \in \lambda^{-m} \QQ[L_\alpha,\lambda]_{  2m}
  \end{equation*}
  Then by using $  Q_{ij} = (L_\alpha  - \lambda_i)^{-1}(L_\alpha  - \lambda_j)^{-1}  \QQ[L_\alpha,\lambda]_{ 1}$, we obtain
  $$
 ( Q^{-1})_{ij}   \in \lambda^{-m} \QQ[L_\alpha, \lambda]_{m+1} .
  $$
  Here we consider both $L_\alpha$ and $\lambda$ as homogeneous degree
  $1$ elements. 
  Furthermore,
  $$
  (x_i)^{-1}_\alpha \cdot  (x_j)^{-1}_\alpha \cdot ( Q^{-1})_{ij}   \in \lambda^{-m} \QQ[L_\alpha, \lambda]_{m-1} .
  $$
\end{proof}

\subsection{Strengthening using the Picard--Fuchs equations}

In this section we will give the degree bound of the $r_k$ for each
$k$, and hence give the degree bound for the $R$-matrix.
\begin{lemma}
  \label{lem:extra}
  Recall that we have proved
  $$
  L_\alpha^k \cdot r_{k\alpha} \in \mathbb Q[L^m]
  $$
  In addition, we have the degree bound
  \begin{equation}
    \label{degreeboundrk}
    \deg_{L^m} ( L_\alpha^k \cdot r_{k\alpha}) \  =   \  k,
  \end{equation}
  and the polynomiality in $m$:
  \begin{equation*}
    (mL_\alpha)^k \cdot r_{k\alpha} \in \mathbb Q[m, L^m]
  \end{equation*}
\end{lemma}
\begin{proof}
  We first illustrate the proof of the degree bound for the formal
  quintic case: $m=5$.
  Recall that the $\I$-function satisfies the following Picard--Fuchs
  equation
$$
\Big[ D ^5 -\lambda^5/z^5 - q\cdot 5D (5 D +1) \cdots (5 D +4)  \Big]  \I_\alpha(q,z) =0 
$$
where $D:=q d/dq$. By using the asymptotic expansion \eqref{asymptI}, we see
$$
\Big[ D_{L_\alpha} ^5 -\lambda^5 - q\cdot 5D_{L_\alpha} (5 D_{L_\alpha} + z) \cdots (5 D_{L_\alpha} + 4z)  \Big]  (1+\sum r_{k\alpha} (-z)^k) =0 
$$
where  (recall $L_\alpha = \zeta^\alpha \lambda \cdot L$)
$$
D_{L_\alpha}:=zD+  {L_\alpha} .
$$
This equation is equivalent to the following recursive relations for
$r_{k\alpha}$:
$$
- D r_{k \alpha} =  \DD_1 r_{k-1, \alpha}+\DD_2 r_{k-2, \alpha}+\DD_3 r_{k-3, \alpha}+\DD_4 r_{k-4, \alpha}
$$
where $X= 1-L^5$ and
\begin{align*}
\DD_1:= & \  \frac{1}{25  L_\alpha} (3X^2-3X+10XD^1+50D^2)\\
\DD_2:= & \  \frac{1}{125 L_\alpha^2} (-24X^3+39X^2-15X+(15X^2+5X)D+150XD^2+250D^3)\\
\DD_3:= & \  \frac{1}{3125 L_\alpha^3} (396X^4-870X^3+575X^2-101X+(-450X^3 +725X^2 -125X) D\\
& \qquad \qquad+1375XD^2+3750XD^3+3125D^4)\\
\DD_4:= & \  \frac{1}{3125 L_\alpha^4} (X\cdot (24D+250D^2+875D^3+1250D^4)+625D^5)
\end{align*}
Notice that $r_{0\alpha}=1$, and that the differential operator
$\DD_i$ satisfies
$$
\DD_i\colon   L_\alpha^{-(k-i)}   \mathbb Q[X]_{\leq (k-i)} \longrightarrow L_\alpha^{-k} X  \mathbb Q[X]_{\leq k},
$$
where we have used $D L_\alpha = -\frac{1}{5}X L_\alpha$ and $D X = X(1-X)$. 
Then together with the initial data \eqref{initialdata} of
$L_\alpha^k \cdot r_{k\alpha}$ at ${L_\alpha=0}$ ($X=1$), 
we see that Equation \eqref{degreeboundrk} follows by induction. 

In the general case, the $\I$-function satisfies the following
Picard--Fuchs equation
$$
\Big[ D ^m -\lambda^m/z^m - q\cdot mD (m D +1) \cdots (m D +m-1)  \Big]  \I_\alpha(q,z) =0 
$$
Letting $X:=1-L^m$, the equation for $r_{k\alpha}$ follows
\begin{multline*}
\Big[  L_\alpha^{-m}(1-X) \cdot D_{L_\alpha} ^m - 1 +  {(m L_\alpha)^{-m}} X \cdot mD_{L_\alpha} (m D_{L_\alpha} +z) \\
 \cdots (m D_{L_\alpha} +(m-1)z)  \Big]  (1+\sum r_{k\alpha} (-z)^k) =0
\end{multline*}
Note that $D L_\alpha = -\frac{1}{m}X L_\alpha$ and that
$D X = X(1-X)$.
We have for $s>0$,
$$
L_\alpha^{-s} D_{L_\alpha}^s =1+\sum_{j=1}^s \ \frac{m z^{j}}{ m^{j} L_\alpha^{j} }\cdot   \sum_{i=0}^j     c_{j,i}(m,X) D^{j-i},
$$
in which
\begin{equation*}
  c_{j,0} =\binom{s}{j}, \qquad  c_{j,i} \in X \mathbb Q[m][X]_{i-1} \text{   for $i>0$} .
\end{equation*}
Then
the above equation gives the following relation for $r_{k\alpha}$
$$
- D r_{k \alpha} =  \DD_1 r_{k-1, \alpha}+\DD_2 r_{k-2, \alpha}+\cdots+\DD_{m-1} r_{k-m+1, \alpha}
$$
such that $\DD_i \in (mL_\alpha)^{-i} \mathbb Q[m, X, mD]$ satisfies
$$
\DD_i\colon   L_\alpha^{-(k-i)}   \mathbb Q[X]_{\leq (k-i)} \longrightarrow L_\alpha^{-k} X  \mathbb Q[X]_{\leq k}.
$$
For example
\begin{align*}
\DD_1 = \ &  {\frac { \left( m+1 \right)   \left( m-1 \right) \left( m-2 \right) }{24{m}^{2}\,L_\alpha}} X(X-1)+{\frac {m-1}{2\,mL_\alpha}} X\,D+{\frac {m-1}{2\,L_\alpha}}{D}^{2}
\\
\DD_2 = \ & {\frac { \left( m+1 \right)   \left( m-1 \right)   \left( m-2 \right)  \left( m-3 \right) }{24{m}^{3}\,{L_\alpha}^{2}}
 \Big(  m  \big( X-{\frac{1}{2}} \big)-X  \Big) }  X(1-X)  \\
 & \  +{\frac {
  \left( m-1 \right)( m-2 )   \big(   
{m}^{2} (X-1) -5 m (X-1)+6\,X+2 \big) }{24\,{m}^{2}{L_\alpha}^{2}}
} X\, D \\
& \ +{\frac { \left( m-1 \right) \left( m-2 \right) }{2\,m{L_\alpha}^{2}}} X\, {D
}^{2}+{\frac {  \left( m-1 \right) \left( m-2 \right)  }{6\,{L_\alpha}^{2}}}{
D}^{3}.
\end{align*}
The rest is similar to the $m=5$ case.
Together with the regularity \eqref{initialdata} of
$L_\alpha^k \cdot r_{k\alpha}$ at ${L_\alpha=0}$, we see that
Equation \eqref{degreeboundrk} follows by induction.

The polynomiality in $m$ also follows from the polynomiality of
\eqref{initialdata} and $\DD_i$ by induction.
\end{proof}

\bibliographystyle{abbrv}
\bibliography{biblio}

\end{document}